\documentclass{amsart}
\usepackage{amsmath,amsthm,amsfonts,graphicx,amssymb}
\usepackage{amscd}
\usepackage{mathrsfs}
\usepackage[utf8]{inputenc}
\usepackage[left=40mm,right=40mm]{geometry}
\usepackage{tikz-cd}
\usepackage{enumitem}
\usepackage{quiver}
\usepackage{thm-restate}
\long\def\comment#1\endcomment{}

\usepackage{hyperref}
 \usepackage{dirtytalk}

\usepackage{mathtools}
\usepackage{caption}
\usepackage{subcaption}
\usepackage{tikz}
\usetikzlibrary{positioning,calc}

\newcommand{\acts}{\curvearrowright}

\newcounter{dcom}

\newcounter{scom}

\newcounter{ccom}

\def\dist{{\rm{d}}}

\usepackage{color}

\title[Weak rank rigidity and navigable path systems]{Weak rank rigidity for groups with a navigable path system}
\author{Cornelia Dru\c{t}u}
\email{cornelia.drutu@maths.ox.ac.uk}
\address{Mathematical Institute \\
 University of Oxford \\ Oxford, UK.}
\author{Davide Spriano}

\email{davide.spriano@warwick.ac.uk}
\address{Mathematics Institute \\ 
 University of Warwick \\ Coventry, UK.}
 \author{Stefanie Zbinden}
\email{stefanie.zbinden@gmail.com}
\address{Mathematical Institute \\
 University of Bonn\\ Bonn, Germany.}

\definecolor{darkgreen}{cmyk}{1,0,1,.2}
\allowdisplaybreaks[1] 

\long\def\comment#1\endcomment{}

\newtheorem{theorem}{Theorem}[section]
\newtheorem{proposition}[theorem]{Proposition}
\newtheorem{corollary}[theorem]{Corollary}
\newtheorem{lemma}[theorem]{Lemma}

\newtheorem{notation}[theorem]{Notation}
\newtheorem{cvn}[theorem]{Convention}

\newtheorem{definition}[theorem]{Definition}

\numberwithin{equation}{section}

\newtheorem{claim}{Claim}
\theoremstyle{definition}
\newtheorem{example}[theorem]{Example}
\newtheorem{remark}[theorem]{Remark}

\newtheoremstyle{citing}
  {3pt}
  {3pt}
  {\itshape}
  {}
  {\bfseries}
  {}
  {.5em}
  {\thmnote{#3}}

\theoremstyle{citing}

\DeclareMathOperator{\diam}{diam}
\DeclarePairedDelimiter\abs{\lvert}{\rvert}%
\DeclarePairedDelimiter\norm{\lVert}{\rVert}
\DeclarePairedDelimiter\ceil{\lceil}{\rceil}
\DeclarePairedDelimiter\floor{\lfloor}{\rfloor}
\newcommand{\eps}{\epsilon}

\newcommand{\mb}{\partial_*}
\newcommand{\Con}{{\mathrm{Con}}}

\newcommand {\la}{\lambda}
\newcommand{\lao}{\lambda_0}
\newcommand {\ka}{\kappa}
\newcommand{\kao}{\kappa_0}

\newcommand{\nn}{{\mathcal N}}

\newcommand{\col}{\qgot}

\newcommand{\spath}{{\mf{h}}}

\newcommand{\dv}{{\mathrm{div}}}
\newcommand{\Dv}{{\mathrm{Div}}}

\newcommand{\length}[1]{{\norm{{#1}}}}
\newcommand{\dlength}[1]{{\abs{{#1}}}}
\newcommand{\isub}[2]{{{#1}{[#2]}}}
\newcommand{\psub}[2]{{[#2]}_{#1}} 
\newcommand{\pref}[2]{{#1}_{\mathrm{pref}}[#2]}
\newcommand{\suf}[2]{{#1}_{\mathrm{suf}}[#2]}

\newcommand{\R}{\mathbb{R}}

\newcommand{\N}{\mathbb{N}}
\newcommand{\Z}{\mathbb{Z}}

\newcommand{\isom}{\mathrm{Isom}}

\newcommand {\qgot}{{\mathfrak q}}
\newcommand {\pgot}{{\mathfrak p}}

\newcommand {\mf}[1]{\mathfrak{#1}}
\newcommand{\mc}{\mathcal}

\newcommand{\deletethis}[1]{}

\numberwithin{equation}{section}

\begin{document}

\begin{abstract}
We show that groups with a mild form of non-positive curvature (a navigable path system) satisfy the weak rank rigidity conjecture: they either have linear divergence or a Morse element. This class includes discrete groups of projective automorphisms of open convex cones, Helly groups (answering a question of Genevois), Coxeter groups, weak Garside groups (in particular Deligne's groups and fundamental groups of Salvetti complexes of oriented matroids), hierarchically hyperbolic groups, and other examples. Along the way, we show that those groups satisfy the Morse local-to-global property, providing a unified proof for the whole class.

In the metric setting, the same condition of non-positive curvature allows to provide a local definition (that is, in a sense, optimal) of rank one/Morse geodesics, mirroring the one using parallel Jacobi fields from Riemannian geometry; to deduce linearity of divergence from linearity on a sequence; to obtain new cases in which Morse geodesics are strongly contracting. 

The main new tool introduced is the \emph{generalised contraction space}, a hyperbolic space that encodes the negative curvature of a given space. 
\end{abstract}

\maketitle

\tableofcontents
\section{Introduction}

Curvature, as a fundamental spatial feature, lies at the core of geometry. Non-positive curvature, in particular, occupies a central place in modern geometry.

Various notions of curvature have been defined and intensively used, at first mainly in differential geometry and in the study and classification of manifolds, using analytic tools. Answering the need for such a notion in the setting of metric spaces, curvature later began to be used in synthetic geometry, via the notion of CAT(0) space, requiring a distance function more convex than in the flat plane.

With the work of Efremovic \cite{Efremovic:growth}, Albert Schwarz \cite{Svarc} and Milnor \cite{Milnor(1968b)}, advancing the idea that  the fundamental group of a compact Riemannian manifold encodes much of the geometry of its universal cover, the need for a coarse notion of curvature, which would make sense for discrete groups as well, began to make itself felt. Thus,  more recently, through various coarse versions, curvature made its way to algebra, combinatorics, topology and non-commutative geometry.  

The seminal work of Gromov \cite{Gromov:hyperbolic} showed that negative curvature can be translated in terms of groups in a clear-cut way, yielding the notion of hyperbolic groups and, more generally, Gromov hyperbolic spaces. In \cite[p. 80-81]{Gromov:hyperbolic}, M. Gromov pointed out that, in the same way in which the example of cocompact groups of isometries of negatively curved spaces led to hyperbolic groups, there should be a notion of semihyperbolic groups, modeled on cocompact groups of isometries of non-positively curved spaces. However, the right approach towards formulating such a notion has been less clear.

Various possible definitions have been suggested, with usefulness depending on the viewpoint and the problem addressed. Suggestions for finitely generated groups include the existence of bounded quasi-geodesic bicombings on Cayley graphs, equivariant or not (\cite{Epstein-Cannon-Holt-Levy-Paterson-Thurston(1992)}, \cite{AlonsoBridson} and \cite{Bridson-Haefliger})    
and for finitely presented groups a quadratic Dehn function \cite{Gromov:Asymptotic}. 

In \cite[$\S 6.D_3$]{Gromov:Asymptotic} Gromov also argued in favour of a definition of `non-positively curved space' that would cover Banach spaces and Hilbert geometries. In this paper, we focus on a version of coarse non-positive curvature that covers many Hilbert geometries and appears to be particularly well adapted for the topic of rank rigidity.

Generally speaking, non-positive curvature displays its most striking and informative manifestations when confronted with problems of rigidity. `Rigidity' in a wide sense covers rigidity of representations under local deformations or bounded perturbations, and rigidity of the structures themselves. The latter has several interpretations. A structure is rigid if it is preserved under loose equivalence relations (e.g. quasi-isometry, orbit equivalence), its group of automorphisms is small etc.

The brand of rigidity that we consider here comes from the Rank Rigidity Theorem of W. Ballmann \cite{Ballman:nonpositive} (see also \cite{KKL:QI}) stating that 
if $X$ is a non-flat, de Rham irreducible, locally compact, complete CAT(0) manifold on which a discrete group $G$ acts cocompactly by isometries, then $X$ is either a symmetric space of non-compact type and rank at least two or $X$ contains a periodic Morse geodesic. In the Riemannian manifold setting, a periodic geodesic is Morse if and only if its space of parallel Jacobi fields is of dimension one, while in the more general metric space setting such a geodesic is Morse if and only if it does not bound a half-plane \cite[Proposition~4.5]{KKL:QI}. 

In the second case, of the existence of a periodic Morse geodesic in $X$, the group $G$ contains Morse elements and free non-abelian subgroups. 

Thus, we are aiming at a class of metric spaces with some weak form of non-positive curvature that would allow for a meaningful investigation of the presence of strictly negative curvature, via Morse quasi-geodesics, which are, in a sense, directions with negative curvature in a space; if possible, also for the formulation of a dichotomy similar to the Rank Rigidity Theorem. 

Note that some extra assumption (e.g. of weak non-positive curvature) is needed, as otherwise the conclusions of the Rank Rigidity Theorem do not hold in general:
 
\begin{itemize}  
\item a group may contain a Morse quasi-geodesic (with respect to a word metric) but not a Morse element \cite{fink:morse};

\item a group may contain a Morse element, but not a free subgroup \cite{OOS}. 
\end{itemize}

The relatively mild condition of `coarse non-positive curvature' that we propose in this paper is that of the {\emph{existence of an undirected navigable path system}}. This condition generalizes the case of groups endowed with a bounded equivariant bicombing, as well as another notion of coarse non-positive curvature, coming from a different area where such a theory has been successful. Indeed, in recent years, in topology and algebra, CAT(0) cubical structures and their coarse versions often came into play, and allowed to solve major problems. These structures have been generalized to median structures, and the latter turn out to be particular cases of spaces with undirected navigable path systems. We refer to Subsection~\ref{subsec:spaces_with_navigable_PS} for further details. Last but not least, the condition is satisfied by Banach spaces, convex cones with the Hilbert metric and many more (see paragraph following Theorem \ref{cori:median-is-md} and Section \ref{sec:metric-spaces}).

The existence of an undirected navigable path system allows us to formulate a dychotomy (Theorem \ref{thmi:dichotomy}). Considering the generality of the setting, the case with no rank one (quasi-)geodesic cannot hope to yield such a very rigid structure as that of a symmetric space or Euclidean building. In this case, the `presence of flat geometry along every geodesic' is expressed via `linear divergence', which is a more flexible condition. Nevertheless, there is some rigidity behind this condition that we briefly explain below.

The notion of divergence that we use here is the one introduced in \cite{DrutuMozesSapir}. It essentially measures the length of minimal paths joining two points  while avoiding a ball around a  third point. Linear divergence for a group is equivalent to the property that no asymptotic cone has global cut-points. The importance of having cut-points in asymptotic cones has been shown in \cite{DrutuSapir:TreeGraded} and \cite{DrutuSapir:splitting}.

Besides lattices of isometries of higher rank symmetric spaces, other groups with linear divergence are groups satisfying a law and groups with an infinite order element in the center.  

Among the instances of rigidity that groups with linear divergence display, we can mention, for instance, that such a group, when embedded in a relatively hyperbolic group, must be parabolic, while when embedded in a mapping class group it must be virtually abelian. This latter theorem can be used to give a proof of the Birman-Lubotzky-McCarthy theorem \cite{BirmanLubotzkyMcCarthy}. It also implies the theorem of Farb-Kaimanovich-Masur \cite{KaimanovichMasur,FarbMasur} stating that a homomorphism from a higher rank lattice into the mapping class group of a surface must have finite image.

\subsection{Main Theorem and Applications} 

In \cite[Question 6.10]{BehrstockDrutu:Divergence} it was asked whether a cobounded CAT(0) space satisfies some form of weak rank rigidity. The question was answered positively in \cite{KentRicks:Asymptotic} and then in \cite{PetytSprianoZalloum:hyperbolic} with different techniques. As far as we are aware, CAT(0) spaces are the only class of cobounded spaces currently known to (non-trivially) satisfy weak rank rigidity. For instance, Genevois asked whether Helly groups satisfy weak rank rigidity, as recorded by Haettel in \cite[Section~13]{Haettel:survey_injective}.

Our main result is the following. 

\begin{theorem}[Weak rank rigidity, {Theorem~\ref{thm:main-dichotomy}}]\label{thmi:dichotomy}
    Let $G$ be a non-virtually cyclic group acting properly and coboundedly on a geodesic metric space $X$ endowed with an undirected, $G$--invariant, and navigable path system. Then the following dichotomy holds. 
    \begin{enumerate}
        \item\label{dich1} Either $G$ has linear divergence, equivalently no asymptotic cone of $G$ has a global cut-point, or 
        \item\label{dich2} $G$ contains an infinite order Morse element, in fact it contains a stable virtually free subgroup, every asymptotic cone of $G$ has cut-points, and $G$ is acylindrically hyperbolic.
    \end{enumerate}
\end{theorem}

We recall that a subgroup is \emph{stable} if it is quasi-convex with respect to quasi-geodesics and every quasi-geodesic between two points of it is Morse. A cyclic subgroup is stable if and only if it is generated by a Morse element. Thus, containing a stable virtually free subgroup should be thought as containing a tree such that every geodesic of it is Morse.

A \emph{path system} $\mc P$ is a family of chosen quasi-geodesics connecting every pair of points. For instance, in a geodesic metric space the set of all geodesics is a path system. 
The key requirement that we impose on our path systems is \emph{navigability} (see Definition~\ref{defn:navigable}). \emph{Navigability} is a property allowing to tighten $\mc P$-polygonal lines that avoid a given ball until their size is comparable to that of the ball. It is satisfied for instance by path systems with some form of coarse continuity, such as path systems coming from bounded geodesic combings. We defer a discussion on how to endow a space with a navigable path system, and for a long list of examples to Subsection~\ref{subsec:spaces_with_navigable_PS}. 

The strength of Theorem \ref{thmi:dichotomy} is illustrated by the breath of examples to which it applies, thus providing a uniform approach to rigidity for many different classes of groups. The list below is not exhaustive, and we refer to Section \ref{sec:groups} for more applications as well as detailed referencing. 

\begin{corollary}\label{cori:groups_satisfying_WRR}
The Weak Rank Rigidity, as in Theorem~\ref{thmi:dichotomy}, is satisfied by the following classes of groups. 
\begin{enumerate}
   \item discrete groups $G$ of projective automorphisms of open convex cones $C$ such that $C/G$ is compact; the geometry of such groups and such cones (called {\emph{divisible convex cones}}) has been studied by Y. Benoist \cite{BenoistTata, BenoistDuke, BenoistDim3};
   
   \item\ naively convex cocompact groups \cite{DancigerGueritaudKassel}; these are groups of automorphisms as above with cocompact action on a convex subset of the cone $C$;
   
   \item\label{cor:helly} Helly groups, which include cubulable groups, FC type Artin groups, Garside groups such as spherical Artin groups;
    \item Coxeter groups; 
     \item CAT(0) groups;
    \item hierarchically hyperbolic groups; in particular mapping class groups, Artin groups of large, extra-large and hyperbolic type, and extensions of lattice Veech subgroups;
    \item finitely presented C(6), C(3)-T(6), C(4)-T(4) small cancellation groups;
    \item\label{cor:median} groups acting properly coboundedly on geodesic median spaces;
    \item weakly systolic groups; in particular systolic groups.
\end{enumerate}
\end{corollary}
Note that Corollary  \ref{cori:groups_satisfying_WRR}, \eqref{cor:helly}, answers Genevois' question. 

Corollary \ref{cori:groups_satisfying_WRR}, \eqref{cor:median}, also yields some interesting information on the connection between amenability and median geometry. Indeed, it is known that every a-T-menable group acts properly on a geodesic median space \cite{CDH}. A natural question to ask is whether (elementary) amenable groups have such actions with additional properties. The following example shows that coboundedness cannot be one such additional property. 

\begin{corollary}
The non-virtually cyclic elementary amenable groups that are lacunary hyperbolic constructed in \cite{OOS} cannot act properly coboundedly on geodesic median spaces.
\end{corollary}

Theorem~\ref{thmi:dichotomy} is well behaved with respect to commensurability, with the exception of the conclusion on acylindrical hyperbolicity. 

\begin{corollary}\label{cori:virtual_wrr}
Let $H$ be a group satisfying the hypotheses of Theorem~\ref{thmi:dichotomy}, and let $G$ be commensurable with $H$. Then the following dichotomy holds. 
    \begin{enumerate}
        \item Either $G$ has linear divergence, equivalently no asymptotic cone of $G$ has a global cut-point, or 
        \item $G$ contains an infinite order Morse element, every asymptotic cone of $G$ has cut-points.
    \end{enumerate}
\end{corollary}

A notable application of Corollary~\ref{cori:virtual_wrr} is that is applies to the outer automorphism group of many one-ended hyperbolic groups, which are known to be virtually hierarchically hyperbolic, but not hierarchically hyperbolic on the nose \cite{HadziosmanovicMangioni:outer}.

\subsection{The family of contraction spaces}

A key step in the proof of Theorem~\ref{thmi:dichotomy} is to associate to every space $X$ endowed with a path system a family of hyperbolic spaces $\hat{X}_{\mc K}$, parametrized by contraction triples $\mc K$ (see Section~\ref{sec:generalised_contraction} for precise definitions) called {\emph{generalized contraction spaces}}, that detect, among special paths, those that are $\mc P$--contracting. This is a custom version of the contraction spaces introduced in \cite{Zbinden:hyperbolic}.

\begin{theorem}[{Theorem~\ref{thm:contraction_space_hyp_for_graphs} and Theorem~\ref{prop:morse-quasi-geo-equivalence}}] \label{thmi:contraction_sp_hyperbolic}
    There exists $\delta_0$ such that for any graph $X$ equipped with a path system $\mc P$ and for large enough $\mc K$, the corresponding $\mc K$--contraction space $\hat{X}_{\mc K}$ is $\delta_0$-hyperbolic, and the canonical vertex inclusion map $\iota_{\mc K} : X\to \hat{X}_{\mc K}$ is $1$--Lipschitz. Moreover, a special path $\spath \in \mc P$ is $\mc P$--contracting if and only if $\iota_{\mc K}\circ \spath$ is a quasi-geodesic in $\hat{X}_{\mc K}$.

    Every graph isomorphism of $X$ preserving $\mc P$ induces a graph isomorphism (hence an isometry) of $\hat{X}_{\mc K}$.
\end{theorem}

By far the hardest part of the proof of Theorem~\ref{thmi:contraction_sp_hyperbolic}, and one of the most involved arguments of the paper, is the one showing that the resulting space is hyperbolic. This is achieved via the introduction of a function called the \emph{triangular pull}. \emph{In the setting of quasi-geodesics, a minimizer of the triangular pull has properties similar to those of a closest point projection in the setting of geodesics (see Section \ref{sec:hyperbolicity})}.

For spaces with navigable and undirected path systems, the contraction spaces are some sort of Elysium fields in which only Morse paths survive as quasi-geodesics. 

\begin{theorem}[{Theorem~\ref{prop:morse-quasi-geo-equivalence}}]\label{propi:Morseiffqiem}
    Let $X$ be a graph equipped with a navigable and undirected path system $\mc P$, and let $\spath \in \mc P$ be a special path. Then $\spath$ is Morse if and only if $\iota_{\mc K} \circ \spath$ is a quasi-geodesic in $\hat{X}_{\mc K}$.
\end{theorem}

Induced group actions on contraction spaces bring substantial information.

\begin{theorem}[{Theorem~\ref{prop:non-uniformly-acylindrical}}]\label{thmli:non-uniformly-acyindrical}
    Let $G$ be a group acting properly by graph isomorphisms on a graph $X$ equipped with a $G$--invariant path system $\mc P$. If $\mc K$ is large enough, then the induced action of $G$ on $\hat{X}_{\mc K}$ is non-uniformly acylindrical.
\end{theorem}

Recall that the action of a group $G$ on a hyperbolic space $X$ is {\emph{non-uniformly acylindrical}} if for every $R>0$ there exists $D>0$ such that for every $x$ and $y$ in $X$ at distance at least $D$, the set $\{ g\in G \; ;\; d(x,gx)<R, d(y,gy)<R\}$ is finite; the action is called {\emph{acylindrical}} if the cardinality of the previous set is at most $N$, with $N$ depending only on $R$.  

For $G$ and $X$ as in Theorem \ref{thmli:non-uniformly-acyindrical} where, moreover, $G$ acts coboundedly on $X$, we prove a strong diameter dichotomy that states that at least one of the generalized contraction spaces of $X$ has diameter either $1$ or $\infty$. 

To prove this dichotomy, we compare two generalized contraction spaces and show that if the first does not have diameter one, the diameter of the second is $\infty$.
Note that throughout, by \emph{diameter of the contraction space} we mean the diameter of the set of vertices. This step is also the reason for the `non-virtually cyclic' assumption on $G$: in the comparison of the different generalized contraction spaces we use a combinatorial argument that relies on $G$ having superlinear growth.

\begin{theorem}[{Theorem~\ref{thm:strict-diameter-dichotomy}}]\label{thmi:diameter_dichotomy}
Let $G$ be a non-virtually cyclic group acting properly and coboundedly by graph isomorphisms on a graph $X$ equipped with a $G$--invariant, undirected path system. Then there exists a contraction triple $\mc K$ so that $\hat{X}_{\mc K}$ has either diameter $1$ or $\infty$.
 \end{theorem}

There is a significant qualitative difference between having uniformly bounded diameter and having diameter 1. In the latter case, $\hat{X}_{\mc K}$ is a complete graph, meaning that every pair of vertices was either adjacent in $X$, or its connecting special path was deemed to have `no negative curvature along it'. With the extra assumption  of navigability which, as mentioned, we see as a weak version of `non-positive curvature',  there is no negative curvature to be found along any direction of $X$. This has a fundamental impact on the geometry of the space, for instance, on the divergence of the space. 

Thus, what we are actually proving is a stronger version of the alternative first described in Theorem \ref{thmi:dichotomy}, and we remark that the hypotheses of the following theorem are satisfied by all groups of Corollary~\ref{cori:groups_satisfying_WRR}.

\begin{theorem}[{Theorem~\ref{thm:divergence-dichotomy1}}, Corollary \ref{cor:contracting-element}]\label{thmi:extended_dichotomy}
Let $G$ be a non-virtually cyclic group acting properly and coboundedly by graph isomorphisms on a graph $X$ equipped with a $G$--invariant, undirected navigable path system. 

Then the following dichotomy holds.
 \begin{enumerate}
        \item\label{edich1} Either there exists a contraction space with diameter $1$ (equivalently all contraction spaces have finite diameter) and both $X$ and $G$ have linear divergence, or 
        \item\label{edich2} there exists a contraction space with diameter $\infty$ (equivalently all contraction spaces have infinite diameter), $X$ contains an infinite periodic Morse quasi-geodesic (in fact, an infinite periodic Morse quasi-tree), and $G$ is acylindrically hyperbolic.
    \end{enumerate}
 \end{theorem}

\subsection{Morse property and divergence in the setting of navigability}

We now focus on the class of spaces that satisfy the non-equivariant version of the assumptions of our main theorem, namely spaces that admit a navigable path system, and refer to Subsection~\ref{subsec:spaces_with_navigable_PS} for examples.  

Evidence that, from the coarse geometry viewpoint, navigability is the right replacement of sectional non-positive curvature in the Riemannian sense comes also from the fact that it allows for a local definition of rank one/Morse geodesic mirroring the one from Riemannian geometry that is, in a certain sense, optimal. Indeed, in Riemannian manifolds, the rank of a geodesic is the dimension of the space of parallel Jacobi fields along that geodesic \cite[Chapter 4, $\S $4]{Ballmann}. In particular, rank one means that the only parallel Jacobi fields are the multiples of the speed of the geodesic, equivalently, that infinitesimally there is no flat strip along the geodesic. This local assumption implies a global feature of the geodesic, i.e. the Morse property.

In the coarse setting, geodesics are replaced by quasi-geodesics - and, in our case, even by special paths, and one is tempted to replace the above by the assumption that, at a fixed but large enough scale, triangles with one edge contained in the given quasi-geodesic are thin. However, this is insufficient (see Example~\ref{ex-e-not-2}), but it is enough if instead of triangles one considers quadrilaterals. 

Thus, we say that a quasi-geodesic $\qgot$ is \emph{$n$--weakly polygonally Morse} if short $(n+1)$--gons with one side contained in $\qgot$ and paths of $\mc P$ as the other sides are thin (for the precise formulation see Definition~\ref{defn:weakpolymorse}). 

\begin{theorem}[Coarse Jacobi criterion, {see Theorem~\ref{thm:combined-wpm-results}\eqref{P:special-path} for the precise statement}]\label{thmi:effectiveMLTG}
    Let $X$ be a geodesic metric space with a navigable, undirected path system. Then if a path is a special path that is locally $3$--weakly polygonally Morse, then it is a Morse quasi-geodesic.
\end{theorem}

It is worth noting that checking the local condition of Theorem~\ref{thmi:effectiveMLTG} is much more feasible than checking that a path is a Morse quasi-geodesic and, if the path is finite or periodic, this could be done algorithmically with a computer.


\medskip

A notable consequence of Theorem~\ref{thmi:effectiveMLTG} is that a space equipped with a navigable and undirected path system satisfies the \emph{Morse local-to-global} (MLTG) property as introduced in \cite{RussellSprianoTran:thelocal}. Being Morse local-to-global has several strong consequences, including having stable free subgroups, having $\sigma$-compact Morse boundary, a combination theorem for stable subgroups, a growth gap and very good algorithmic properties \cite{RussellSprianoTran:thelocal, CordesRussellSprianoZalloum:regularity, AbbottZbinden:sigma-compact,DrutuSprianoZbinden:schism}. For an extensive survey of the consequences of the MLTG property we refer to the appendix of \cite{DrutuSprianoZbinden:schism}.

\begin{corollary}[{Corollary~\ref{cor:MLTG}}]\label{cori:MLTG}
    Let $X$ be a geodesic metric space admitting a navigable, undirected path system $\mc P$. Then $X$ has the Morse local-to-global property.
\end{corollary}

We already mentioned that admitting a navigable path system has an impact on the behaviour of divergence in the setting of groups and equivariant path systems. A similar phenomenon occurs for metric spaces. In this generality, the divergence can be a rather complicated function. We show that, for spaces admitting a navigable path system, there are significant restrictions on it. 

\begin{theorem}[{Theorem~\ref{thm1:div}}]\label{thmi:divergence}
  Let $X$ be a geodesic metric space equipped with a bounded, navigable and undirected path system, and assume that $X$ admits a cobounded group action.
    
    Assume moreover that there exists a sequence $n_k\to \infty$ and $\delta>0$ such that $\Dv (n_k, \delta )\leq C n_k$, where $C>0$ is independent of $k$. Then there exists $C'>0$ such that $\dv' (n,2, \delta )\leq C'n$ for every $n$. 
\end{theorem}
In fact, we prove a more general version, namely Theorem~{\ref{thm1:div}}, that does not require a cobounded group action, but what sometimes is called a \emph{visibility condition}, namely the requirement that every point is close to the midpoint of a sufficiently long special path. 

Results such as Theorem \ref{thmi:divergence}, deducing global behaviour from controlled behaviour on a sequence, are known, under appropriate assumptions, for few quasi-isometry invariants, to our knowledge only for the Dehn function and the growth function. Same as for these other two invariants, the above theorem translates into a result of the type: a property known for one asymptotic cone propagates to all asymptotic cones, see Theorem \ref{thm2:div} for details.

Finally, we provide a sufficient criterion for a space to be \emph{Morse dichotomous} \cite{Zbinden:hyperbolic}. A space is \emph{Morse dichotomous} if all of its Morse quasi-geodesics are strongly contracting. Strongly contracting geodesics are significantly better behaved than Morse geodesics. For instance, a group acting geometrically on a (non-elementary) space with a strongly contracting geodesic is acylindrically hyperbolic, while there are examples of groups with a Morse geodesic (even a Morse element) that are not acylindrically hyperbolic \cite{OOS}. The property of being Morse dichotomous is not invariant under quasi-isometries and hence requires finer tools to be established. 
Our result is the following. 

\begin{theorem}[Corollary~{\ref{cor:geodesic-combing-implies-morse-dichotomous}}]\label{thmi::geodesic-combing-implies-morse-dichotomous}
    Let $X$ be a geodesic metric space whose set of geodesics forms a navigable path system. Then $X$ is Morse dichotomous.
\end{theorem}

\subsection{Metric spaces with navigable path systems}\label{subsec:spaces_with_navigable_PS}

We already mentioned some of the most notable examples of spaces and groups admitting navigable path systems. Here we provide a more extensive discussion, as well as sufficient criteria, for navigability. 

A first source of examples comes from spaces endowed with (bi-)combing. See Section~\ref{sect:combings} for definitions. Given a (bi)-combing, we can define a \emph{combing path system} essentially by taking all subpaths of the (bi-)combing (Definition~\ref{defn:combing_path_system}). Combing path systems are undirected even if the (bi)-combing is not reversible. It suffices to have very mild conditions to obtain navigability.

\begin{proposition}[{Proposition~\ref{prop:combing-navigability}}]\label{propi:combing-navigability}
    Let $X$ be a geodesic metric space equipped with a bounded geodesic combing. The set of all geodesics in $X$ is a navigable path system. 
\end{proposition}

The above combined with Theorem \ref{thmi::geodesic-combing-implies-morse-dichotomous} yields the following. 

\begin{corollary}[Corollary~\ref{cor:geodesic-combing-implies-morse-dichotomous}]\label{cori:geodesiccombing-is-md}
    A geodesic metric space equipped with a bounded geodesic combing is Morse dichotomous.
\end{corollary}
Requiring the combing to be geodesic might be too strong a condition. We relax it to admit quasi-geodesics, with the trade-off of requiring a \textit{bi}combing and quasi-consistency.  

\begin{proposition}[{Proposition~\ref{prop:bicombing-is-navigable}}]\label{propi:bicombing-is-navigable}
    The combing path system induced by a quasi-consistent, bounded, quasi-geodesic bicombing on a geodesic metric space is navigable. 
\end{proposition}

A second source of examples comes from (spaces quasi-isometric to) \emph{median spaces}. Median spaces have been studied first from an algebraic viewpoint in \cite{VandeVel:book,BandeltHedlikova,Isbell:median,Sholander1,Sholander2}. Their investigation from a geometric viewpoint can be found in \cite{Roller:median,NicaMaster,Bow14,Bow16,Bowditch:injective}. Examples of median metric spaces include trees, $\R^n$ with the $\ell^1$ metric for any $n\geq 1$, and CAT(0) cube
complexes with the Euclidean metric on cubes replaced by the $\ell^1$ metric. According to Chepoi, Gerasimov and Roller \cite{Chepoi:graphs,Gerasimov:semisplittings,Gerasimov:fixedpoint}, the
class of $1$-skeleta of CAT(0) cube complexes coincides with the class of {\it median graphs}, i.e. simplicial graphs whose $0$-skeleton with the combinatorial distance is median. General median spaces can be thought of as non-discrete versions of $0$-skeleta of CAT(0) cube complexes. This analogy is further emphasized by work of Bowditch \cite{Bow14,Bow16} who proved that the metric of a complete connected finite rank median metric space has a bi-Lipschitz deformation that is CAT(0). Median graphs are relevant in graph theory, computer science \cite{ChepoiBandelt}, and optimization theory \cite{MMR,Wildstrom}. Kazhdan's Property (T) and a-T-menability can be defined in terms of actions on median spaces \cite{CDH-T}.

\begin{proposition}[{Proposition~\ref{prop:median-implies-navigable}}]\label{propi:median-implies-navigable} The path system consisting of all geodesics in a median geodesic metric space is navigable.     
\end{proposition}

The previous results combine to yield the following. 

\begin{theorem}[{Corollary~\ref{cor:median-implies-morse-dichotomous}}]\label{cori:median-is-md}
    Every median geodesic metric space is Morse dichotomous.
\end{theorem}

\medskip

The following metric spaces satisfy the coarse Jacobi criterion for recognition of Morse quasi-geodesics Theorem \ref{thmi:effectiveMLTG}, the Morse local-to-global property and the criterion for linear divergence Theorem \ref{thmi:divergence} and are, for the most part, Morse dichotomous. Note that, for the last property, the assumption needed is that the set of geodesics forms a navigable path system. We refer to Section \ref{sec:metric-spaces} for details on which spaces satisfy this assumption, as well as for a longer list of examples.  

\begin{itemize}
    \item Properly convex cones endowed with the Thompson's metric and projectivizations of such cones endowed with the Hilbert metric;
    \item Injective metric spaces and their $\sigma$-convex subspaces; in particular spaces of probability measures endowed with the $1$-Kantorovich-Wasserstein distance and Teichmüller spaces with the Weil-Petersson metric;
    \item Systolic simplicial complexes;
    \item Hierarchically hyperbolic spaces;
    \item Helly graphs.
\end{itemize}

\subsection*{Outline of the paper}Section~\ref{sect:prelims} contains standard definitions and can be used mainly for reference. 

In Section~\ref{sect:path_systems} we recall path systems and introduce properties that play a central part in our results, first and foremost navigability (see Definition~\ref{defn:navigable}), and negative curvature along path systems (see Definition~\ref{defn:weakpolymorse}). We also define projection points and discuss their properties in Section~\ref{sec:projections}. In Section~\ref{sec:globalizations} we prove metric results about path systems including various local-to-global results such as Theorem~\ref{thmi:effectiveMLTG}, Corollary~\ref{cori:MLTG} and the propagation of linearity from a sequence to the entire divergence function (Theorem~\ref{thmi:divergence}).

In Section~\ref{sec:finding_navigable} we prove that various spaces admit navigable path systems (Propositions~\ref{prop:combing-navigability}, \ref{propi:bicombing-is-navigable}, and \ref{propi:median-implies-navigable}). We also show that navigability is invariant under quasi-isometries (Lemma~\ref{lem:navigable-push-forward}) and follows from a weak version of the classical boundedness of bicombings, when coupled with another property called \emph{avoidance} (Proposition~\ref{prop:bounded-avoidant-implies-navigable}). Further, we prove Proposition~\ref{propi:Morseiffqiem} and Corollaries~\ref{cori:median-is-md} and~\ref{cori:geodesiccombing-is-md}.

Section~\ref{sec:generalised_contraction} contains the construction of the generalized contraction space and the proof of Theorems~\ref{thmi:contraction_sp_hyperbolic}, ~\ref{propi:Morseiffqiem},~\ref{thmli:non-uniformly-acyindrical}. 

In Section~\ref{sect:comparing_contraction_spaces} we prove the diameter dichotomy (Theorem~\ref{thmi:diameter_dichotomy}). In Section~\ref{sect:the_dichotomy} we prove the main Theorem~\ref{thmi:dichotomy}. Finally, in Section~\ref{sec:examples} we list examples to which our theorems apply.

\section*{Acknowledgments} 
We are thankful to Jingyin Huang for providing many examples of spaces and groups to which our results apply, to Macarena Arenas, Giuliano Basso, Victor Chepoi, Thomas Haettel, Urs Lang, Damian Osajda, Harry Petyt and Cormack Walsh for useful conversations on bicombings and their applications, and to Stephan Stadler for interesting inputs on metric geometry. 

The authors are grateful to Max Planck Institute for Mathematics in Bonn for its hospitality while work at this paper has been carried out. They would also like to thank the Isaac Newton Institute for Mathematical Sciences, Cambridge, for support and hospitality during the programme Operators, Graphs, Groups, where work on this paper was undertaken. This work was supported by EPSRC grant EP/Z000580/1 and by Christ Church Oxford's research centre. The second author is supported by a Royal Society
University Research Fellowship URF\textbackslash R1\textbackslash 251707.  The third author is supported by the Postdoc Mobility grant \#P500PT\_230322 of the Swiss National Science Foundation, the European Union (ERC, SATURN, 101076148), and the Deutsche Forschungsgemeinschaft (DFG, German Research Foundation) under Germany's Excellence Strategy - EXC-2047/1 - 390685813.

\medskip

For the purpose of Open Access, the authors have applied a CC BY public copyright licence to any Author Accepted Manuscript version arising from this submission.

\medskip 

\noindent {\textit{Declaration of AI use:}} AI has only been used for spellcheck.

\section{Preliminaries}\label{sect:prelims}

\begin{cvn}\label{conv:geo-iso}
Throughout the paper, unless otherwise stated, all metric spaces are geodesic and $X$ denotes a geodesic metric space; all actions of groups on metric spaces are by isometries.

\end{cvn}

\subsection{Quasi-geodesics}

\begin{definition}
    A function $f: X \to Y$ between geodesic metric spaces $(X, d_X)$ and $(Y, d_Y)$ is a $C$--quasi-isometry if $\dist_{\mathrm{Haus}}(f(X), X) \leq C$  and for all $x, y\in X$
    \begin{align*}
        \frac{d_X(x, y)}{C}-C \leq d_Y(f(x), f(y))\leq Cd_X(x, y)+C.
    \end{align*}
\end{definition}

\begin{definition}
A \emph{$(\la, \ka)$--quasi-geodesic in $X$}, for $\la \geq 1$ and $\ka \geq 0$, is a map $\gamma \colon I \to X$, where $I\subseteq {\mathbb{R}}$ is an interval, such that \[\frac{1}{\la}\vert t_1 - t_2 \vert  - \ka \leq \dist (\gamma(t_1), \gamma(t_2)) \leq \la \vert t_1- t_2\vert + \ka.\]

We call a pair of constants $(\la, \ka)$, with $\la \geq 1$ and $\ka \geq 0$, a \emph{quasi-geodesic pair of constants}, or simply a \emph{quasi-geodesic pair}. We call a $(\la, \la)$--quasi-geodesic a $\la$--quasi-geodesic. 
\end{definition}

Given $r>0$, the {\em $r$--neighbourhood} of
a subset $A$, i.e. $\{x\in X: \dist (x, A)\leq r\}$, is denoted by $\nn_r(A)$.  In particular, if $A=\{a\}$ then $\nn_r(A)=B(a,r)$
is the {\em closed $r$--ball centred at~$a$}.

 \begin{notation}
  Consider a path $\pgot \colon [a,b] \to X$. Given a point $x\in X$ and a subset $A\subseteq X$, we write, by abuse of notation, $x\in \pgot$ to mean that $x$ belongs to the image of $\pgot$, and $\pgot \subseteq A$ to mean that the image of $\pgot$ is contained in $A$.
 \end{notation}
 
For all $x, y\in X$, we denote by $[x, y]$ a choice of geodesic between $x$ and $y$. We assume that the choice of geodesic is invariant under the action of the isometry group of $X$. For a path $\pgot \colon [a,b] \to X$ we denote by $\pgot^{-1}$ the path $\pgot^{-1} : [a, b]\to X$ defined as $\pgot^{-1}(t) = \pgot(b+a - t)$, and we denote $\pgot(a)$ by $\pgot^-$ and $\pgot(b)$ by $\pgot^+$. We denote the length of a rectifiable path $\pgot$ by $\length {\pgot}$ and the length of its domain, $\abs{b-a}$, by $\dlength{\pgot}$. We call $\pgot((b-a)/2)$ the \emph{midpoint} of $\pgot$. With an abuse of notation, we denote by $\isub{\pgot}{s,t}$ the restriction of $\pgot$ to the interval $[s,t]$. 
Let $u, v\in \pgot$. We say that $x\in \pgot$ \emph{lies between} $u$ and $v$ if there are $s_1\leq s_2\leq s_3$ in the domain of $\pgot$ such that $\pgot(s_1) = u$, $\pgot(s_2) = x$, $\pgot(s_3)= v$. If $s_1$ is the smallest parameter so that $\pgot(s_1) = u$ and $s_3$ is the largest parameter so that $\pgot(s_3) = v$ we denote by $\psub{\pgot}{u,v}$ the restriction of $\pgot$ to $[s_1, s_3]$ (i.e. the maximal sub-path of $\pgot$ with image composed of points that lie between $u$ and $v$). For any point $x\in \pgot$, we define the \emph{prefix of $\pgot$ before $x$, respectively the suffix of $\pgot$ after $X$,} as $\pref{\pgot}{x} = \isub{\pgot}{a, s_1}$ and $\suf{\pgot}{x}=\isub{\pgot}{s_2,b}$, where $s_1, s_2\in [a, b]$ are the smallest, respectively the largest parameter with $\pgot(s_i) = x$. With this notation, for any $u, v\in \pgot$ (with at least one point $x$ between $u$ and $v$) we have $\pgot = \pref{\pgot}{u}\ast \psub{\pgot}{u, v}\ast \suf{\pgot}{v}$.

\begin{lemma}[Improved quasi-geodesics \cite{Bridson-Haefliger}, Lemma 1.11, \cite{Burago-Ivanov}, Proposition 8.3.4]\label{lemma:tamingqg}
Let $X$ be a geodesic metric space. For every  $(\la, \ka)$-quasi-geodesic $\gamma: [a, b]\to X$ there exists a continuous $(\la, \ka')$-quasi-geodesic $\bar{\gamma} : [a, b]\to X$ such that 
\begin{enumerate}
    \item $\gamma(a) = \bar{\gamma}(a)$ and $\gamma(b) = \bar{\gamma}(b)$;
    \item $\ka' = 2(\la + \ka)$;
    \item $\length {\isub{\bar{\gamma}}{t, t'}}\leq k_1 \dist (\bar{\gamma}(t), \bar{\gamma}(t'))+k_2$ for all $t, t'\in [a, b]$, where $k_1 = \la(\la+\ka)$ and $k_2 = (\la\ka' + 3)(\la+\ka)$;
    \item the Hausdorff distance between the images of $\gamma$ and $\bar{\gamma}$ is less than $\la+\ka$.
\end{enumerate}
\end{lemma}

We call such a quasi-geodesic $\bar{\gamma}$ an {\em improvement} of $\gamma$. A $(\la, \ka)$--quasi-geodesic is \emph{improved} if it is continuous and condition $(3)$ of Lemma~\ref{lemma:tamingqg} holds.

\begin{cvn}
From now on, all quasi-geodesics are assumed to be improved.
\end{cvn}

\begin{lemma}\label{lemma:quasi-geodesic-concatenation}
   Let $Q\geq 1$ and $\rho\geq 0$ be constants. Let $\gamma, \eta$ be $Q$-quasi-geodesic segments with $d(\gamma^-, x)\leq \rho$ for some $x\in \eta$. If $\abs{\gamma}\geq (16Q^4 + Q\rho)$, then 
   \begin{align}\label{eq:far-away-line}
       \max\left\{d(\gamma^+,\psub{\eta}{\eta^-, x}),  d(\gamma^+, \psub{\eta}{x, \eta^+}) \right\}\geq \frac{\abs{\gamma}}{Q'},
   \end{align}
   where $Q' = 32Q^3$.  
\end{lemma}

\begin{proof}
    Let $s$ be such that $\eta(s) = x$. Let $z_1 = \eta(t_1)$ and $z_2 = \eta(t_2)$ be points on $\psub{\eta}{\eta^-, x}$ and $\psub{\eta}{x, \eta^+}$ respectively. Assume that for $i = 1$ and $i=2$,  $d(z_i, \gamma^+)\leq \ell$ for some $\ell\leq \abs{\gamma}/(4Q)$. For $i = 1, 2$, by the triangle inequality, 
    \begin{align*}
        d(z_i, x)\geq  d(\gamma^-, \gamma^+) - d(\gamma^-, x) - d(\gamma^+, z_i)\geq \frac{\abs{\gamma}}{Q} - Q -  \rho - \ell \geq \frac{\abs{\gamma}}{8Q} \geq 2Q^2.
    \end{align*}
    For the last step, we used that $\ell\leq\abs{\gamma}/(4Q)$ and $\abs{\gamma}/(4Q)\geq Q +\rho$. It follows that $t_1\leq s \leq t_2$ and for $i = 1, 2$
    \begin{align*}
        \abs{s - t_i}\geq \frac{\abs{\gamma}}{8Q^2} - 1\geq \frac{\abs{\gamma}}{16Q^2}.
    \end{align*}
    Here we used that $\abs{\gamma}\geq 16Q^2$. Hence
    \begin{align*}
        2\ell \geq d(z_1, z_2)\geq \frac{2\abs{\gamma}}{16Q^3} - Q\geq \frac{\abs{\gamma}}{16Q^3}.
    \end{align*}
    Here, we used that $\abs{\gamma}/(8Q^3)\geq Q$. Hence $2\ell\geq \abs{\gamma}/(16Q^3)$ and the statement follows.
\end{proof}

\subsection{Quasi-geodesic combings}\label{sect:combings}
Consider a metric space $X$ and two constants $\lao \geq 1$ and $\kao \geq 0$. 
A \emph{$(\lao,\kao)$--quasi-geodesic (bi)-combing} is a way of assigning to  every ordered pair of points $(x, y)\in X\times X$, a $(\lao,\kao)$--quasi-geodesic $\col_{xy}$ connecting them. The quasi-geodesics $\col_{xy}$ are called \emph{combing lines}. When needed, we assume the  quasi-geodesics to be extended to $\R$ by constant maps. 

We say that a $(\lao, \kao)$--quasi-geodesic combing is \emph{bounded} if 
\begin{equation}\label{eq:dist1}
 \dist (\col_{xy_1}(t), \col_{xy_2}(t))\leq \kao\dist (y_1,y_2)+\kao,    
\end{equation} for all $t\in \R$ and $x, y_1,y_2$ in $X$.

Similarly, a $(\lao, \kao)$--quasi-geodesic bicombing is \emph{$\kao$--bounded} if
\begin{equation}\label{eq:dist2}
 \dist (\col_{x_1y_1}(t), \col_{x_2y_2}(t))\leq \kao\left[\dist (x_1,x_2)+\dist (y_1,y_2)\right]+\kao,    
\end{equation} for all $t\in \R$ and $(x_1,y_1), (x_2,y_2)$ in $X\times X$.

A $(\lao, \kao)$--quasi-geodesic bicombing is called \emph{$\kao$--quasi-consistent} if for any combing line $\col_{xy}$ and points $z_1, z_2$ in the image of it, we have \[\dist_{\mathrm{Haus}}(\psub{\col_{xy}}{z_1z_2}, \col_{z_1z_2})\leq \kao.\]

\subsection{Morse properties}

\begin{definition}[Morse gauge]
    A \emph{Morse gauge} is a function $M \colon \mathbb{R}_{\geq1} \times \mathbb{R}_{\geq0} \to \mathbb{R}_{\geq 0}$ that is non-decreasing in each of the two variables, and is continuous in the second variable. 
 \end{definition}

As shown in \cite{CordesSistoZbinden:corrigendum}, the continuity in the second variable ensures that the $M$--Morse stratum $\mb ^M X$ is compact in the Morse boundary for all Morse gauges $M$. 

\begin{definition}[Morse quasi-geodesic]\label{defn:Morse_QG}
A quasi-geodesic $\gamma$ is \emph{$M$--Morse}, for a Morse gauge $M$, if any $(Q, q)$--quasi-geodesic $\eta$ with endpoints $\gamma(s)$ and $\gamma(t)$ satisfies 
    \[
    \eta \subseteq N_{M(Q,q)}(\isub{\gamma}{s,t}).
    \]
\end{definition}

One of the goals of this paper is to determine when and to what extent local-to-global properties hold for Morse quasi-geodesics, in the sense of the following definition, first formulated in \cite{RussellSprianoTran:thelocal}. 

\begin{definition}
    We say that a path $\pgot \colon [a, b] \to X$ \emph{satisfies a property $P$ at scale $L$} (or that \emph{$\pgot$ $L$--locally satisfies $P$}, or \emph{$\pgot$ is $L$--locally $P$}) if for every $t_1, t_2 \in [a,b]$ with $\vert t_1- t_2 \vert \leq L$ the restriction $\isub{\pgot}{t_1, t_2}$ has the property $P$. The quantity $L$ is called the \emph{scale}. We say that a path \emph{is locally $P$} when it is $L$--locally $P$ for some scale $L$.  
\end{definition}

A `local-to-global property' for paths requires that, whenever a path is $L$--locally $P$, for $L$ large enough, it is globally $P'$. For instance, if we take $P=P'$ to be the property of being a geodesic, then CAT(0) spaces (or in general Busemann spaces) satisfy the local-to-global property for all $L>0$. Another notable example is a theorem of Gromov stating that a geodesic space is hyperbolic if any only if for each $(\lao, \kao)$ there exists $L$ and $ (\la, \ka)$ so that each path that $L$--locally is a $(\lao, \kao)$--quasi-geodesic is globally a $(\la, \ka)$--quasi-geodesic \cite{Gromov(1983)}.

We focus on local-to-global properties that govern how paths that are locally Morse quasi-geodesics are globally Morse quasi-geodesics. The most studied such relation is the Morse local-to-global property, discussed at length in \cite[Appendix~A]{DrutuSprianoZbinden:schism}.

\begin{definition}\label{defn:MLTG}
    A metric space $X$ satisfies the \emph{Morse local-to-global} (\emph{MLTG} for short) \emph{property} if the following holds. For any quasi-geodesic pair $(\la, \ka)$ and Morse gauge $M$ there exists a scale $L$, a quasi-geodesic pair $(\la', \ka')$ and a Morse gauge $M'$ such that every path that is $L$--locally an $M$--Morse $(\la, \ka)$--quasi-geodesic is globally an $M'$--Morse, $(\la', \ka')$--quasi-geodesic. 
\end{definition}

We conclude this subsection by recording the notion of {\emph{strongly contracting quasi-geodesic}}, which is a strong version of being Morse.

\begin{definition}\label{def:strongly_contracting}
    A quasi-geodesic $\qgot$ is called \emph{$C$--strongly contracting} if for each ball $B$ disjoint from $\qgot$ we have 
    \[\mathrm{diam}(\pi_\qgot(B)) \leq C,\]
    where $\pi_\qgot \colon X \to 2^\qgot$ denotes the closest point projection.
\end{definition}

All strongly contracting geodesics are Morse (\cite[Theorem~1.3]{ACHG:contraction_morse_divergence}, see also \cite{Sultan:CAT0}) and a space is called \emph{Morse dichotomous} if the converse holds in a quantitative fashion. 

\begin{definition}\label{def:morse-dichotomous}
    A geodesic metric space $X$ is called Morse dichotomous if for every Morse gauge $M$ there exists a constant $C$ such that all $M$--Morse geodesics are $C$--contracting. 
\end{definition}

\subsection{Divergence}\label{sec:div}

The divergence is a function that estimates the minimal length of a path joining two points while staying away from a ball around a third point. 

Let $(X, d)$ be a geodesic metric space, let $0< \delta  < 1$ and let $\eps> 0$. For points $a, b, c\in X$ we define $\dv_{\eps}(a, b,c; \delta)$ as the infimum of $\norm{\pgot}$ for all paths $\pgot$ from $a$ to $b$ with $d(\pgot, c) \geq \delta d(c, \{a, b\}) - \eps$. 

The  \emph{$(\eps, \delta)$--divergence function} $\Dv_\eps(n ,\delta)$ of the space $X$ is defined as the supremum of $\dv_\eps (a,b,c;\delta)$ for all $a, b, c, \in X$ with $d(a,b)\leq n$. 

Recall that, for all functions estimating some large scale geometric feature, we work with equivalence classes with respect to the following relation. Two functions $f, g$ are {\emph{equivalent}} if there exists a constant $C\geq 1$ such that $g(n/C)/C-n-1\leq f(n)\leq Cg(Cn) + Cn + C$ for all $n$. 

\begin{lemma}[{\cite{DrutuMozesSapir}}]\label{lemma:divergence-def}
    Let $G$ be a finitely generated group acting properly and coboundedly on a geodesic metric space $X$. Then there exists $\delta_0, \eps_0 > 0$ such that for all $\delta \leq \delta_0$ and $\eps \geq \eps_0$ the equivalence class of $\Dv_\eps(n, \delta)$ depends only on $G$. 
\end{lemma}

For $X$ (and $G$) as in Lemma~\ref{lemma:divergence-def}, we say that the \emph{divergence} of $X$ (or $G$) is the equivalence class of $\Dv_{\eps_0}(n, \delta_0)$.

There are other ways to define the divergence function (\cite{Gersten:divergence, Gersten:divergence3, DrutuMozesSapir}) and those are shown to be equivalent to the above definition by \cite{DrutuMozesSapir}.

\subsection{Asymptotic cones}

For the notion of asymptotic cone, we refer to M. Gromov \cite{Gromov:Asymptotic}. Further properties and open questions can be found in \cite{Drutu:survey} and in  \cite{DrutuSapir:TreeGraded}. 

Given a metric space  $(X, \dist )$, a sequence $(o_n)$ of basepoints in it, and $(d_n)$ a sequence of positive numbers diverging to $\infty$, when building the corresponding asymptotic cone, the goal is to build a complete metric space that appears as limit of the sequence of rescaled metric spaces  $(X, \frac{1}{d_n}\dist )$ with basepoints $o_n$. This is done \textit{via} the choice of a non-principal ultrafilter $\omega$ which, in some sense, selects a converging subsequence from the given sequence.   The limit thus obtained is denoted by  $\Con^\omega(X, (o_n),(d_n))$.

In what follows, by {\emph{cut-point}} of a topological space $X$ we always mean global cut-point. 

With the terminology of \cite{DrutuMozesSapir}, a metric space $X$ is called {\em wide} if none of its asymptotic cones has a
cut-point; it is called \emph{unconstricted} if one of its asymptotic cones does not have cut-points.

\begin{proposition}[\cite{DrutuMozesSapir}, Lemma 3.17]\label{prop2:one-ended}
 Let $X$ be a geodesic metric space whose isometry group acts coboundedly on $X$. The following are equivalent:
 \begin{enumerate}
     \item $X$ is wide;
     \item $\Dv_\eps (n;\delta)$ is bounded by a linear function for small enough $\eps$ and $\delta$.
 \end{enumerate}
    \end{proposition}

For a more detailed discussion relating existence of cut-points, divergence and existence of Morse geodesics we refer to \cite{DrutuMozesSapir}.

\section{Path systems}\label{sect:path_systems}

In this section, we recall the notion of {\emph{path system}}, introduced in \cite{Sisto:contracting}, and a few of its immediate properties. We also define new properties, either referring to properties shared by the paths (such as {\em contracting} or {\em midthin}), or to the geometry of a space endowed with a path system (such as {\em navigable} and {\em avoidant}). This will allow us to isolate a minimal set of assumptions necessary for our theorems. Examples of path systems with the latter properties are provided in Section \ref{sec:finding_navigable}. Roughly, the main sources of examples come from bounded (bi)combings, geodesics in  median geodesic spaces and their images by quasi-isometry.  

\subsection{Definitions, notation, basic properties}\label{sect:basicspathsyst}

Recall that we are assuming that all quasi-geodesics are improved (see Convention~\ref{conv:geo-iso}).

\begin{definition}\label{def:heap}
    Let $X$ be a geodesic metric space. A \emph{heap of paths} $\mc Q$ on $X$ is a set of uniform quasi-geodesic segments in $X$ that contains the trivial path from $x$ to itself for all $x\in X$.
\end{definition}

 Let $X$ be a geodesic metric space and $\mc Q$ a heap of paths on $X$. We say that $\mc Q$ is \emph{$(\lao, \kao)$--quasi-geodesic} if all paths in $\mc Q$ are $(\lao, \kao)$--quasi-geodesics. In this case, we call $(\lao, \kao)$ the \emph{quasi-geodesic constants} of $\mc Q$. We call the paths in $\mc Q$ \emph{special paths}. We denote by $\overline{\mc Q}$ the set containing $\mc Q$ and all of its subsegments, and call it the \emph{consistent closure} of $\mc Q$. A heap of paths $\mc Q$ is called \emph{undirected} if for each special path $\spath \in \mc Q$, its inverse $\spath^{-1}$ is in $\mc Q$. Given an arbitrary heap of paths $\mc Q$, we denote by $\mc Q^{\pm}$ the undirected heap of paths it induces, that is, the heap of paths which contains all paths of $\mc Q$ and their inverses. We call $\mc Q^\pm$ the \emph{directional closure of~$\mc Q$}. We say that a path $\pgot$ is an $(n, \mc Q)$--polygonal line if it is the concatenation of up to $n$ paths $\spath_i$ in $ \mc Q$.

We use the following to simplify computations later in the paper.

\begin{remark}\label{remark:definition-of-dp}
    Given a heap of paths $\mc Q$, there exists a constant $C_{\mc Q}\geq 1$ and a linear function $D_{\mc Q}(r) = C_{\mc Q}r+C_{\mc Q}$ depending only on the quasi-geodesic constants of $\mc Q$ satisfying the following. Let $\spath\in \mc Q$ be a special path and let $A = \{\abs{\spath}, \norm{\spath}, \diam(\spath), d(\spath^-, \spath^+)\}$. Then $D_{\mc Q}(\min(A))\geq \max(A)$.
\end{remark}

A heap of paths with extra properties is called a \emph{path system}, a notion introduced in \cite[Definition~2.1]{Sisto:contracting}.

\begin{definition}\label{def:path syst}
    A heap of paths $\mc P$ on $X$ is called a \emph{path system} if it is equal to its consistent closure and for every pair $x, y\in X$ there exists a path in $\mc P$ from $x$ to $y$.
\end{definition}

Note that, in the previous definition, there may be several paths in $\mc P$ joining two given points $x$ and $y$. Given a path system $\mc P$ on $X$ and a pair of points $x, y\in X$, we denote by $[x, y]_{\mc P}\in \mc P$ a choice of special path from $x$ to $y$.

Let $(X, d_X)$ and $(Y, d_Y)$ be geodesic metric spaces and let $f : X\to Y$ be a $C$--quasi-isometry from $X$ to $Y$. Further, let $\mc P$ be a path system in $X$. We want to use $f$ and $\mc P$ to define a path system $\mc P_{f}$ in $Y$, which we will call a {\emph{$C$--push-forward of}} $\mc P$ by $f$. Given a path $\spath \in \mc P$, the composition $f\circ \spath$ is not necessarily continuous (or improved). However, we can use \emph{rectifications of paths} (see below) to get a path $\spath'$ which is at bounded Hausdorff distance from $f\circ \spath$. 

\begin{definition}
    Let $\pgot : [a, b]\to X$ be a not necessarily continuous path and let $k = \ceil{b-a}$. We say that $\qgot$ is a \emph{rectification of $\pgot$} if $\qgot = \gamma_0\ast \ldots \ast \gamma_k$ such that for $0\leq i < k$, $\gamma_i$ is a geodesic from $\pgot(a+i)$ to $\pgot(a+i+1)$ and $\gamma_k$ is a geodesic from $\pgot(a+k-1)$ to $\pgot (b)$.
\end{definition}

We are now ready to define the push-forward of a path system.

\begin{definition}
    Let $f: (X, d_X) \to (Y, d_Y)$ be a $C$--quasi-isometry. The \emph{$C$--push forward $\mc P_f$ of $\mc P$ by $f$} is obtained as follows. Let $\mc{P}_f'$ be the set of paths $[y, f(x)]\ast \spath' \ast[f(x'), y']$ where $\spath\in \mc{P}$ is a special path between $x$ and $x'$, $\spath'$ is the rectification of $f \circ \spath$, and $y,y'$ are such that $d_Y(f(x), y)\leq C$ and $d_Y(f(x'), y')\leq C$. Then $\mc{P}_f$ is the set of all subsegments of paths in $\mc{P}_f'$.
\end{definition}

\begin{remark}\label{rem:push-forward-properties}
   The following properties of the $C$--push-forward $\mc P_f$ follow directly from the definition.
    \begin{enumerate}
        \item \label{prop:push-forward:quasi-geo-constants} If $\mc P$ is a $(\lao, \kao)$--quasi-geodesic path system, then $\mc P_f$ is a $(\lao', \kao')$--quasi-geodesic path system where $\lao', \kao'$ depend only on $\lao, \kao$ and $C$.
        \item \label{prop:equiv-push-forward} Let $G$ be a group acting on both $X$ and $Y$. If $f$ is $G$--equivariant and $\mc P$ is $G$--invariant, then $\mc P_f$ is $G$--invariant.
       \item \label{prop:hausdorff-push-forward} For every special path $\spath\in \mc P$ there exists a special path $\spath '\in \mc P_{f}$ which has the same endpoints as $f( \spath )$ and such that $d_{\mathrm{Haus}}(f(\spath), \spath')\leq C_f$, where $C_f$ is a constant depending only on $C$ and the quasi-geodesic constants of $\mc P$. 
       \item \label{prop:hausdorff-pull-back}For every special path $\spath'\in \mc P_{f}$ there exists a special path $\spath \in \mc P$ such that $d_{\mathrm{Haus}}(f(\spath), \spath')\leq C_f$ and the endpoints of $f(\spath)$ are at distance at most $C$ from the respective endpoints of $\spath'$. Again, $C_f$ is a constant depending only on $C$ and the quasi-geodesic constants of $\mc P$.
       \item \label{prop:unidrected-push-forward} If $\mc P$ is undirected, then Properties \eqref{prop:push-forward:quasi-geo-constants} - \eqref{prop:hausdorff-pull-back} also hold after replacing $\mc P_f$ with $ (\mc P_f)^\pm$.
    \end{enumerate}
\end{remark}

\subsection{Morse and contraction in terms of path systems}
We recall the definition of {\em contraction} with respect to a path system, as introduced in \cite{Sisto:contracting}. We then introduce the notion of \emph{polygonally Morse} geodesics. Both represent weak versions of the notion of geodesics along which the curvature is negative (e.g. geodesics for which the space of parallel Jacobi fields is of dimension one).

\begin{definition}[{\cite[Definition 2.2]{Sisto:contracting}}]\label{def:contracting} Let $\mc P$ be a path system. A subset $A \subseteq X$ is called $\mathcal{P}$--{\emph{contracting with constant $C$}} if there
exists a map $\pi_A \colon X \to A$ such that
\begin{enumerate}
    \item $\dist (x, \pi_A(x)) \leq C$  for every $x \in  A$,\label{5.2.1}
    \item for every $x, y \in  X$, if $\dist (\pi_A(x), \pi_A(y))\geq C$ then for every special path $[x, y]_{\mc P}$ from $x$ to $y$ we have $\dist ([x, y]_{\mc P}, \pi_A(x))\leq C$, $\dist ([x, y]_{\mc P}, \pi_A(y))\leq C$.\label{5.2.2}
\end{enumerate}
\end{definition}

The notion of $\mc P$--contracting is very reminiscent of the notion of strongly contracting (Definition~\ref{def:strongly_contracting}), where (\ref{5.2.2}) plays the role of the  \emph{bounded geodesic image property} (see \cite[Corollary~7.2]{ACHG:contraction_morse_divergence}). The next two results establish the link between being $\mc P$--contracting, which is a property that \textit{a priori} only concerns the path system, and the global geometry of the space.

\begin{lemma}[{\cite[Lemma~2.4 (2)]{Sisto:contracting}}]\label{lem:projection_is_Lipschitz}
    If $A$ is $\mc P$--contracting with constant $C$ the map $\pi_A \colon X \to A$ is coarsely Lipschitz, where the constants depend only on $C$ and the quasi-geodesic constants of the path system.
\end{lemma}

\begin{theorem}[{\cite[Lemma~2.8~(1)]{Sisto:contracting}}]\label{theorem:contracting-implies-morse}
    Let $\mc P$ be a path system on a metric space $X$. For every $C$ there exists $D$ such that if $A\subset X$ is $\mc P$--contracting with constant $C$, then any $C$--quasi-geodesic with endpoints in $A$ is contained in $\mc N_D (A)$. In particular, if $\gamma$ is a quasi-geodesic which is $\mc P$--contracting with constant $C$, then $\gamma$ is $M$--Morse for some Morse gauge $M$ depending only on $C$, $\mc P$ and $X$. 
\end{theorem}

Note that a subset that is $\mathcal{P}$--contracting with constant $C$ is also $\mathcal{P}$--contracting with constant $C'$, for every $C'\geq C$. We record the following lemma. 

\begin{lemma}\label{lem:contracting_preserved_Hdistance}
    Let $\mc P$ be a path system on a geodesic metric space $X$. If $\eta$ is a quasi-geodesic which is $\mc P$--contracting with constant $C$ and $\gamma$ is a quasi-geodesic with endpoints at distance at most $E$ from $\eta$, then $\gamma$ is $\mc P$--contracting with constant $C'$, where $C'$ depends only on $C$ and the quasi-geodesic constants of $\gamma$ and $\eta$. 
\end{lemma}
\begin{proof}
    By appending a geodesic segment of length at most $E$ to each endpoint of $\gamma$, we can assume that $\gamma$ is a subsegment of a quasi-geodesic $\gamma'$ that has endpoints on $\eta$, has the same multiplicative constant as $\gamma$, and additive constant increased by $2 E$. By Theorem~\ref{theorem:contracting-implies-morse}, there is $D$ so that $\gamma'$ is in the $D$--neighbourhood of $\eta$. Thus, we can define a map $p_\gamma \colon \eta \to \gamma$ so that for each $x\in \eta$ we have $d(x, \gamma) \leq d(x, p_\gamma(x)) + 1$. Define $\pi_\gamma = p_\gamma \circ \pi_\eta$.  Since both $p_\gamma$ and $ \pi_\eta$ are uniformly coarsely Lipschitz by Lemma~\ref{lem:projection_is_Lipschitz}, and the former has bounded displacement, \eqref{5.2.1} and \eqref{5.2.2} hold.  
\end{proof}

We now introduce notions that, in the context of coarse geometry, describe a rank one quasi-geodesic. Recall that, in a Riemannian manifold, a geodesic is rank one when its space of parallel Jacobi fields is of dimension one, which roughly says that infinitesimally there is no flat strip along the geodesic \cite[Chapter 4, $\S $4]{Ballmann}. 

\begin{definition}[Weakly polygonally Morse]\label{defn:weakpolymorse}
Let $X$ be a geodesic metric space with a path system $\mc P$. A quasi-geodesic $\gamma$ is \emph{$(\epsilon, A;  n, \mc P)$--proportionally thin}, for some constants $A\geq 1, \epsilon \geq 0$ and $n\in \N $, if the $\eps \dist (\gamma^-, \gamma^+)$--neighbourhood of every $(n, \mc P)$--polygonal line $\pgot$ from $\gamma^+$ to $\gamma^-$ with $\length{\pgot} \leq A \dist (\gamma^-, \gamma^+)$ contains $\gamma$.

A quasi-geodesic $\gamma$ is \emph{$(R; \epsilon,A; n, \mc P)$--weakly polygonally Morse} if each subsegment $\gamma'$ of $\gamma$ with $\abs{\gamma'}\geq R$ is $(\epsilon, A;  n, \mc P)$--proportionally thin. 
\end{definition}

The above definition is only meaningful if the constant $\epsilon$ is small enough and $\abs{\gamma} \geq R$. For instance, any geodesic $\gamma$ is $(R;\frac{1}{2}, A; n, \mc P)$--weakly polygonally Morse and any quasi-geodesic $\gamma$ is $(R; \epsilon, A; n, \mc P)$--weakly polygonally Morse if $R > \abs{\gamma}$. 

Below we show that being weakly polygonally Morse is indeed a weak notion of negative curvature, as it follows from being Morse.

\begin{lemma}[Morse implies weakly polygonally Morse]\label{lemma:localmltg-implies-path-avoidence}
Let $X$ be a geodesic metric space with a $(\lao, \kao)$--quasi-geodesic path system $\mc{P}$. Then for all constants $\la ,\ka, n, A, \epsilon$ and Morse gauges $M$, there exists $R$ so that any $(\la, \ka)$--quasi-geodesic $\gamma$ which is $M$--Morse is $(R; \epsilon,A; n, \mc P)$--weakly polygonally Morse.
\end{lemma}
\begin{proof} We argue for a contradiction and assume that for $R$ large enough, to be determined later, there exists a $(\la, \ka)$--quasi-geodesic $\gamma$ that is $M$--Morse, but not $(R; \epsilon, A; n, \mc P)$--weakly polygonally Morse. That is, $\gamma$ contains a sub-path $\gamma_1$ with $\abs{\gamma_1}\geq R$ such that for $D = d(\gamma_1^-, \gamma_1^+)$ there exists a $(n, \mc P)$--polygonal line $\qgot$ with $\norm{\qgot}\leq A D$ from $\gamma_1^-$ to $\gamma_1^+$ and a point $x\in \gamma_1$ such that $\qgot$ is disjoint from a ball $B(x,\epsilon D)$.

Let $\gamma_1[s, t]$ be a maximal subpath of $\gamma_1$ which contains $x$ and is contained in $B(x, \epsilon D/2)$. Due to the running assumption of continuity of quasi-geodesics, $d(\gamma(s), x) = d(\gamma(t), x) ) =\epsilon D/2$. The path $\qgot' = \suf{\gamma_1}{\gamma_1(t)}\ast \qgot \ast \pref{\gamma_1}{\gamma_1(s)}$ satisfies $\norm{\qgot'}\leq\norm{\qgot}+\norm{\gamma_1}\leq AD + D_{\mc P}(D)$, while $d(\gamma_1(s), \gamma_1(t))\geq D_{\mc P}^{-2}(\epsilon D/2)$. Thus, if $R$ is large enough, then there exists a constant $A'$ depending only on $\kao, \lao, \ka, \la, \epsilon$ and $A$ such that $\norm{\qgot'}\leq A'd(\gamma_1(t), \gamma_1(s))$. 

However, the quasi-geodesic $\gamma_2 = \gamma_1[s, t]$ is $M$--Morse, therefore, according to \cite[Theorem 1.2]{aougabdurhamtaylor:pulling} (where $M$--Morse quasi-geodesics are called $M$--stable), it is $K$--recurrent, for some recurrence function $K$ that can be bounded from above only in terms of $M$. In particular, there exists a point $a\in \qgot'$ at distance at most $K(A')$ from some point `in the mid portion of $\gamma_2$`, that is from some point $y\in \gamma_2$ with $d(y, \gamma_1(s)\cup \gamma_1(t)) \geq \frac{1}{3}d(\gamma_1(s), \gamma_1(t))$. For $R$ large enough, $K(A') < D\epsilon/2$. Thus, $a$ cannot lie on $\qgot$ -- otherwise, by the triangle inequality, its distance to $x$ would be less than $D\epsilon$, a contradiction. Hence $a$ has to lie on $\gamma_1$ and either $\gamma_1(s)$ or $\gamma_1(t)$ has to lie between $a$ and $y$. But then $K(A') \geq d(a, y) \geq D_{\mc P}^{-1}\left(\frac{D\epsilon}{6} \right)$, a contradiction for large enough $R$.
\end{proof}

\subsection{Bounded, avoidant and navigable path systems}\label{sect:bded-avoidant-navigable}

The goal of this section is to introduce a framework which covers important known classes of spaces such as injective metric spaces and median spaces (see Section~\ref{sec:examples} for more). The unifying notion we introduce will be the \emph{navigable path system}. Roughly, a path system is \emph{navigable} if a polygonal line avoiding a ball can be `tightened' to a polygonal line avoiding a similar ball and whose length is comparable to the size of the ball.

To assist in the proof of the property of navigability, we introduce the notions of \emph{avoidant and bounded} path systems, defined by conditions easier to verify. We prove in Section \ref{sec:finding_navigable} that a bounded and avoidant path system is navigable. Sometimes, it is sufficient to verify the above properties only for a subset of a path system, so we formulate our definitions in terms of heaps of paths, as defined in Section \ref{sect:basicspathsyst}. 

\begin{definition}[(Sub-)bounded]\label{defn:bounded_path_system}
    A heap of paths $\mc Q$ is called $\kao$--bounded if for every special path $\spath \in \mc Q$ with endpoints $x$ and $y$, and for every $y'$ with $d(y, y')\leq 1$ there exists a special path $\spath'\in \mc Q$ from $x$ to $y'$ with 
    \begin{align*}
        d_{\mathrm{Haus}}(\spath, \spath')\leq \kao.
    \end{align*}
    A $(\kao, \lao)$--quasi-geodesic path system $\mc P$ is called \emph{$\kao$--sub-bounded} if there is a heap of paths $\mc Q\subset \mc P$ which is $\kao$--bounded and such that for every pair of points $x, y\in X$ there exists a path in $\mc Q$ joining $x$ and $y$. We call such a heap of paths $\mc Q$ a \emph{$\kao$--bounded support of~$\mc P$}.
\end{definition}

The main differences by comparison with the (bi)combing case are that we allow for several special paths to join the same two points, and in the boundedness condition we only require that a single special path joining $x$ and $y'$ be close to a given path joining $x$ and $y$. 

\medskip

We continue with the definition of \emph{avoidant heap of paths}. Morally, this says that if the space is locally looking one-ended, then you can avoid a point efficiently. 

Recall that, given a heap of paths $\mc Q$, we denote by $\overline{\mc Q}$ its \emph{consistent closure}, that is, the set containing $\mc Q$ and all of its subsegments (Section \ref{sect:basicspathsyst}).

\begin{definition}[Avoidant heap of paths]\label{defn:avoidant}
    Let $C> 1$ be a constant and let $k\geq 1$ be an integer. A heap of paths $\mc Q$ in a geodesic metric space $X$ is called \emph{$(C,k)$--avoidant} if for all $R\geq C/2$ the following holds. If for $i\in \{1, 2\}$, $y, z_i, m\in X$ are points with $d(m, z_i)\leq 4R$ and $\spath_i\in \mc Q$ are special paths from $z_i$ to $y$ such that
    \begin{align*}
        d(\spath_i, m)>R,
    \end{align*}
    then there exists a $(k, \overline{\mc Q})$--polygonal line $\pgot$ from $z_1$ to $z_2$ such that 
    \begin{align*}
        d(m, \pgot) > 2R/C \qquad \text{and} \qquad \norm{\pgot}\leq CR.
    \end{align*}
    A heap of paths is called \emph{avoidant} if it is $(C, k)$--avoidant for some $C$ and $k$.
\end{definition}

\begin{definition}[Navigable heap of paths]\label{defn:navigable} 
    Let $C\geq 1$ be a constant and $k\geq 1$ an integer. A heap of paths $\mc P$ in a geodesic metric space $X$ is \emph{$(C, k)$--navigable} if the following holds. For any $R\geq C$, point $m\in X$ and $(n,\mc Q)$--polygonal line $\alpha$ with $d(\alpha, m)\geq R$, $d(\alpha^-, m)\leq 2R$ and $d(\alpha^+, m)\leq 2R$, there exists a $(kn,\overline{\mc Q})$-polygonal line $\pgot$ which has the same endpoints as $\alpha$, satisfies $\norm{\pgot}\leq CnR$ and avoids $B(m, R/C)$.
\end{definition}

As mentioned, we will show in Section~\ref{sec:finding_navigable}, that a heap of paths that is bounded and avoidant is also navigable. 

\subsection{Projections}\label{sec:projections}
We define projection points from a special path onto another special path. We show that if we project onto a weakly polygonally Morse special path, then the two paths diverge linearly after the projection point. This generalizes the similar result proved for Hadamard spaces \cite[Proposition 4.5]{KKL:QI}

\begin{definition}[Path projections]
    Let $X$ be a geodesic metric space equipped with a path system $\mc P$ and let $\gamma, \spath\in \mc P$ be special paths with $\spath^+\in \gamma$. 
    We call the first point $\spath(t)$ on $\spath$ in the $R$--neighbourhood of $\gamma$ the \emph{$R$--upper projection point of $\spath$ onto $\gamma$}. We call any point on $\gamma$ in the $R$--neighbourhood of the $R$--upper projection point $\spath(t)$ an \emph{$R$--lower projection point of $\spath$ onto $\gamma$}. 
\end{definition}

In the case where $\mc P$ is the set of all geodesics, there is always at least one special path $\spath^+$ from $x$ to $\gamma$ which is `orthogonal' to $\gamma$, in the sense that $\spath^+$ and any $R$--lower projection point of $\spath$ onto $\gamma$ have distance at most $2R$ from each other and $\spath$ diverges linearly from $\gamma$. 

Due to the uncontrolled behaviour of quasi-geodesics (think for example of the log-spiral quasi-geodesics in the plane) any analogue we can prove in the general case necessarily has to be weaker. We will show that if $\gamma$ is weakly-polygonally Morse, then for any point $x\in X$ there exists an `almost orthogonal' special path $\spath$ from $x$ onto $\gamma$ (see definition below) and after the projection point, the special path $\spath$ diverges linearly from~$\gamma$.

\begin{definition}[Almost orthogonal]
     Let $X$ be a geodesic metric space equipped with a path system $\mc P$. Let $\gamma, \spath\in \mc P$ be special paths with $\spath^+\in \gamma$ and let $\gamma(t)$ be an $R$--lower projection point of $\spath$ onto $\gamma$. We say that $\spath$ is \emph{$C$--almost $R$--orthogonal to $\gamma$} if $d(\gamma(t), \spath^+)\leq d(\spath^-, \spath^+)/C +4D_{\mc P}(R+1)$.
\end{definition}

We now show that, after exiting a small neighbourhood, special paths diverge linearly from weakly polygonally Morse special paths. 

\begin{proposition}\label{prop:path-projections-are-well-defined}
    Let $X$ be a geodesic metric space equipped with an undirected path system $\mc P$. There exist constants $A\geq 1$, $\eps > 0$ such that for all $R\geq D_{\mc P}(1)$ the following holds. Let $\gamma \in \mc P$ be an $(R; \eps, A; 3, \mc P)$--weakly polygonally Morse special path. Let $\spath\in \mc P$ be a special path with $d(\spath^+, \gamma(t_0)) = R$ for some $t_0$ and $d(\spath, \gamma)\geq R$. Then 
    \begin{align*}
        d(\spath, \gamma(t)) > \eps \abs{t_0 - t},
    \end{align*}
    for all $t$ in the domain of $\gamma$. 
\end{proposition}

\begin{proof}
    Assume that the constant $A$ is large enough, while $1\gg\theta \gg \eps$ are small enough, to be determined later. Let $\gamma, \spath, t_0$ be as in the statement. Assume by contradiction that there exists at least one parameter $t$ with 
    \begin{align}\label{eq:lin}
         d(\spath, \gamma(t)) \leq \theta \abs{t_0 - t} \leq \eps \abs{t_0 - t}.
    \end{align}
    Define $t_1$ to be such a parameter minimizing $\abs{t_0 - t_1}$, and denote $d(\gamma(t_0), \gamma(t_1))$ by $\ell$. Let $\spath(s_1)$ be a point on $\spath$ closest to $\gamma(t_1)$. Further, let $\alpha_0, \alpha_1\in \mc P$ be special paths from $\spath^+$ to $\gamma(t_0)$, respectively from $\gamma(t_1)$ to $\spath(s_1)$. This is depicted in Figure~\ref{fig:projections}. Finally, define a $(3, \mc P)$--polygonal line as $\alpha = \alpha_1\ast \psub{\spath}{\spath(s_1), \spath^+}\ast\alpha_0$.

    Observe the following.
    
    \begin{figure}
        \centering
        \includegraphics[width=0.5\linewidth]{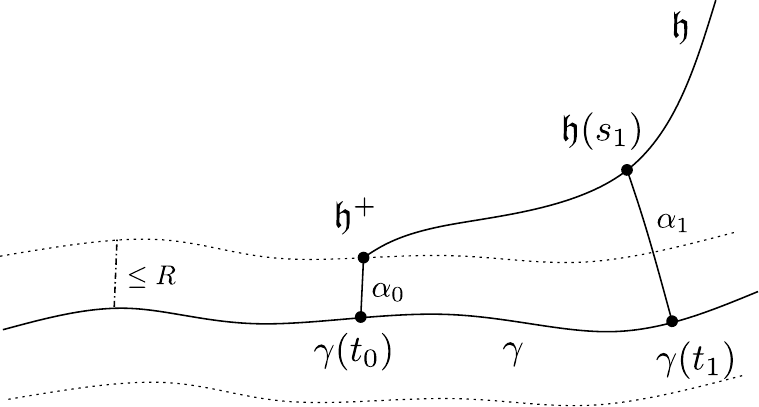}
        \caption{The points $\gamma(t_0)$ and $\spath(s_0)$ are the $R$--lower and $R$--upper projection points of $\spath$ onto $\gamma$, the point $\gamma(t_1)$ falsifies \eqref{eq:lin}.}
        \label{fig:projections}
    \end{figure}

    \smallskip
    {\sc{Property 1.:}} $\abs{t_1 - t_0}\geq R/\theta \geq R$. This follows from \eqref{eq:lin} combined with $d(\spath, \gamma)\geq R$. In particular, since $R\geq D_{\mc P}(1)\geq 1$, we obtain $\abs{t_0 - t_1}\geq 1/\theta$ and $\ell\geq 1$.
    \smallskip

    \smallskip

    {\sc{Property 2.:}} $\norm{\alpha}\leq A\ell$. Property~1 yields that $d(\alpha_i^-, \alpha_i^+)\leq \theta \abs{t_1 - t_0}$ for $i = 0, 1$.We have $d(\alpha_i^-, \alpha_i^+)\leq \theta\abs{t_1-t_0}\leq \abs{t_1 -t_0}$ for $i = 0, 1$. Hence $d(\spath^+, \spath(s_1))\leq \ell + 2\abs{t_1-t_0}$. Since all quasi-geodesics are improved and $\ell\geq 1$, we obtain $\norm{\alpha}\leq A\ell$ for $A$ large enough compared to the quasi-geodesic constants of $\mc P$. 
    \medskip

    Since $\gamma$ is $(R; \eps, A; n, \mc P)$--weakly polygonally Morse, the above properties imply that $d(m, \alpha) \leq \eps \ell$ for $m = \gamma((t_1 +t_0)/2)$.
    
    \begin{claim}\label{claim:at-least-one-far}
        If $\eps, \theta$ are small enough, then $d(m, \alpha_i) > \eps \ell$ for $i = 0, 1$.
    \end{claim}
    
    \textit{Proof of Claim \ref{claim:at-least-one-far}.}
    Since $\gamma$ is a special path and $\ell \geq 1$, we have $d(m, \gamma(t_0))\geq \abs{t_0 - t_1}/C - C\geq \ell/C' - C'\geq \ell/C''$ for large enough $C, C'$ and $C''$ depending only on the quasi-geodesic constants of $\mc P$. Moreover, for $i = 0, 1$ we have $\diam(\alpha_i)\leq K\theta\abs{t_0 - t_1}+K\leq 2K\theta\abs{t_0-t_1} \leq \theta(K'\ell+K')\leq \theta2K'\ell$ for large enough $K, K'$ depending only on the quasi-geodesic constants of $\mc P$. Thus, for $\eps < \frac{1}{3C''}$ and $\theta < \frac{1}{3C''K'}$, we have by the triangle inequality that $d(\alpha_i, m)\geq \ell/C' -\theta K'\ell > \eps \ell$.
    \hfill$\blacksquare$

    \smallskip
    
    Consequently, we obtain
    \begin{align}\label{eq:i1}
        d(m, \psub{\spath}{\spath(s_1), \spath^+})\leq \eps \ell.
    \end{align}
    However, by \eqref{eq:lin} (and using that $\abs{t_0 - t_1}\geq 1/\theta$) we obtain for $1\gg \eps \gg \theta$ that 
    \begin{align}\label{eq:i2}
        d(m,\spath)\geq \theta \frac{\abs{t_1 - t_0}}{2} > \eps \ell,
    \end{align}
    a contradiction to \eqref{eq:i1}.
\end{proof}

\begin{corollary}\label{cor:close-before-endpoints}
    Let $X$ be a geodesic metric space equipped with an undirected path system $\mc P$. Let $\eps, A$ be the constants from Proposition~\ref{prop:path-projections-are-well-defined} and let $R\geq D_{\mc P}(1)$. There exists a constant $K\geq 1$ such that the following holds. Let $\gamma\in \mc P$ be a $(R; \eps, A; 3, \mc P)$--weakly polygonally Morse special path. Let $\spath\in \mc P$ be a special path with $\spath^+ = \gamma(t)$ and let $x$ and $\gamma(t')$ be an $R$--upper and $R$--lower projection point of $\spath$ onto $\gamma$. Then $\suf{\spath}{x}$ is contained in the $K\cdot R$--neighbourhood of $\isub{\gamma}{t, t'}$.
\end{corollary}

\begin{proof}
    It suffices to show that any subpath $\spath'\subset \spath$ with $d(\spath', \gamma)\geq R$ and whose endpoints are at distance $R$ from $\gamma$ is contained in the $K'R$--neighbourhood of $\gamma$ for a constant $K'\geq 1$ independent of $\spath$ or $\gamma$. Then $\suf{\spath}{x}$ would be contained in the $K'R$--neighbourhood of $\gamma$ and the result would follow from \cite[Lemma~4.1]{DrutuSprianoZbinden:schism}.

    Let $\spath'$ be defined as above with endpoints $x_0$ and $x_1$ at distance $R$ from $\gamma(t_0)$ and $\gamma(t_1)$ respectively. By Proposition~\ref{prop:path-projections-are-well-defined}, we have $ R = d(x_1, \gamma(t_1)) > \eps \abs{t_1 - t_0}$, providing an upper bound on $\abs{t_1-t_0}$ linear in $R$, hence in $\diam(\spath')$. 
\end{proof}

\begin{lemma}\label{lem:almost-ortho-projection}
    Let $X$ be a geodesic metric space equipped with an undirected path system $\mc P$. For every constant $C > 0$ there exist constants $\epsilon, A$ such that for all $R\geq D_{\mc P}(1)$ the following holds. Let $\gamma\in \mc P$ be a $(R; \epsilon, A; 3, \mc P)$--weakly polygonally Morse special path and let $x\in X$. There exists a special path $\spath\in \mc P$ from $x$ to $\gamma$ which is $C$--almost $R$--orthogonal to $\gamma$.
\end{lemma}

\begin{proof}
    Let $\eps_0, A_0$ be the constants from Proposition~\ref{prop:path-projections-are-well-defined}, and let $\eps \ll \eps_0$ and $A\gg A_0$ be constants to be determined later. Let $\gamma : [0, T]\to X$ be as in the statement and let $x\in X$. For $0\leq i \leq T$, denote by $\spath_i\in \mc P$ a special path from $x$ to $\gamma(i)$. Moreover, let $u_i$ and $y_i = \gamma(s_i)$ be the $R$--upper and a $R$--lower projection points of $\spath_i$ onto $\gamma$. We say that $\spath_i$ is \emph{left leaning} if $s_i\leq i$ and we call it \emph{right leaning} otherwise.

    \begin{claim}\label{claim:switch-implies-orthogonal}
        If $\spath_i$ is left leaning and $\spath_{i+1}$ is right leaning or vice versa, then $\spath_i$ is $C$--almost $R$--orthogonal.
    \end{claim}
    
    \textit{Proof of Claim \ref{claim:switch-implies-orthogonal}.}
        We assume that $\spath_{i+1}$ is right leaning and $\spath_{i}$ is left leaning, but not $C$--almost $R$--orthogonal and show that this yields a contradiction. The case where $\spath_{i+1}$ is left leaning but not $C$--almost $R$--orthogonal is solved analogously, by simply swapping $\spath_i$ and $\spath_{i+1}$. 

        \begin{figure}
            \centering
            \includegraphics[width=0.65\linewidth]{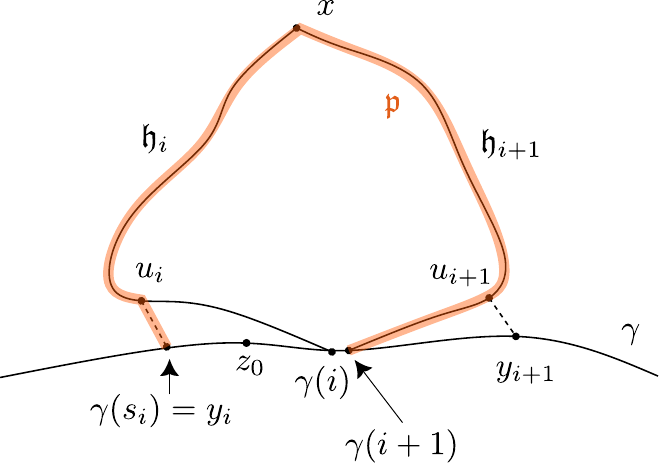}
            \caption{Proof of Claim~\ref{claim:switch-implies-orthogonal}. The path $\pgot$ has to come close to $z_0$.}
            \label{fig:left-leaning}
        \end{figure}
        
        Consider the following $(3, \mc P)$--polygonal line $\pgot = [y_i, u_i]_{\mc P} \ast \pref{\spath_i}{u_i}^{-1}\ast \spath_{i+1}$. This is depicted in Figure~\ref{fig:left-leaning}. Since $\spath_i$ is not $C$--almost $R$--orthogonal, we know that $d(\gamma(i), \gamma(s_i))\geq d(x, \gamma(i))/C + 4D_{\mc P}(R+1)$, implying that $\abs{s_i -(i+1)}\geq R$. Thus, for $A$ large enough compared to $C$ and the quasi-geodesic constants of $\mc P$, we have that $\norm{\pgot}\leq A \ell$, where $\ell = d(\gamma(s_i), \gamma(i+1))$. Hence, we can use that $\gamma$ is $(R; \epsilon, A; 3, \mc P)$--weakly polygonally Morse to obtain $d(z, \pgot)\leq \eps \ell$ for all $z\in \isub{\gamma}{s_i, i+1}$. Let $z_0 = \gamma(t)$ be a point on $\isub{\gamma}{s_i, i+1}$ at distance at least $\ell/2$ from both of $\gamma(s_i)$ and $\gamma(i+1)$.

        We want to find a contradiction by showing that $d(z_0, \pgot) > \eps \ell$. First observe that $\ell\geq 4D_{\mc P}(R)$. Thus, by the triangle inequality and as long as $\eps < 1/4$, we have that $d([y_i, u_i]_{\mc P}, z_0) > \ell/2 - D_{\mc P}(R)\geq  \eps \ell$. Moreover, since $\eps\ll \eps_0$ and $A\gg A_0$, we have $d(\pref{\spath_i}{u_i}, z_0)\geq \eps_0\abs{t - s_i} > \eps\ell$. Similarly, $d(\pref{\spath_{i+1}}{u_{i+1}}, z_0) > \eps \ell$. Finally, Corollary~\ref{cor:close-before-endpoints} shows that $\suf{\spath_{i+1}}{u_{i+1}}$ is contained in the $KR$--neighbourhood of $\isub{\gamma}{i+1, s_{i+1}}$. For $\eps$ small enough compared to the quasi-geodesic constants of $\mc P$, this also yields $d(z_0, \suf{\spath_{i+1}}{u_{i+1}}) > \eps \ell$. Hence $d(z_0, \pgot) > \eps\ell$, a contradiction.
    \hfill $\blacksquare$

    If there exists at least one left leaning and one right leaning $\spath_n$, then there have to be consecutive ones, in which case Claim~\ref{claim:switch-implies-orthogonal} shows that there exists a $C$--almost $R$--orthogonal one. 

    The only other cases to consider are if for all integers $i$, the $\spath_i$ are left leaning, or if they are all right leaning. Assume they are all left leaning. Consequently, $\spath_0$ is left leaning, implying that $s_0 =0$. Hence $\spath_0$ is $C$--almost $R$--orthogonal. Analogously, if they are all right leaning then $\spath_{\floor{T}}$ is $C$--almost $R$--orthogonal. 
\end{proof}

\section{Globalization theorems. First theorem on linear divergence}\label{sec:globalizations}

This section contains our first set of consequences of properties of path systems.  The first three subsections are concerned with deducing from local date whether a path is a Morse quasi-geodesic. The spaces where such a phenomenon occurs will satisfy the MLTG property, which has many significant implications, described in the appendix of  \cite{DrutuSprianoZbinden:schism}. In fact, we show a stronger property than the MLTG property, namely that for a quasi-geodesic $\qgot$ the property of being Morse can be verified locally by checking that all quadrilaterals with one side a uniformly short subsegment of $\qgot$ and the other three sides uniformly short special paths are thin. It is quite remarkable that one can check Morseness only on short quadrilaterals, instead of having to consider all possible quasi-geodesics with endpoints on the given quasi-geodesic. In particular, if the space and the quasi-geodesic $\qgot$ are regular enough, this provides an algorithmic way to determine whether a quasi-geodesic is Morse.

The last subsection is concerned with the stability of asymptotic properties under change of scale, for spaces endowed with navigable path systems. Specifically, we show that if the divergence is linear for a certain sequence, then it is always linear, and similarly that if the asymptotic cones do not have global cut-points for a scaling sequence, then all asymptotic cones are without cut-points.

\subsection{Locally weakly polygonally Morse special paths are $\mc P$--contracting}\label{sec:wpm-implies-morse}

In this section, we prove Proposition~\ref{prop:local_poly_P_contracting}, which states that in a geodesic metric space $X$ equipped with a navigable path system $\mc P$, special paths that are locally weakly polygonally Morse are $\mc P$--contracting. As a consequence, we obtain that such a space $X$ is MLTG (see Corollary~\ref{cor:MLTG}) and that all Morse special paths are $\mc P$--contracting (see Corollary~\ref{cor:morse-implies-mccontracting}). 

\begin{proposition}[Locally weakly polygonally Morse implies globally $\mc P$--contracting]\label{prop:local_poly_P_contracting}
    Let $X$ be a geodesic metric space equipped with a $(C_0, k)$--navigable undirected quasi-geodesic path system $\mc P$. There are constants $A, \epsilon$ so that for every $R\geq 0$ there exists $L, C\geq 0$ satisfying the following. Let $\gamma$ be a $(\la, \ka)$--quasi-geodesic that is $L$--locally $(R; \epsilon, A; 7k, \mc{P})$--weakly polygonally Morse and such that every subpath of $\gamma$ has Hausdorff distance at most $C_0$ from a special path between its endpoints. Then $\gamma$ is $\mc P$--contracting with constant $C$. 
\end{proposition}

In the cases where we want to apply Proposition~\ref{prop:local_poly_P_contracting} we usually get the condition that subpaths of $\gamma$ have to be close to special paths for free. Namely, we either apply it if $\gamma$ is a special path, or if the path system is sub-bounded. In the latter case, the condition follows from Proposition~\ref{prop:globalisation_QG}.

\begin{proof}
    Let $C\gg D \gg R$ be large enough constants to be determined later and let $\gamma$ be as in the statement. We construct a projection $\pi : X\to \gamma$ and show that $\gamma$ is $\mc P$--contracting with constant $C$.

    \smallskip
    
    \textbf{Choice of projection.} For $x\in X$, choose any special path $\spath_x\in \mc P$ from $x$ to $\gamma$ and define $\tau(x)$ and $\pi(x)$ as upper and $D$--lower projection points of $\spath_x$ onto $\gamma$.

\smallskip

        If $x\in \gamma$, then $x = \tau(x)$, implying that Condition \eqref{5.2.1} of Definition~\ref{def:contracting} of being $\mc P$--contracting holds. 

        It remains to show Condition~\eqref{5.2.2}. Assume it does not hold, that is, assume that there are points $x,y \in X$ and a special path $\beta\in \mc P$ from $x$ to $y$ such that $\dist (\pi(x), \pi(y))\geq C$ and such that $\dist (\beta, \pi(x))\geq C$.
        
        \smallskip
        
        Let $z$ be a point on $\gamma' = \psub{\gamma}{\pi(y), \pi(x)}$ with $\dist (z, \pi(x)) = C/4$. In addition, let $z_2, z_1$ be the first and respectively last point on $\gamma'$ in the $D/2$--neighbourhood of $z$. Define the $(7, \mc P)$--polygonal line $\alpha = \alpha_1\ast \cdots \ast \alpha_7$ as follows (see Figure~\ref{fig:proof-contracting}). 
        
        \begin{figure}
            \centering
            \includegraphics[width=\linewidth]{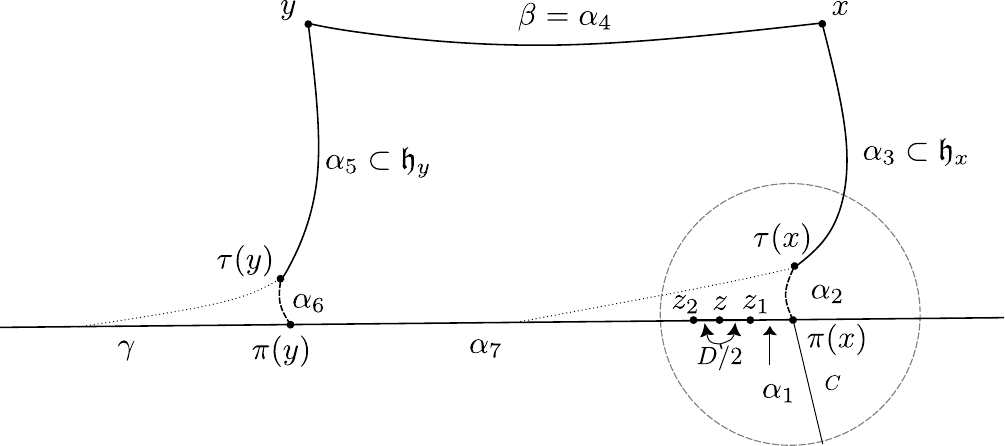}
            \caption{The path $\alpha$ does not intersect the $D$--ball around $z$.}
            \label{fig:proof-contracting}
        \end{figure}
        \begin{align*}
            \alpha_1 &= [z_1, \pi(x)]_{\mc P},&
            \alpha_2 = [\pi(x), \tau(x)]_{\mc P},\quad
            \alpha_3 &= \pref{\spath_x}{\tau(x)}^{-1},\quad
            \alpha_4 = \beta,\\
             \alpha_5 &= \pref{\spath_y}{\tau(y)},&
            \alpha_6  = [\tau(y),\pi(y)]_{\mc P},\quad
            \alpha_7 &= [\pi(y), z_2]_{\mc P}.\quad&
        \end{align*}
        
        By setting $D\geq 4C_0$, we can assume that $\alpha_1$ and $\alpha_7$ have Hausdorff distance at most $D/4$ from $\suf{\gamma'}{z_1}$ and $\pref{\gamma'}{z_2}$, respectively.
        
        \begin{claim}\label{claim:alpha-far-from-z}
            We have that $d(\alpha, z) > D/4$.
        \end{claim}

        \textit{Proof of Claim \ref{claim:alpha-far-from-z}.}
         For $i = 3, 5$, we have $\dist (z, \alpha_i)\geq D$ by the definition of the projection $\pi$. For $i = 4$, $d(\alpha_i, z)\geq d(\beta, \pi(x)) - d(\pi(x), z) > D/4$. For $i=1$ we have $d(\alpha_i, z)\geq d(\pref{\gamma'}{z_1}, z) - C_0 > D/4$. A similar argument works for $i = 7$. Lastly, for $i = 2, 6$ it follows from $\diam(\alpha_i)\leq D_{\mc P}(D)$, $C\gg D$ and the triangle inequality.
            \hfill $\blacksquare$
        \smallskip
        
        Using navigability on $\alpha$ for $D/4$ and $z$, we obtain a $7k$--polygonal line $\pgot$ from $z_1$ to $z_2$ with $\norm{\pgot}\leq 7kC_0D$ and $d(\pgot, z) \geq D/(4C_0)$.
        
        For $A\geq 7kC_0D$ and $L \gg D \gg R$, we obtain $\length{\pgot} \leq A d(z_1, z_2)$ and $R\leq \dlength{\psub{\gamma}{z_1, z_2}}\leq L$. Using that $\gamma$ is $L$--locally $(R; \eps, A; 7k, \mc P)$--weakly-polygonally Morse, we obtain $d(z, \pgot)\leq \eps D$, a contradiction for $\eps < \frac{1}{4C_0}$. 
\end{proof}

\begin{corollary}\label{cor:morse-implies-mccontracting}
    Let $X$ be a geodesic metric space equipped with a navigable undirected path system $\mc P$. For each Morse gauge $M$ there exist constants $C, L$ such that any $L$--locally $M$--Morse special path $\spath\in \mc P$ is $\mc P$--contracting with constant $C$.
\end{corollary}

\begin{proof}
Lemma~\ref{lem:combing-avoidance} states that an $M$--Morse special path $\gamma$ is $(R; \eps, A; n, \mc P)$--weakly polygonally Morse for any $\eps, n, A$ as long as $R= R(\eps, n, A, M)$ is suitably large. Proposition~\ref{prop:local_poly_P_contracting} allows to conclude the proof.
\end{proof}

\begin{corollary}\label{cor:MLTG}
    Let $X$ be a geodesic metric space equipped with a navigable undirected path system $\mc P$. Then $X$ is Morse local-to-global. 
\end{corollary}

\begin{proof}
    By \cite[Theorem~3.12 and Remark~3.14]{AbbottZbinden:sigma-compact}, it suffices to prove that for any Morse gauge $M$, there exists a Morse gauge $M'$ and a scale $L$ such that any special path $\gamma \in \mc P$ which is $L$--locally $M$--Morse is $M'$--Morse. Since any $\mc P$--contracting path with constant $C$ is $N$--Morse for some Morse gauge $N$ depending only on $C$ (see Theorem~\ref{theorem:contracting-implies-morse}), Corollary~\ref{cor:morse-implies-mccontracting} allows to finish the argument.
\end{proof}

\subsection{Upgrading from $(n, \mc P)$--polygonal lines to $(3, \mc P)$--polygonal lines}\label{sec:3-to-n}

In Section~\ref{sec:wpm-implies-morse}, we have seen that if $\mc P$ is an undirected $(C_0, k)$--navigable path system, then to check whether a special path $\gamma$ is $\mc P$--contracting, it suffices to show that all `short' $(7k, \mc P)$--polygonal lines with endpoints on $\gamma$ are `thin'. In this section, we show that, in fact, it suffices to check thinness for $(3, \mc P)$--polygonal lines. This is optimal: in Example~\ref{ex-e-not-2} we construct a sequence of special paths of an undirected navigable path system which are $(0; 0, A, 2, \mc P)$--weakly polygonally Morse for all $A$, but not uniformly Morse.

\begin{proposition}\label{prop:3-implies-n}
    Let $X$ be a geodesic metric space endowed with $\mc P$, an undirected path system, and let $D$ be a constant. For every $n\geq 3, \epsilon'> 0$ and  $A'\geq 1$  there exists $\epsilon > 0$ and  $A\geq 1$ such that for all $R\geq 0$ there exists $R'\geq 0$ satisfying the following. Let $\gamma\in \mc P$ be a $(R; \epsilon, A; 3, \mc{P})$--weakly polygonally Morse quasi-geodesic such that every subpath of $\gamma$ has Hausdorff distance at most $D$ from a special path between its endpoints. Then $\gamma$ is $(R'; \epsilon', A'; n, \mc{P})$--weakly polygonally Morse.
\end{proposition}

\begin{proof}[Proof of Proposition~\ref{prop:3-implies-n}]

Let $\eps_0, A_0$ be the constants from Proposition~\ref{prop:path-projections-are-well-defined}. Let $C\geq 1$ be large enough compared to $\eps', A'$, to be determined later. Let $\eps\ll \eps'$ small enough and $A\gg A'$ large enough compared to $C$, to be determined later. In particular, we can assume that $\eps\leq \eps_0$ and $A\geq A_0$ are small/large enough to apply Proposition~\ref{prop:path-projections-are-well-defined} and Lemma~\ref{lem:almost-ortho-projection} for the constant $C$. Further, let $R\geq 0$ and let $R'\gg R$ to be determined later. Assume that $\gamma$ is $(R; \eps, A; 3, \mc{P})$--weakly polygonally Morse but not $(R', \eps', A'; n, \mc{P})$--weakly polygonally Morse. In particular, there exists an $(n, \mc{P})$--polygonal line $\pgot= \pgot_1\ast \ldots \ast \pgot_n $ with endpoints on $\gamma$, with $\norm{\pgot}\leq A'\ell'$, and there exists a point $z'' = \gamma(s'')$ on $\gamma$ satisfying $d(z'', \pgot) >  \eps' \ell'$, where $\ell' = d(\pgot^-, \pgot^+)$. Let $\gamma'$ be a special path between $\pgot^-$ and $\pgot^+$. By choosing $R'$ large enough compared to $D, \eps'$, we know that there is a point $z' = \gamma'(s')$ on $\gamma'$ with $d(z', \pgot) > \eps'\ell' - D > \eps'\ell'/2$. 

For $1\leq i \leq n$, let $\spath_i\in \mc P$ be a $C$--almost $R$--orthogonal special path from $\pgot_i^+$ to $\gamma$. Such a path exists by Lemma~\ref{lem:almost-ortho-projection}. Denote the trivial path from $\pgot_0^-$ to itself by $\spath_0$ and denote $\spath_i^+$ by $x_i = \gamma'(t_i)$ for all $i$.

Recall that $C_{\mc P}$ is a constant such that $D_{\mc P}(r)\leq C_{\mc P}r+C_{\mc P}$, as outlined in Remark~\ref{remark:definition-of-dp}.

\begin{claim}\label{claim:good-i-exists}
    There exists an index $1\leq i \leq n$ and a point  $z$ on $\isub{\gamma'}{t_{i-1}, t_i}$ such that
        \begin{enumerate}[label = \roman*)]
            \item  $d(\gamma'(t_{i-1}), \gamma'(t_{i}))\geq \frac{\eps' \ell'}{16C_{\mc P}n}\geq D_{\mc P}(R)$,\label{p:far-ends}
            \item  $d(z, \gamma'(t_{i}))\geq \frac{\eps'\ell'}{16C_\mc Pn}$ and  $d(z, \gamma'(t_{i-1}))\geq \frac{\eps'\ell'}{16C_\mc Pn}$,\label{p:to-endps}
            \item $d(z, z')\leq \frac{\eps'\ell'}{4}$.\label{p:close-to-z'}
        \end{enumerate}
\end{claim}

\textit{Proof of Claim \ref{claim:good-i-exists}.}
Let $t'\geq s'$ be maximal such that $\isub{\gamma'}{s', t'}$ is contained in the $\frac{\eps' \ell'}{4}$ neighbourhood of $z' = \gamma'(s')$. Further, define $s' = r_0\leq r_1\leq \ldots \leq r_{2n} = t'$ such that $d(\gamma'(r_j), \gamma'(r_{j+1}))\geq \frac{\eps'\ell'}{8n}$. By the pigeonhole principle, there has to exist $i$ and $k$ such that $t_{i-1}\leq r_{k-1}\leq r_k\leq t_i$. Let $s = \frac{r_{k-1}+r_k}{2}$ and $z = \gamma'(s)$. With this definition, property~\ref{p:close-to-z'} holds by construction. Using $R'\gg 1$, we obtain $d(\gamma'(t_{i-1}), z)\geq D_{\mc P}^{-1}\left(\frac{\eps'\ell'}{8n}\right)\geq \frac{\eps'\ell'}{16C_{\mc P}n}$, yielding the second part of property~\ref{p:to-endps}. The first part of property~\ref{p:to-endps} is obtained analogously and for property~\ref{p:far-ends}  we use $R'\gg R$ which gives $\ell' \gg R$.
\hfill$\blacksquare$
\smallskip

Now let $i$ be an index satisfying the conditions of Claim~\ref{claim:good-i-exists} for some $z = \gamma'(s)$ and let $\ell = d(\gamma'(t_{i-1}), \gamma'(t_i))$. Also observe that property~\ref{p:to-endps} implies $d(z, \gamma'(t_i)\cup \gamma'(t_{i-1}))\geq \ell /Q$ for some constant $Q\geq 1$ depending only on $\eps', n$ and the quasi-geodesic constants of $\mc P$. We will choose $\eps \ll \frac{1}{2Q}$. We want to argue that the $(3, \mc P)$--polygonal line $\qgot = \spath_{i-1}^{-1}\ast\pgot_i\ast \spath_i$ contradicts $\gamma'$ being $(R; \eps, A; 3, \mc{P})$--weakly polygonally Morse. Claim~\ref{claim:good-i-exists}\ref{p:far-ends} yields $\abs{t_{i-1}- t_i}\geq R$.  Moreover, by simple calculations, $\norm{\pgot}\leq A'\ell'$ yields $\norm{\qgot}\leq KA'\ell'$ for some $K$ large enough depending only on the quasi-geodesic constants of $\mc P$. Combined with property~\ref{p:far-ends}, this yields $\norm{\qgot}\leq A\ell$ for $A\geq A'KC_{\mc P}16n/\eps'$. Hence, using that $\gamma'$ is $(R; \eps, A; 3, \mc P)$--weakly polygonally Morse, we obtain $d(z, \qgot)\leq \eps \ell$.

\begin{claim}\label{claim:not-sides}
    $d(z, \spath_j) > \eps \ell$ for $j = i$ and $j = i-1$.
\end{claim}

\textit{Proof of Claim \ref{claim:not-sides}.}
Let $x_j$ and $\gamma'(s_j)$ be an $R$--upper and $R$--lower projection point of $\spath_j$ onto $\gamma'$. We can write $\spath_j = \eta_1\ast \eta_2$, where $\eta_2 = \psub{\spath_j}{x_j, \spath_j^+}$. For $C$ large enough compared to $\eps', A'$, and hence large enough compared to $Q$, and $R'\gg R$ we have $\diam(\eta_2) < \frac{\ell}{3Q}$ because $\spath_j$ is $C$--almost $R$--orthogonal. Consequently, $d(\eta_2, z)\geq d(\gamma'(t_j), z) - \frac{\ell}{3Q}>\eps \ell$.

For $\eta_1$ we observe the following: By Proposition~\ref{prop:path-projections-are-well-defined}, $d(\eta_1, z')\geq \eps_0\abs{s_j - s}$. Since $d(z, \gamma'(t_j))\geq \frac{\ell}{Q}$, this implies $d(\eta_1, z') > \eps \ell$ as long as $\eps$ is small enough compared to $\eps_0, Q$ and the quasi-geodesic constants of $\mc P$.
\hfill$\blacksquare$

\smallskip

Claim~\ref{claim:not-sides} implies that $d(z, \pgot_i)\leq \eps\ell\leq \frac{\eps'\ell'}{6}$. In particular, $d(z', \pgot_i)\leq d(z', z) + d(z, \pgot_i)<\frac{\eps'\ell'}{2}$. This is a contradiction to the definition of $z'$, concluding the proof.
\end{proof}

\begin{example}\label{ex-e-not-2}
   Consider the Cayley graph of $\Z^2$ with respect to the standard generating set, and consider the bicombing consisting of geodesics that first try to go as close as possible to the $x$ axis, then move in the $x$ direction, and then finish with the remainder of the $y$ direction. Formally, given $a = (x_a, y_a)$ and $b= (x_b, y_b)$ let $z = \mathrm{sgn}(y_a) \max \{ 0, \mathrm{sgn}(y_a) y_b \}$ and define $\eta_{ab}$ as the concatenation of geodesics between $(x_a, y_a) - (x_a, z)- (x_b, z) - (x_b, y_b)$. It is a straightforward computation to show that this is a bounded and consistent (i.e. restrictions of combing lines are combing lines) combing and hence the set of all combing lines is an undirected path system, which is navigable by Proposition~\ref{prop:bicombing-is-navigable}. Note that the $x$--axis is $(R; \epsilon, A; 2, \mc{P})$-polygonally Morse in the strongest possible sense, with constants $R= 0, \epsilon= 0, A= \infty$; any 2-polygonal path based on the $x$--axis contains the restriction of the $x$--axis between its endpoints. However, it is easily seen that the $x$--axis is not Morse or $(R; \epsilon, A; 3, \mc{P})$--polygonally Morse for any $\epsilon \in \left(0,\frac{1}{2}\right)$ as long as $A \geq 3$. Indeed, the 3-polygonal path $[(0,0), (0,n)]\ast [(0,n),(n,n)] \ast [(n,n),(n,0)]$, with arbitrary $n\in {\mathbb{N}}$, does not contain in its $\epsilon$--neighbourhood the geodesic segment $\left[ (0,0),(n,0) \right]$, for any $\epsilon\in \left(0, \frac{1}{2}\right)$.
\end{example}

\subsection{Globalising the quasi-geodesic property and consequences}

\begin{proposition}[Globalization of the quasi-geodesic property]\label{prop:globalisation_QG}
    Let $X$ be a geodesic metric space equipped with an undirected $\kao$--sub-bounded $(\lao, \kao)$--quasi-geodesic path system $\mc P$ and let $(\la, \ka)$ be a quasi-geodesic pair. Then there exist $\epsilon_0, A_0$ so that for all $R\geq 0$ there are constants $L, \ka', \la'$ so that any path $\pgot$ which is $L$--locally a $(R;\epsilon_0,A_0;3, \mc P)$--weakly polygonally Morse $(\la, \ka)$--quasi-geodesic is globally a $(\la', \ka')$--quasi-geodesic.
    
    Moreover, there is $D_0= D_0(\la, \ka, \lao, \kao, R, \epsilon_0, A_0)$ so that the Hausdorff distance between $\pgot$ and some special path between its endpoints is at most $D_0$.
\end{proposition}

\begin{proof} We may assume that $\la \geq \lao$ and $\ka \geq \kao$. 
    We will show that the result holds for $\epsilon_0 = \frac{1}{\mu^6}$ and $A_0 = 8\mu^8$, where $\mu = 2 \la+\ka $. Let $\mc Q\subset \mc P$ be a $\kao$--bounded support of $\mc P$.
    We first prove that $\pgot$ is contained in the $D = \mu(R+\mu)$--neighbourhood of a special path joining its endpoints, which, by \cite[Lemma~2.8]{DrutuSprianoZbinden:schism} implies that $\pgot$ is a $(\la', \ka')$--quasi-geodesic for constants $\la', \ka'$ depending only on $D$ and $\lao, \kao, \la, \ka$. A standard argument (for instance \cite[Lemma 2.6]{DrutuSprianoZbinden:schism}) yields the desired estimate on the Hausdorff distance. 
    
    Assume by contradiction that $\pgot \colon [S, T]\to X$ is not contained in the $D$--neighbourhood of any special path between its endpoints. Boundedness and a (coarse) continuity argument implies that there exist $t\in [S, T]$ and a special path $\spath\in \mc Q$ from $\pgot(S)$ to $\pgot(t)$ whose $D$--neighbourhood does not contain $\isub{\pgot}{S,t}$ but whose $(D+\kao)$--neighbourhood contains $\isub{\pgot}{S, t}$. 
    
    Let $\pgot(s)\in \isub{\pgot}{S, t}$ be a point with $\dist (\pgot(s), \spath)> D$ and let $\mu = 2(\la + \ka)$. Let $s_1 = \max\{S, s - \mu^5 D\}$ and let $s_2= \min \{t, s+ \mu^5 D\}$. Let $x_1, x_2\in \spath$ be such that $\dist (x_i, \pgot(s_i))\leq D + \kao$, and let $\eta_1, \eta_2\in \mc P$ be special paths from $\pgot(s_1)$ to $x_1$ and from $x_2$ to $\pgot(s_2)$. If $x_i = \pgot(s_i)$, we assume that $\eta_i$ is the trivial path. The concatenation $\qgot = \eta_1\ast \psub{\spath}{x_1, x_2}\ast \eta_2$ is a $(3, \mc P)$--polygonal line from $\pgot(s_1)$ to $\pgot(s_2)$ which does not intersect the $D$--neighbourhood of $\pgot(s)$. Indeed, by the choice of $\pgot(s)$, the path $\spath$, and hence $\psub{\spath}{x_1, x_2}$, does not intersect the $D$--neighbourhood of $\pgot(s)$. Furthermore, if $\eta_i$ is non-trivial, we have $\abs{s - s_i} = \mu^5 D$. Hence, a simple calculation shows $\diam(\eta_i) <  \dist (\pgot(s_i), \pgot(s)) - D$. As $\dist (\pgot(s_1), \pgot(s_2)) \leq \mu^6D$, the choice of $\eps_0$ guarantees that $\qgot$ does not intersect the $\eps_0 \dist (\pgot(s_1), \pgot(s_2))$--neighbourhood of $\pgot(s)$. 

    Another calculation shows that $\norm{\qgot}\leq A_0 \dist (\pgot(s_1), \pgot(s_2))$. Since $d(\pgot(s), \spath) > D$, and by the choice of $D$, we have $R\leq \abs{s - s_1}\leq \abs{s_2 - s-1}$. Finally, choosing $L \geq 2\mu^5D$ ensures that $\abs{s_1 - s_2}\leq L$. Therefore, the existence of $\qgot$ yields a contradiction to $\pgot$ being $L$--locally $(R; \eps_0, A_0; 3, \mc P)$--polygonally Morse. Hence, $\pgot$ is contained in the $D$--neighbourhood of a special path between its endpoints, concluding the proof. 
\end{proof}

If we remove the boundedness assumption from Proposition~\ref{prop:globalisation_QG} the moreover part still holds, as long as the path $\pgot$ is already assumed to be a quasi-geodesic and globally $(3; \eps_0, A_0; 3, \mc P)$--weakly polygonally Morse.

\begin{lemma}\label{lem:close-to-special-paths-globally}
    Let $X$ be a geodesic metric space equipped with a path system $\mc P$ and let $(\la, \ka)$ be a quasi-geodesic pair. There exist constants $\eps$ and $A$ such that for all $R$, there exists a constant $D$ for which the following holds. Let $\gamma$ be a $(R; \eps, A; 3, \mc P)$--weakly polygonally Morse $(\la, \ka)$--quasi-geodesic. Then $d_{\mathrm{Haus}}(\gamma, [\gamma^-, \gamma^+]_{\mc P})\leq D$.
\end{lemma}

\begin{proof}
    Let $\gamma :[0, T]\to X$ be as in the statement for $\eps$ small and $A, D$ large enough to be determined later. Let $\spath = [\gamma^-, \gamma^+]_{\mc P}$. It suffices to prove that for large enough $D'$ not depending on $\gamma$, the quasi-geodesic $\gamma$ is contained in the $D'$ neighbourhood of $\spath$. A standard argument, for instance \cite[Lemma~2.6]{DrutuSprianoZbinden:schism}, then yields the desired bound on the Hausdorff distance. 

    Let $t$ be the parameter that maximizes $d(\gamma(t), \spath)$ and assume by contradiction that $\mu = d(\gamma(t), \spath) > D'$. Let $t_1\leq t \leq t_2$ be such that $\frac{\mu}{4\eps}\leq d(\gamma(t_1), \gamma(t_2))$. Further, if $d(\gamma(t), \gamma(t_1))\leq 10D_{\la, \ka}(\mu)$, then we replace $t_1$ with some $t_1'< t_1$ such that $d(\gamma(t_1'), \gamma(t)) = 10D_{\la, \ka}(\mu)$ or with $0$ if such a $t_1'\leq t_1$ does not exist. We similarly replace $t_2$. With this, for $\eps$ small enough compared to $\la $ and $ \ka$, we have $\frac{\mu}{8\eps}\leq \ell := d(\gamma(t_1), \gamma(t_2))\leq \frac{\mu}{2\eps}$. For $i = 1, 2$, let $\spath(s_i)$ be a point on $\spath$ closest to $\gamma(t_i)$. Recall that by the definition of $t$, we have $d(\gamma(t_i), \spath(s_i))\leq \mu$. Define the $(3, \mc P)$--polygonal line $\pgot$ as $[\gamma(t_1), \spath(s_1)]_{\mc P} \ast \isub{\spath}{s_1, s_2}\ast [\spath(s_2), \gamma(t_2)]$. For $A$ large enough compared to $\la, \ka$ we have that $\norm{\pgot}\leq A\ell$ and for $D'$ large enough compared to $\eps, \la, \ka$ we can use the lower bound on $\ell$ to deduce $\abs{t_1 - t_2}\geq R$. Using that $\gamma$ is $(R; \eps, A; 3, \mc P)$--weakly polygonally Morse we obtain $d(\gamma(t),\pgot)\leq \eps \ell < \mu$. However, $t_1, t_2$ were chosen specifically that $d(\gamma(t), [\gamma(t_i), \spath(s_i)]_{\mc P}) \geq \mu$ for $i = 1, 2$, implying $d(\isub{\spath}{s_1, s_2},\gamma(t)) \leq \eps \ell < \mu$, a contradiction.
\end{proof}

\begin{theorem}\label{thm:combined-wpm-results}
    Let $X$ be a geodesic metric space equipped with an undirected navigable path system $\mc P$. Let $\la, \ka, R$ be constants. There exist constants $\eps_0, A_0, L, \la', \ka' ,C$ such that the following holds. Assume we have one of the following:
    \begin{enumerate}
        \item $\gamma$ is a path which is $L$--locally a $(R; \eps_0, A_0; 3, \mc P)$--weakly polygonally Morse $(\la, \ka)$--quasi-geodesic and $\mc P$ is $\kao$--sub-bounded;\label{P:sub-bounded}
        \item $\gamma\in \mc P$ is a special path which is $L$--locally $(R; \eps_0, A_0; 3, \mc P)$--weakly polygonally Morse;\label{P:special-path}
        \item $\gamma$ is a $(\la, \ka)$--quasi-geodesic which is globally $(R; \eps_0, A_0; 3, \mc P)$--weakly polygonally Morse.\label{P:globally-wpm}
    \end{enumerate}
    Then $\gamma$ is a $(\la', \ka')$--quasi-geodesic and $\mc P$--contracting with constant $C$.
\end{theorem}

\begin{proof}
   Assume that $A_0, L, \la', \ka',C$ are large enough and $\eps_0$ is small enough. In all three cases, we want to use Proposition~\ref{prop:3-implies-n} to show that $\gamma$ satisfies the assumptions of Proposition~\ref{prop:local_poly_P_contracting}. The latter shows that $\gamma$ is $\mc P$--contracting with constant $C$ as desired.

   For \ref{P:sub-bounded}, we can do this due to Proposition~\ref{prop:globalisation_QG}, for \ref{P:special-path} it is immediate and for \ref{P:globally-wpm} it follows from Lemma~\ref{lem:close-to-special-paths-globally}.

    If $\mc P$ is $\ka$--sub-bounded, then by Proposition~\ref{prop:globalisation_QG}, $\gamma$ is a $(\la', \ka')$--quasi-geodesic and there exists a constant $D$ such that any subpath of $\gamma$ has Hausdorff distance at most $D$ from a special path between its endpoints.

    Thus, in either case, Proposition~\ref{prop:3-implies-n} shows that $\gamma$ satisfies the assumptions of Proposition~\ref{prop:local_poly_P_contracting}. The latter shows that $\gamma$ is $\mc P$--contracting with constant $C$ as desired.
\end{proof}

\subsection{Linearity of divergence. Global cut-points in asymptotic cones}

Another consequence of navigability is that linearity of the divergence function on a sequence diverging to infinity implies global linearity of the divergence function. 

\begin{theorem}\label{thm1:div}
    Let $X$ be a geodesic metric space equipped with a $(\lao, \kao)$--quasi-geodesic $(C_0, k)$--navigable undirected path system $\mc P$. Assume that for every point $x\in X$ the ball $B(x, \kappa_0)$ intersects a sequence $\spath_{a_n, b_n}$ of special paths between points $a_n, b_n$ such that the distances from $x$ to $a_n$, respectively to $b_n$, diverge to~$\infty$.
    
    Assume that there exists $\delta, \eps$ and a sequence $n_k\to \infty$ such that $\Dv_\eps (n_k, \delta )\leq C n_k$, where $C>0$ is independent of $k$. Then there exists $C'>0, \delta', \eps'$ such that $\Dv_{\eps'} (n, \delta' )\leq C'n$ for every $n$.
\end{theorem}

\begin{proof} 
We will show that the statement holds for $C', \eps'$ large enough and $\delta'$ small enough to be determined later. Let $a, b, c\in X$ be points with $n = d(a, b)$ and $r = d(c, \{a, b\})$. Let $\rho = \delta'r - \eps'$. Firstly, if there is a special path $\gamma$ from $a$ to $b$ with $d(\gamma, c)> \rho$, we are done. So we may assume from now on that there exists a point $m\in \gamma$ with $d(m, \gamma)\leq \rho$ and $\rho \geq 0$. With this, $r(1-\delta')\leq r/2\leq  D_{\mc P}(n)$, giving a lower bound on $n$ in terms of $r$.

\begin{figure}
    \centering
    \includegraphics[width=0.7\linewidth]{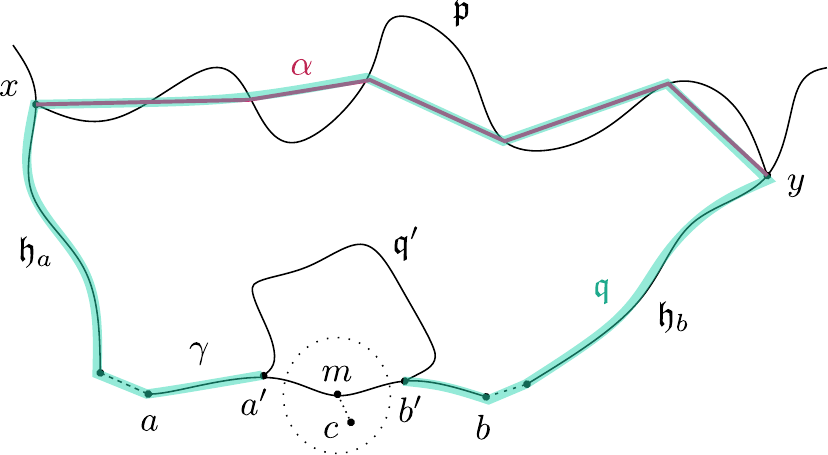}
    \caption{Proof of Theorem~\ref{thm1:div}. The path $\qgot$ stays far from $m$.}
    \label{fig:space-divergence}
\end{figure}

Let $k$ be such that $n_k$ is much larger than all other distances involved. Using the assumption on $X$ we can find special paths $\spath_a, \spath_b\in \mc P$ such that for $u = a, b$, we have $d(\spath_u^-, u)\leq \kao$, $d(\spath_u, m)\geq 2\delta' r C_0$ and $\diam(\spath_u)\geq n_k$, cf Lemma~\ref{lemma:quasi-geodesic-concatenation} (we choose $\eps'\gg 1/\delta'$ to obtain the lower bound in the assumption of Lemma~\ref{lemma:quasi-geodesic-concatenation}).

Let $x, y$ be points on $\spath_a, \spath_b$ with $d(x, a) = n_k/3$ and $d(y, b) = n_k/3$. Since $n_k$ is much larger than any of the other distances involved, we have $d(x, c)\geq n_k/6$ and $d(y, c)\geq n_k/6$ and hence $r ' = d(c, \{x,y\})\geq n_k/6\gg r$. Since $\Dv_{\eps} (n_k, \delta )\leq C n_k$, we can find a path $\pgot:[0, T]\to X$ of length at most $Cn_k$ from $x$ to $y$ with $d(\pgot, c)\geq \delta r' - \eps \geq 4\delta'r C_0$.

Let $M$ be an integer constant. For $1\leq i \leq M$, let $\alpha_i\in \mc P$ be a special path from $\pgot((i-1)T/M)$ to $\pgot(iT/M)$. Define the $(M, \mc P)$--polygonal line $\alpha$ as $\alpha_1\ast \ldots \ast \alpha_M$. For $M$ large enough compared to the quasi-geodesic constants of $\mc P$, we have by the triangle inequality that $d(\alpha, m) >d(\pgot, m)/2\geq 2\delta'rC_0$.

Let $a', b'$ be the first and last point on $\gamma$ in the $4\delta'rC_0$--neighbourhood of $m$. Consider the $(M+6, \mc P)$--polygonal line 
\begin{align*}
    \qgot = \pref{\gamma}{a'}^{-1}\ast [a, \spath_a^-]_{\mc P}\ast \pref{\spath_{a}}{x}\ast \alpha \ast  \pref{\spath_{b}}{y}^{-1} \ast [\spath_b^-, b]_{\mc P}\ast \suf{\gamma}{b'}^{-1},
\end{align*}
as depicted in Figure~\ref{fig:space-divergence}. 

We have that $d(\qgot, m) > 2\delta'rC_0$, (for this, we use that $\delta'$ is much smaller than the quasi-geodesic constants of $\mc P$). Hence, by navigability there exists a $(k(M+6), \mc P)$--polygonal line $\qgot'$ from $a'$ to $b'$ with $\norm{\qgot''}\leq 2\delta'rC_0(M+6)$ and with $d(\qgot', m) > 2\delta'r$. 

Thus, the path $\qgot'' = \pref{\gamma}{a'}\ast \qgot''\ast \suf{\gamma}{b'}$ satisfies $d(\qgot'', c) \geq \delta'r $ and has length at most $\norm{\gamma}+2\delta'rC_0(M+6)$. Since $M$ is a constant and $r$ has a linear upper bound in terms of $n$, we have that for large enough $C'$, $\norm{\qgot''}\leq C'n$, concluding the proof.
 \end{proof}

\begin{theorem}\label{thm2:div}
    Let $X$ be a geodesic metric space as in Theorem \ref{thm1:div}, moreover with a cobounded action of its group of isometries, and let $x_0\in X$ be an arbitrary basepoint.  If there exists a scaling sequence $(\alpha_n)$ diverging to $\infty$, such that for all non-principal ultrafilters $\omega$, the asymptotic cone ${\mathrm{Cone}}_\omega (X, x_0, \alpha_n)$ has no global cut-points then all the asymptotic cones of $X$ are without global cut-points. 
\end{theorem}

\begin{proof} Since $X$ has a cobounded action of its group of isometries, every asymptotic cone is isometric to an asymptotic cone with sequence of observation points constant equal to $(x_0)$. Assume that some asymptotic cone ${\mathrm{Cone}}_\omega (X, x_0, \beta_n )$ has a global cut-point $\lim_\omega c_n$. In particular, let $\lim_\omega a_n$ and  $\lim_\omega b_n$ be two points that are in two distinct connected components of the complementary of the singleton set $\{ \lim_\omega c_n \}$. This implies that the limit set of the geodesics $[a_n, b_n]$ must contain $\lim_\omega c_n$, thus there exists $c_n'\in [a_n,b_n]$ such that $\lim_\omega c_n'=\lim_\omega c_n$. By potentially replacing either $a_n$ or $b_n$ (and hence their limit point) with points on the geodesic that are closer to $c_n'$ we may also assume that $c_n'$ is the 
midpoint of  $[a_n,b_n]$.   

As in the proof of Theorem \ref{thm1:div}, we now construct appropriate $(\lao, \kao)$-quasi-geodesics $\spath_n$ and $\spath_n'$ emanating from $a_n$ and $b_n$ respectively, and long enough. The `long enough' will be defined in what follows, depending on the case we are in.

\medskip

{\sc{Case 1.}} \quad Assume that $\lim_\omega \frac{\alpha_n}{\beta_n}>0$. In this case, we choose $\spath_n$ and $\spath_n'$ with lengths of order $\alpha_n$. With an argument similar to the one in the proof of Theorem \ref{thm1:div}, we can produce a polygonal path composed of a uniformly bounded number of special paths, each of length of order $\alpha_n$ or less and avoiding a ball $B(c_n, \alpha_n \delta C_0-\epsilon)$. We then use navigability to obtain a polygonal line with $\alpha_n$ replaced everywhere by $\beta_n$. In the asymptotic cone ${\mathrm{Cone}}_\omega (X, x_0, \beta_n )$, this contradicts the fact that $\lim_\omega c_n$ separates $\lim_\omega a_n$ and $\lim_\omega b_n$.  

\medskip

{\sc{Case 2.}}\quad Assume that $\lim_\omega \frac{\alpha_n}{\beta_n}=0$. Since $\alpha_n \to \infty$, for every $n\in \N$ there exists a $\varphi (n)$ such that for every $k\geq \varphi (n)$, $\alpha_k \geq \beta_n$. We can pick the function $\varphi:\N \to \N$ to be strictly increasing. We define the ultrafilter $\nu $ with support $\{ \varphi (n) \mid n\in \N \}$ and such that $\nu \circ \varphi = \omega$, where by abuse of notation we denote also by $\varphi$ the map defined by $\varphi$ on the set of subsets of $\N$. The asymptotic cone  ${\mathrm{Cone}}_\nu (X, x_0, \alpha_n )$ can be identified with ${\mathrm{Cone}}_\omega (X, x_0, \alpha_{\varphi (n)})$. Our assumption implies that ${\mathrm{Cone}}_\omega (X, x_0, \alpha_{\varphi (n)})$ has no cut-points. Since we have that $\lim_\omega \frac{\alpha_{\varphi (n)}}{\beta_n}\geq 1$, we can repeat the argument from Case 1. 
\end{proof}

\begin{corollary}
    Let $G$ be a finitely generated group equipped with a navigable, undirected path system.  
    If there exists a scaling sequence $(\alpha_n)$ diverging to $\infty$ such that for every non-principal ultrafilter $\omega$ the asymptotic cone ${\mathrm{Cone}}_\omega (G, \alpha_n)$ has no global cut-points then all the asymptotic cones of $G$ are without global cut-points. 
\end{corollary}

\section{Spaces with a navigable path system}\label{sec:finding_navigable}

A key assumption in the previous two sections is the presence of a navigable path system. The goal of this section is to show that many naturally occurring metric spaces admit a navigable path system. 

Firstly, we show that a bounded and avoidant path system is navigable, providing a criterion for check navigability. Then we show that the push forward of a navigable path system is navigable. Lastly, we show that the following spaces admit navigable path systems: geodesic median spaces, spaces admitting a bounded geodesic combing and spaces admitting a bounded and quasi-consistent bicombing. In the first two cases, the navigable path system can be taken as the set of all geodesics in the space and it is in particular $G$--invariant for any group $G$ acting by isometries on the space. In the case of a bounded and quasi-consistent bicombing, the navigable path system is the so called `combing path system' induced by the combing lines and their subsegments. As a corollary, we obtain that geodesic median spaces and spaces admitting a geodesic combing are {\emph{Morse-dichotomous}} (see Definition~\ref{def:morse-dichotomous}), that is, all Morse geodesics are strongly contracting.

\subsection{Navigability from boundedness and avoidance}

We start by showing that a bounded and avoidant path system is navigable. At the core of the proof lies the fact that sub-boundedness allows to slide the starting polygonal line until each of its segments intersects the correct annulus, and then uses avoidance to control the length of the resulting curve, following the blueprint of \cite[Lemma~7.1]{PetytSprianoZalloum:hyperbolic} for CAT(0) spaces. However, the fact that the assumptions are so much weaker requires upgrades of the sliding techniques. 

\begin{proposition}[Bounded and avoidant imply navigable]\label{prop:bounded-avoidant-implies-navigable}
    Let $\mc{Q}$ be a $\kao$--bounded undirected heap of paths on a geodesic metric space $X$ such that $\overline{\mc Q}$ is $(C, k)$--avoidant for some $C\geq 3\kao\geq 3$. Then $\mc Q$ is $(C, 2k)$--navigable.
\end{proposition}

\begin{proof}
Let $R\geq C$, $m\in X$ and $\alpha = \alpha_1\ast \ldots \ast\alpha_n$ be an $(n, \mc Q)$--polygonal line as in the definition of navigability.  For $1\leq i \leq n$, let $\beta_i, \gamma_i$ be paths in $\overline{\mc Q}$ with $R - \kao \leq d(\beta_i, m) \leq R$, $R - \kao \leq d(\gamma_i, m)\leq R$ and such that $\beta_i^- = \alpha_i^-$, $\beta_i^+=\gamma_i^-$ and $\gamma_i^+ = \alpha_i^+$. This is depicted in Figure~\ref{fig:navigabiblity-replacement}, Left.

	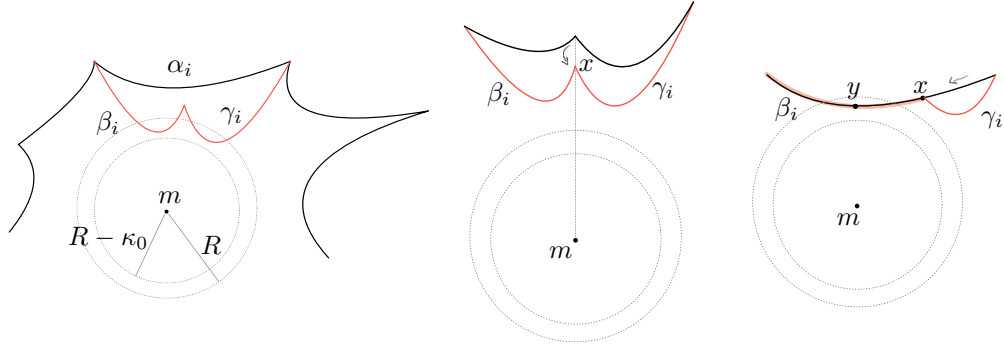
\begin{figure}[h]
		\begin{tikzpicture}
			\node[anchor=south west,inner sep=0] (image) at (0,0) {\includegraphics[width=\textwidth]{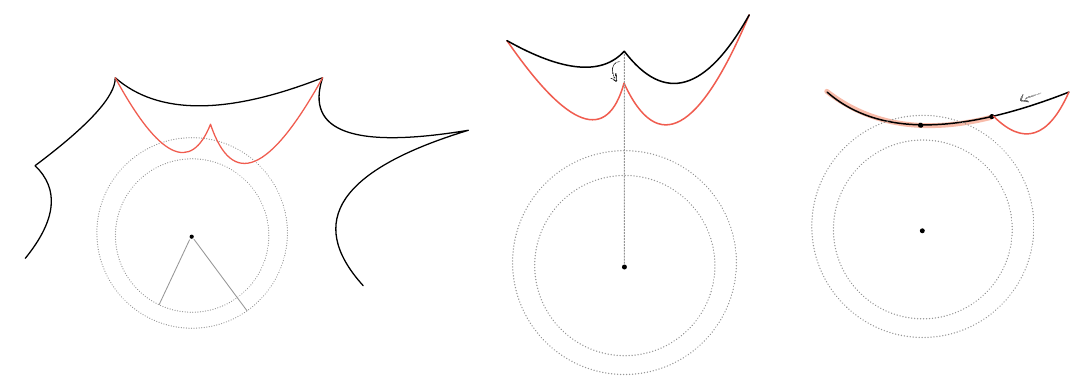}};
			\begin{scope}[x={(image.south east)},y={(image.north west)}]
				\node at (0.18, 0.43) {$m$};
				\node at (0.22, 0.3) {$R$};
				\node at (0.12, 0.32) {$R - \kao$};
				\node at (0.19, 0.78) {$\alpha_i$};
				\node at (0.12, 0.63) {$\beta_i$};
				\node at (0.24,0.65) {$\gamma_i$};
                \node at (0.5, 0.71) {$\beta_i$};
                \node at (0.66, 0.72) {$\gamma_i$};
                \node at (0.585, 0.788) {$x$};
                \node at (0.56, 0.28) {$m$};
                \node at (0.84, 0.37) {$m$};
                \node at (0.845, 0.72) {$y$};
                \node at (0.91, 0.73) {$x$};
                \node at (0.78, 0.67) {$\beta_i$};
                \node at (0.98, 0.64) {$\gamma_i$};
			\end{scope}
    \end{tikzpicture}\caption{\textbf{Left:} we replace $\alpha_i$ by two special paths that have the correct distance from $m$. \textbf{Center:} a depiction of the central slide. \textbf{Right:} a depiction of the side slide. }\label{fig:navigabiblity-replacement}
	\end{figure}

Such paths $\beta_i, \gamma_i$ always exist: start with $\beta_i = \alpha_i$ and $\gamma_i$ a trivial path from $\alpha_i^+$ to itself. With this definition, we satisfy all the requirements except perhaps the upper bound on $d(\beta_i, m)$ and $d(\gamma_i, m)$. We now  `move $\beta_i^+ = \gamma_i^-$ to $m$', that is, we repeat the following process, called \emph{central slide}, until $d(\beta_i, m)\leq R$ or $d(\gamma_i, m)\leq R$: Let $x$ be a point on $[m, \beta_i^+]$ at distance $1$ from $\beta_i^+$. Replace $\beta_i$ by $[\beta_i^-, x]_{\mc Q}$ and $\gamma_i$ by $[x, \gamma_i^+]_{\mc Q}$. By boundedness of $\mc Q$ (and undirectedness), we may assume that the new and old paths have Hausdorff distance at most $\kao$ from each other. Thus, $R - \kao < d(\beta_i, m)$ and $R - \kao < d(\gamma_i, m)$. At each step, the distance from $\beta_i^+$ to $m$ decreases by $1$. Thus, after at most $d(\alpha_i^+, m)$ steps, $d(\beta_i, m)\leq d(\beta_i^+, m) < 1\leq \kao$, implying that the central slide has to finish after finitely many steps.

\smallskip

After repeating the central slide finitely many times, we will have $d(\beta_i, m)\leq R$ or $d(\gamma_i, m)\leq R$. We assume we are in the former case, as the latter is completely analogous. Thus, we have $d(\beta_i, m)\leq R$. Next we `slide the endpoint of $\gamma_i$ along $\beta_i$', that is, we repeat the following process, called \emph{side slide}, until $d(\gamma_i, m)\leq R$: Let $y = \beta_i(t)$ be a point on $\beta_i$ with $d(y, m)\leq R$. Let $s\geq t$ be minimal such that $d(\beta_i(s), \beta_i^+)\leq 1$. Denote $\beta_i(s)$ by $x$. Replace $\beta_i$ by $\pref{\beta_i}{x}$ and $\gamma_i$ by $[x,\gamma_i^+]_{\mc Q}$. Again by boundedness and undirectedness, we may assume that the Hausdorff distance between the old and new $\gamma_i$ is at most $\kao$. Thus, we know that $R - \kao < d(\gamma_i, m)$. Moreover, since $s\geq t$, we know that $R- \kao < d(\beta_i, m)\leq R$. When the side slide process ends, $\beta_i$ and $\gamma_i$ satisfy all desired conditions. The side slide process has to end after a finite number of steps: if $x = \beta(t)$, then $d(\gamma_i, m)\leq d(x, m)\leq R$. If $x\neq \beta(t)$, then by minimality of $s$, for the old $\beta_i$, $d(x, \beta_i^+) = 1$. Since quasi-geodesics are continuous, the latter can only happen finitely many times.

It is important to note that in the side slide process, $\beta_i$ is in $\overline{\mc Q}$ but not necessarily in $\mc Q$ while $\gamma_i$ is in $\mc Q$, allowing us to do the sliding. If the central slide ends with $d(\gamma_i, m)\leq R$, we do the side slide but reverse the roles of $\beta_i$ and $\gamma_i$. 

\smallskip

Define $\alpha'_{2i-1} = \beta_i$ and $\alpha'_{2i} = \gamma_i$ for all $i$. Let $z_1 = {\alpha_1'}^-$, $z_{2n} = {\alpha'}_{2n}^+$. Further, for $1 < i < 2n$, let $z_i$ be a point on $\alpha_i'$ with $R - \kao < d(z_i,  m) \leq  R$. Those exist by the definition of $\beta_i$ and $\gamma_i$. Define $R ' = R - \kao$. Observe that $2R'\geq R$. Hence, for $1\leq i \leq 2n$, we can use that $\overline{\mc Q}$ is $(C, k)$--avoidant, to get $(k, \overline{\mc Q})$--polygonal lines $\pgot_i$ from $z_i$ to $z_{i+1}$ which are disjoint from $B(m, 2R'/C)\supseteq B(m, R/C)$ and satisfy $\norm{\pgot_i}\leq C R'$. Consequently, $\pgot = \pgot_1\ast \ldots \ast \pgot_{2n}$ is a $(2kn, \overline{\mc Q})$--polygonal line from $\alpha^-$ to $\alpha^+$ which does not intersect $B(m, R/C)$.
\end{proof}

\subsection{Navigability of the push-forward}

We now record an important fact, namely that being navigable is a property preserved by quasi-isometries.

\begin{lemma}\label{lem:navigable-push-forward}
    Let $f : X \to Y$ be a $C$--quasi-isometry between geodesic metric spaces $(X, d_X)$ and $(Y, d_Y)$. Let $\mc P$ be an undirected $(C', k)$--navigable path system on $X$. Then $(\mc P_f)^{\pm}$, the directional closure of the $C$--push-forward $\mc P_f$, is a $(C'', 3k+4)$--navigable path system, for a constant $C''$ depending only on $C,C',k$ and the quasi-geodesic constants of $\mc P$. 
\end{lemma}
\begin{proof}
   We will show that this holds for $C''$ large enough. Let $R\geq C''$ be a constant, $m\in Y$ a point and $\alpha = \alpha_1\ast \ldots\ast \alpha_n$ be a $(n, \mc P_f^\pm)$--polygonal line disjoint from $B_Y(m, R)$ but whose endpoints have distance $2R$ from $m$. We want to `transport' this picture into $X$, which is depicted in Figure~\ref{fig:navigable-in-x}. So we define the following. Let $m'\in X$ be a point such that $d(f(m'), m)\leq C$. By Remark~\ref{rem:push-forward-properties}, there exist special paths $\alpha_i'\in \mc P^\pm$ with $d_{\mathrm{Haus}}(f(\alpha_i'), \alpha_i)\leq C_f$ and such that the endpoints of $f(\alpha_i')$ have distance at most $C$ from the endpoints of $\alpha_i$ respectively. 
        \begin{figure}
        \centering
        \includegraphics[width=.9\linewidth]{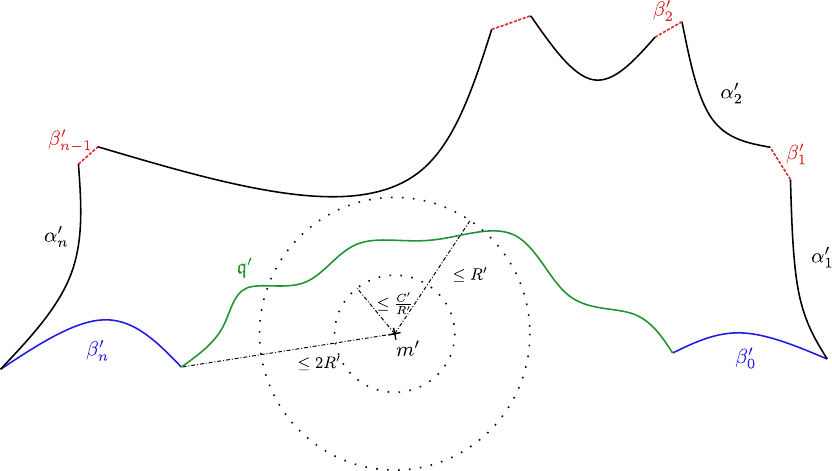}
        \caption{Situation in $X$.}
        \label{fig:navigable-in-x}
    \end{figure}
    \begin{figure}
        \centering
        \includegraphics[width=.9\linewidth]{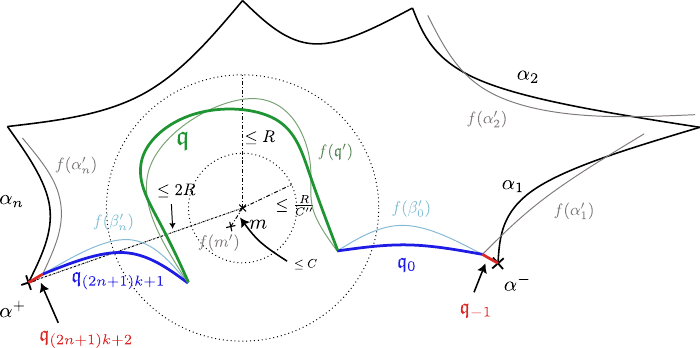}
        \caption{Situation in $Y$.}
        \label{fig:navigable-in-y}
    \end{figure}

   For $1\leq i \leq n$, let $\beta_i'\in \mc P$ be a special path from $\alpha_i'^+$ to $\alpha_{i-1}'^-$. 
   Since $f$ is a quasi-isometry, it follows that, for an appropriate constant $K$ and $R' = R/K - K$, the path $\pgot' = \alpha_1'\ast \beta_1'\ast \alpha_2'\ast \ldots \ast \beta_{n-1}'\ast \alpha_n'$ is a $(2n-1, \mc P)$--polygonal line which is disjoint from $B_X(m', R')$. Let $\beta_0'\in \mc P$ be a special path which is disjoint from $B_X(m', R')$, for which $d(\beta_0'^-, m')\leq 2R'$ and $\beta_0'^+ = \alpha_1'^-$. Such a path exists, for example a suitable subpath of a special path from $m'$ to $\alpha_1'^-$ satisfies these constraints. Similarly, let $\beta_n'\in \mc P$ be a special path from $\alpha_{n}'^+$ to a point in $B_X(m', 2R')$ such that $\beta_n'$ is disjoint from $B_X(m', R')$. For $C''$ large enough, we have that $R'\geq C'$ and hence we can apply the navigability of $\mc P$ to the $(2n+1, \mc P)$--polygonal line $\pgot'' = \beta_0'\ast \pgot'\ast \beta_n'$. 

    This yields a $((2n+1)k, \mc P)$--polygonal line $\qgot' = \qgot_1'\ast \ldots \ast \qgot_{k(2n+1)}'$ from $\beta_0'^-$ to $\beta_n'^+$ which is disjoint from $B_X(m', R'/C')$ and with $\norm{\qgot'}\leq C'R'(2n+1)$. Now, we need to `transport' this back to $Y$, as depicted in Figure~\ref{fig:navigable-in-y}. For $1\leq i \leq (2n+1)k$, let $\qgot_i\in \mc P_f^\pm$ be a special path with the same endpoints as $f(\qgot_{i}')$ and  at Hausdorff distance at most $C_f$ from $f(\qgot_{i}')$ (such a path exists, by Remark~\ref{rem:push-forward-properties}). Furthermore, let $\qgot_0$ and  $\qgot_{(2n+1)k+1}$ be special paths in $P_f^{\pm}$ such that $\qgot_0$ and $\qgot_{(2n+1)k+1}$ have the same endpoints as $f(\beta_0'^{-1})$ and $f(\beta_n'^{-1})$ and are at Hausdorff distance at most $C_f$ from $f(\beta_0'^{-1})$ and $f(\beta_n'^{-1})$, respectively. Lastly, let $\qgot_{-1}, \qgot_{(2n+1)k+2}\in \mc P_f^{\pm}$ be special paths from $\alpha^-$ to $f(\beta_0'^+)$ and from $f(\beta_n'^-)$ to $\alpha^+$. Both of those special paths have endpoints at most $C$ away from each other. The $((2n+1)k+4, \mc P_f^\pm)$--polygonal line $\qgot = \qgot_{-1}\ast \qgot_0\ast \ldots \ast \qgot_{(2n+1)k+2}$ has endpoints $\qgot^- = \alpha^-$ and $\qgot^+ = \alpha^+$ and for an appropriate constant $K'$, depending only on $C, C'$ and the quasi-geodesic constants of $\mc P$, we have that $\qgot$ is disjoint from $B_Y(m, R/K'-K')$. Using the fact that quasi-geodesics are improved, we get that for $C''$ large enough compared to $K'$ and $k$
   \begin{align*}
       \norm{\qgot}\leq K'\norm{\qgot'} + K'(2n+1)k + \norm{\qgot_0} + \norm{\qgot_{-1}} + \norm{\qgot_{(2n+1)k+1}} + \norm{\qgot_{(2n+1)k+2}}  \leq C''Rn.
   \end{align*}
   Hence, for large enough $C''$ compared to $K''$ and $k$, we have that $\norm{\qgot}\leq C''Rn$ and that $\qgot$ is disjoint from $B_Y(m, R/C'')$, implying that $\mc P_f^\pm$ is indeed $(C'', 3k+4)$--navigable.
\end{proof}

\subsection{Bicombing path systems}

In this section we associate a path system, called combing path system, to every (bi)-combing. We then show that if the (bi)-combing is good enough -- namely, if it is a quasi-consistent bounded bicombing or a geodesic combing -- then the combing path system is navigable. Additionally, we obtain that a geodesic metric space with a bounded combing is Morse dichotomous.

\begin{definition}[Combing path system]\label{defn:combing_path_system}
Let $X$ be equipped with a (bi)-combing. The induced heap of paths $\mc P$ that contains all combing paths, inverses thereof and trivial paths between any point and itself is called {\emph{the combing heap}} and it is an undirected heap of paths. The consistent closure $\mc {P} = \overline{\mc Q}$ of the combing heap is called the \emph{combing path system} and it is an undirected path system.
\end{definition}

\subsubsection{Quasi-consistent and bounded bicombings}

We show that the combing path system of a quasi-consistent and bounded bicombing is avoidant and bounded and we use Proposition~\ref{prop:bounded-avoidant-implies-navigable} to conclude that it is navigable.

\begin{lemma}\label{lem:bicombng-is-bounded}
    Let $\mc P$ be the combing path system induced by a $\kao$--quasi-consistent $\kao$--bounded $(\lao, \kao)$--quasi-geodesic bicombing on a geodesic metric space $X$. Then $\mc P$ is $3\kao$--bounded.
\end{lemma}
\begin{proof}
    This is a direct consequence of quasi-consistency and boundedness. Namely, if $\spath'\in \mc P$ is a subpath of a combing line $\spath$, and $\qgot$ is the combing line between $\spath'^-$ and $\spath'^+$  then by quasi-consistency $d_{\mathrm{Haus}}(\spath', \qgot)\leq \kao$. Now, if $y$ is a point with $d(y, \spath'^+)\leq 1$ and $\qgot'$ is the combing line between $\spath'^-$ and $y$ we have by boundedness that $d_{\mathrm{Haus}}(\qgot, \qgot')\leq 2\kao$ and hence $d_{\mathrm{Haus}}(\spath',{\qgot'})\leq 3\kao$.  The case where $\spath'$ is a subpath of the inverse of a combing line follows analogously, we just have to reverse all paths.
\end{proof}

\begin{lemma}\label{lem:bicombing-is-avoidant}
    Let $\mc P$ be the combing path system induced by a $\kao$--bounded $\kao$--quasi-consistent $(\lao, \kao)$--quasi-geodesic bicombing on a geodesic metric space $X$. Then $\mc P$ is $(C,6)$--avoidant, for some constant $C$ depending only on $(\lao, \kao)$. 
\end{lemma}
\begin{proof}
    We first prove the following claim, which states that the `set of combing paths' is slightly better than $(C', 3)$--avoidant, for large enough $C'$. 
    \begin{claim}\label{claim:bic-avoid}
        There exists a constant $C'$ such that the following holds for all $R\geq C'/2$. Let $m\in X$ be a point and let $\qgot_1, \qgot_2$ be two combing lines for which $d(\qgot_1^-, m)\leq 4R, d(\qgot_2^-, m)\leq 4R,\qgot_1^+ = \qgot_2^+$ and $d(m, \qgot_1\cup \qgot_2)\geq R - 4\kao$. There exists a $(3, \mc P)$--polygonal line $\pgot$ from $\qgot_1^-$ to $\qgot_2^-$ with $\norm{\pgot}\leq C'R$ and $d(\pgot, m) > 2R/C'$.
    \end{claim}
    
    \textit{Proof of Claim \ref{claim:bic-avoid}.}
    Let $\qgot_1, \qgot_2, m, R$ be as in the statement of the claim. By boundedness of the bicombing, for any point $x_1\in \qgot_1$, there exists a point $x_2\in \qgot_2$ with $d(x_1, x_2)\leq 8\kao R+\kao\leq 9\kao R$. Choosing $C'$ large enough, we get $R - 4\kao \geq 2R/C'$. 

    Fix a constant $K$, which will depend only on $\lao$ and $\kao$. If $\abs{\qgot_1}\leq KR$, then we can define $\pgot = \qgot_1\ast \qgot_2^{-1}$ and we have that $\norm{\pgot}\leq C'R$ as long as $C'$ is large enough compared to $\lao$ and $\kao$ (and $K$). Note that $d(m, \pgot) > R - 4\kao \geq 2R/C'$.

    \smallskip
    
    If $\abs{\qgot_1} > K R$, let $x_1 = \qgot(t_1)$ be the point on $\qgot_1$ whose domain distance from $\qgot_1^-$ is $KR$. Let $x_2\in \qgot_2$ be a point with $d(x_1, x_2)\leq 9\kao R$ and define $\pgot = \psub{\qgot_1}{\qgot_1^-, x_1}\ast [x_1, x_2]_{\mc P} \ast \psub{\qgot_2^{-1}}{x_2,\qgot_2^-}$. By repeatedly applying the triangle inequality and using that all quasi-geodesics are improved, we get $\norm{\pgot}\leq C'R$ as long as $C'$ is large enough compared to $\lao$ and $\kao$ (and $K$). Moreover, if $K$ is large enough compared to $\lao, \kao$, then $d([x_1, x_2]_{\mc P}, m)\geq d(x_1, m) -  \diam([x_1, x_2]_{\mc P}) > R$, implying $d(m, \pgot) > 2R/C'$.
    \hfill$\blacksquare$

    \smallskip

    By an analogous argument, we see that Claim~\ref{claim:bic-avoid} holds if both $\qgot_1$ and $\qgot_2$ are inverses of combing lines instead of combing lines. Now we consider the general case. 
    
    \smallskip

    Let $C = 2C'$ and let $R\geq C/2$. Let $\spath_1, \spath_2\in \mc P$ be special paths and let $m \in X$ be a point such that $d(m, \spath_1\cup \spath_2)\geq R, \spath_1^+=\spath_2^+ = y$ and $d(m, \spath_i^-)\leq 4R$ for $i\in \{1, 2\}$. For $i\in \{1, 2\}$, let $\spath_i'$ be a combing line (or an inverse thereof) with the same endpoints as $\spath_i$ and at Hausdorff distance at most $\kao$ from $\spath_i$ (this exists by quasi-consistency of the combing). Now, if both $\spath_i'$ are combing lines, Claim~\ref{claim:bic-avoid} concludes the proof. Similarly, if both $\spath_i'$ are inverses of combing lines, an analogue of Claim~\ref{claim:bic-avoid} concludes the proof. Hence, without loss of generality, we assume that $\spath_1'$ is a combing line and $\spath_2'$ is an inverse of a combing line. Note that $d(m, \spath_i')\geq R - \kao$ for $i\in \{ 1,2\}$. By a (coarse) continuity argument, there exists a combing line $\qgot$ from a point $x\in \spath_2'$ to $y$ such that $R - 3\kao \leq d(\qgot, m)\leq 4R$. Let $z\in \qgot$ be a point with $d(z, m)\leq 4R$. This is depicted in Figure~\ref{fig:bic-nav}.

    \begin{figure}[ht]
        \centering
        \includegraphics[width=0.6\linewidth]{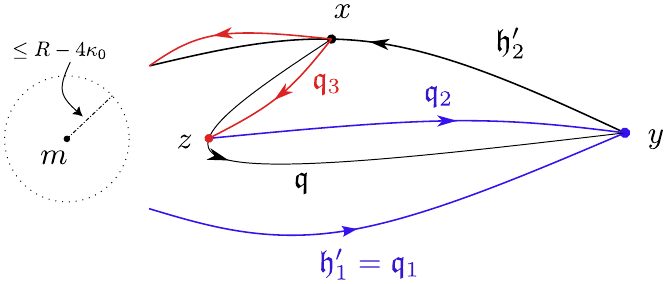}
        \caption{The arrows indicate which direction comes from a combing line.}
        \label{fig:bic-nav}
    \end{figure}
 
    Let $\qgot_1 = \spath_1'$, let $\qgot_2$ be a combing line from $z$ to $y$, let $\qgot_3$  a combing line from $x$ to $z$ and let $\qgot_4$ be a combing line from $x$ to $\spath_2^-$. By quasi-consistency, we have that $d(m, \qgot_i)\geq R -4\kao$ for $1\leq i \leq 4$. Applying Claim~\ref{claim:bic-avoid} to $\qgot_1$ and $\qgot_2$ and then to $\qgot_3$ and $\qgot_4$ and concatenating the two obtained paths, we get a path $\pgot$ from $\spath_1^-$ to $\spath_2^-$ with $\norm{\pgot}\leq 2C'R$ and which avoids $B(m, 2R/C')$. Since $C = 2C'$, this concludes the proof.
\end{proof}

We can now put everything together and obtain the following. 

\begin{proposition}\label{prop:bicombing-is-navigable}
    Let $\mc P$ be the combing path system induced by a $\kao$--quasi-consistent $\kao$--bounded $(\lao, \kao)$--quasi-geodesic bicombing on a geodesic metric space $X$. Then $\mc P$ is $(C, 12)$--navigable for some constant $C$ depending only on $(\lao, \kao)$.
\end{proposition}

\begin{proof}
    Lemmas~\ref{lem:bicombng-is-bounded} and \ref{lem:bicombing-is-avoidant} show that $\mc P$ is $3\kao$--bounded and $(C, 6)$--avoidant for some constant $C$ depending only on $(\lao, \kao)$. Proposition~\ref{prop:bounded-avoidant-implies-navigable} allows to conclude the proof.
\end{proof}

\subsubsection{Geodesic combings}

We show that if a space admits a bounded geodesic combing then any geodesic path system that contains the combing heap is navigable. Combined with the results from Section~\ref{sec:globalizations}, this implies that such spaces are Morse dichotomous, that is, all their Morse geodesics are strongly contracting.

\begin{lemma}\label{lem:combing-avoidance}
    Let $\mc P$ be the combing path system of a $\kao$--bounded geodesic combing on a geodesic metric space $X$. Then $\mc P$ is $(C, 5)$--avoidant for a constant $C$ depending only on~$\kao$.
\end{lemma}

\begin{figure}
    \centering
    \includegraphics[width=0.5\linewidth]{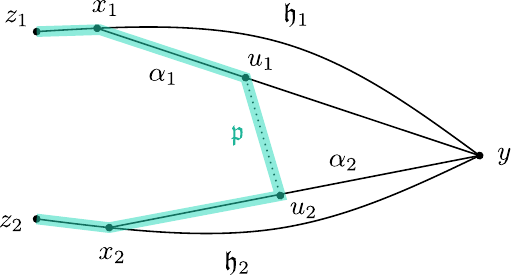}
    \caption{Construction of the path $\pgot$ in the proof of Lemma~\ref{lem:combing-avoidance}.}
    \label{fig:geodesic-combing}
\end{figure}

\begin{proof} 
    We show this for $C = 100(\kao+1)$. Let $R\geq C/2$. Let $m\in X$ and, for $i = 1, 2$, let $\spath_i\in \mc P$ be special paths from $z_i$ to $y$, as in the definition of avoidance. If $d(y, z_i)\leq (20\kao + 5)R$ for $i = 1$ or $i = 2$, then we can define $\pgot = \spath_1\ast \spath_2^{-1}$. By the triangle inequality, $\norm{\pgot}\leq CR$.

    Thus, we can assume that $d(y, z_i)\geq (20\kao + 5)R$ for $i = 1, 2$. Consider the point  $x_i$ on $\spath_i$ at distance $5R$ from $z_i$, define $\alpha_i$ as the combing line from $y$ to $x_i$ and define $u_i$ as the point on $\alpha_i$ at distance $20\kao R$ from $x_i$. This is depicted in Figure~\ref{fig:geodesic-combing}. Define the $(5, \mc P)$--polygonal line $\pgot$ as $\psub{\spath_1}{z_1, x_1}\ast\psub{\alpha_1^{-1}}{x_1, u_1}\ast [u_1, u_2]_{\mc P} \ast \psub{\alpha_2}{u_2, x_2}\ast \psub{\spath_2^{-1}}{x_2, z_2}$. For $\alpha_1$ and $\alpha_2$, it follows from the triangle inequality that $d(m, \alpha_i)\geq d(z_i, x_i)-d(z_i, m)\geq R$. By boundedness, $d(u_1, u_2)\leq \kao d(x_1, x_2)\leq 18\kao R + \kao$. Hence, $d([u_1, u_2]_{\mc P}, m)\geq R$ by the triangle inequality as well. It remains to show that $\norm{\pgot}\leq CR$, which again follows from the triangle inequality.
\end{proof}

\begin{proposition}\label{prop:combing-navigability}
    Let $\mc Q$ be the combing heap of a $\kao$--bounded geodesic combing of a geodesic metric space $X$. If $\mc P\supseteq \mc Q$ is a geodesic system of paths, then $\mc P$ is $(C, 42)$-navigable, for a constant $C$ depending only on $\kao$.   
\end{proposition}

Most notably, we can choose $\mc P$ as the combing path system or the set of all geodesics in $X$.

\begin{proof}
    
   First observe the following claim, which will help us replace geodesics from $\mc P$ by $(2, \mc Q)$--polygonal lines.

    \begin{claim}\label{claim:replacement-geodesics}
        Let $\gamma$ be a geodesic with $d(\gamma, m)\geq R$. There exists a point $x$ on $\gamma$ and geodesics $\beta_1, \beta_2\in \mc Q$ from $\gamma^-$ to $x$ and from $x$ to $\gamma^+$ with $d(\beta_i, m)\geq R/4 - 2\kao$, for $i = 1,2$. 
    \end{claim}

    \textit{Proof of Claim \ref{claim:replacement-geodesics}.}    
        Let $x\in \gamma$ and define $\beta_{1,x}, \beta_{2,x}$ as combing lines from $\gamma^-$ to $x$ and from $\gamma^+$ to $x$, respectively. In particular, $\beta_x = \beta_{1, x}\ast \beta_{2, x}^{-1}$ is a geodesic and it contains $x$. We first argue that $d(m, \beta_{i, x})\geq R/4$ for at least one $i \in \{ 1, 2\}$. Assume by contradiction that there exist points $y_i$ on $\beta_i$, for $i = 1, 2$, with $d(y_i, m)< R/4$. Now, $\psub{\beta_x}{y_1, y_2}$ is a geodesic from $y_1$ to $y_2$ and contains $x$. By the triangle inequality we get $d(x, m)\leq d(y_1, m) + d(y_1, x)\leq 2d(y_1, m) + d(y_2, m) \leq 3R/4 < R$. This contradicts $d(\gamma, m) \geq R$.

        We call a point $x\in \gamma$ a \emph{left}, respectively \emph{right}, \emph{point} if $d(\beta_{1, x}, m)\geq R/4$, respectively $d(\beta_{2, x}, m)\geq R/4$. The above shows that each point is a left point or right point (or both). Moreover, $\gamma^-$ and $\gamma^+$ have to be left,  respectively right, points. Consequently, there exists points $x, x'\in \gamma$ with $d(x, x')\leq 1$ such that $x$ is a left point and $x'$ is a right point. 

        Since $x$ is a left point, we have $d(\beta_{1, x}, m)\geq R/4$. Since $x'$ is a right point we have $d(\beta_{2,x'}, m)\geq R/4$. Finally, by boundedness of the combing, we have $d(\beta_{2, x}, m) \geq R/4 - 2\kao$, hence we can take $\beta_1 = \beta_{1, x}$ and $\beta_2 = \beta_{2, x}$. 
    \hfill$\blacksquare$

    \smallskip

    Since the combing is geodesic, the trivial paths in the combing heap $\mc Q$ are also combing lines. The $\kao$--boundedness of the combing implies the $2\kao$--boundedness of $\mc Q$. Moreover, by Lemma~\ref{lem:combing-avoidance}, $\overline{\mc Q}$, which is the combing path system, is $(C', 5)$--avoidant for a constant $C'\geq 3\kao\geq 3$ depending only on $\kao$. Consequently, by Proposition~\ref{prop:bounded-avoidant-implies-navigable}, $\mc Q$ is $(C', 10)$--navigable.
    
    Let $C = 32\max\{ C', \kao\}$. We want to show that $\mc P$ is $(C, 42)$--navigable. Let $R\geq C$, $m \in X$ and $\alpha = \alpha_{1}\ast\alpha_{2} \ldots \ast\alpha_{n}$ be an $(n, \mc P)$--polygonal line as in the definition of navigability.

    Define $R' = (R/4 - 2\kao)/2 \geq R /16$. We can use Claim~\ref{claim:replacement-geodesics} to replace each of the $\alpha_i$ by a $(2, \mc Q)$--polygonal line to obtain a $(2n, \mc Q)$--polygonal line $\beta = \beta_1\ast \ldots\ast \beta_{2n}$ from $\alpha^-$ to $\alpha^+$ with $d(\beta, m)> R'$. Further, let $u_-, u_+$ be points at distance $2R'$ from $m$ closest to $\beta^-$ and $\beta^+$. Now, $\beta' = [u_-, \alpha^-]_{\mc Q}\ast \beta \ast [\alpha^+, u_+]_{\mc Q}$ is a $(2n+2, \mc Q)$--polygonal line from $u_-$ to $u_+$ and has $d(\beta', m) > R'$. Since $\mc Q$ is $(C', 10)$--navigable, there exists a $(10(2n+2), \overline{\mc Q})$--polygonal line $\pgot$ from $u_-$ to $u_+$ with $d(\pgot, m) > R'/C'$ and $\norm{\pgot}\leq (2n+2)C'R'$. Let $\pgot' = [\alpha^-, u_-]_{\mc Q} \ast \pgot \ast [u_+, \alpha^+]_{\mc Q}$. This is a $(42n, \overline{\mc Q})$--polygonal line, moreover, $\norm{\pgot'}\leq \norm{\pgot}+2R\leq (2n+2)C'R'+2R\leq nCR$. Hence $\mc P$ is $(42, C)$--navigable.
\end{proof}

We finish this subsection by recording that if the path system consisting of all geodesics is navigable, then the space is Morse dichotomous. As a corollary, we obtain that spaces equipped with a geodesic combing are Morse dichotomous.

\begin{proposition}\label{lem:geodesic-path-navigable-implies-morse-dichotomous}
    Let $X$ be a geodesic metric space such that the path system $\mc P$ consisting of all geodesics in $X$ is navigable. Then $X$ is Morse dichotomous, that is, for all Morse gauges $M$ there exists a constant $C$ such that all $M$--Morse geodesics are $C$--contracting.
\end{proposition}

\begin{proof}
    Let $A, \eps$ be constants suitably chosen to apply Proposition~\ref{prop:local_poly_P_contracting}. Let $M$ be a Morse gauge. By Lemma~\ref{lemma:localmltg-implies-path-avoidence}, there exists $R>0$ such that any $M$--Morse geodesic is $(R; \eps, A; 7k, \mc P)$--weakly polygonally Morse. By Proposition~\ref{prop:local_poly_P_contracting}, there exists $C$ such that any $M$--Morse geodesic is $\mc P$--contracting with constant $C$. 

    It remains to prove the following claim, which concludes the proof.
    
    \begin{claim}\label{claim:p-contracting-implies-contracting}
        If $\gamma$ is a $\mc P$--contracting geodesic with constant $C$, then $\gamma$ is $(14C+2)$--contracting.
    \end{claim}

    \textit{Proof of Claim \ref{claim:p-contracting-implies-contracting}.}
    Let $\pi_{\gamma}$ be associated projection satisfying the conditions \eqref{5.2.1} and \eqref{5.2.2} of the Definition \ref{def:contracting} of a $\mc P$--contracting set. Let $\tau_{\gamma}$ denote a choice of closest point projection of $X$ onto $\gamma$. We start by showing that $\tau_{\gamma}$ and $\pi_{\gamma}$ coarsely agree. Let $x\in X$. Since $\gamma$ is $\mc P$--contracting with constant $C$, at least one of the following holds: $d(\pi_\gamma(x), \pi_\gamma(\tau_\gamma(x)))\leq C$ or $d([x, \tau_{\gamma}(x)], \pi_{\gamma}(x))\leq C$. If the former happens, we use the fact that $\tau_\gamma(x)$ lies on $\gamma$ and \eqref{5.2.1} to obtain $d(\pi_\gamma(x), \tau_\gamma(x))\leq 2C$. If the latter happens, then there exists a point $y$ on $[x, \tau_\gamma(x)]$ with $d(y, \pi_{\gamma}(x))\leq C$. Since $\tau_\gamma$ is a closest point projection, we have $d(y, \tau_\gamma(x))\leq C$, implying $d(\tau_\gamma(x), \pi_\gamma(x))\leq 2C$. Thus, indeed, $\pi_\gamma$ and $\tau_\gamma$ have pointwise distance at most $2C$.

    Let $x\in X$ and let $B$ be a ball around $x$ disjoint from $\gamma$. Let $y\in B$. We have that $d(x, y)\leq d(x, \gamma)$. If $d(\tau_\gamma(x), \tau_\gamma(y))\geq 7C+1$, then $d(\pi_\gamma(x), \pi_\gamma(y)) \geq 3C +1$, and $\gamma$ being $\mc P$--contracting with constant $C$ implies that there exists $z, z'\in [x, y]$ with $d(z, \pi_\gamma(y))\leq C$ and $d(z', \pi_\gamma(x))\leq C$. In particular, by the triangle inequality, $d(z, z') \geq C +1$. But now, $d(x, \gamma)\leq \min\{d(x, z), d(x, z')\}+C < d(x, y)$, a contradiction to $d(y, x)\leq d(x, \gamma)$. Hence, $d(\tau_\gamma(x), \tau_\gamma(y) )\leq 7C+1$, implying that $\gamma$ is $(14C+2)$--contracting.
    \hfill$\blacksquare$
\end{proof}

\begin{corollary}\label{cor:geodesic-combing-implies-morse-dichotomous}
   A geodesic metric space equipped with a bounded geodesic combing is Morse dichotomous. 
\end{corollary}
\begin{proof}
    Let $\mc P$ be the path system containing all geodesics in $X$. By Proposition~\ref{prop:combing-navigability}, $\mc P$ is navigable. Proposition~\ref{lem:geodesic-path-navigable-implies-morse-dichotomous} allows to conclude the proof.
\end{proof}

\subsection{Geodesics in median spaces yield bounded, avoidant, navigable path systems}

We show that in median geodesic spaces the path system consisting of all geodesic segments is bounded and navigable. As a corollary, we obtain that all geodesic median spaces are Morse dichotomous.

\begin{lemma}
Let $X$ be a median geodesic metric space and let $\mc P$ be the path system composed of all geodesics in $X$. The path system $\mc P$ is $2$--bounded.
\end{lemma}
\begin{proof}
    Let $x, y, y'\in X$ so that $d(y,y')\leq 1$ and let $\spath \in \mc P$ be a geodesic from $x$ to $y$. We have to show that there exists a geodesic from $x$ to $y'$ whose Hausdorff distance to $\spath$ is at most $2$. We prove the statement by induction on $d(x, y) + d(x, y')$. If $d(x, y) + d(x, y')\leq 2$, then the statement follows trivially and holds for any $\spath'\in \mc P$ from $x$ to $y'$. 
    
    From now on, we assume that $d(x, y) + d(x, y')\geq 2$. The induction assumption implies that for all points $a, a'\in X$ and every geodesic $\gamma$ from $x$ to $a$ such that $d(a, a')\leq 1$ and $d(x, a) + d(x, a')\leq d(x, y) + d(x, y') - 1$, there exists a geodesic $\gamma'$ from $x$ to $a'$ with $d_{\mathrm{Haus}}(\gamma, \gamma')\leq 2$. 

    \begin{figure}[ht]
        \centering
        \includegraphics[width=0.5\linewidth]{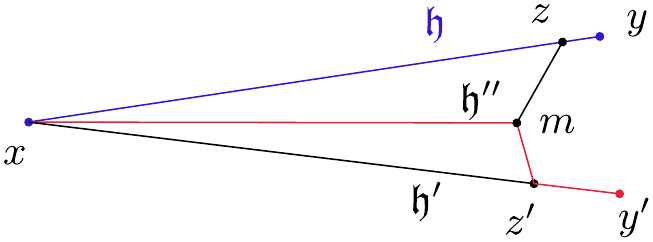}
        \caption{The red path has Hausdorff distance at most $2$ to $\spath$.}
        \label{fig:bounded-median}
    \end{figure}
    
    Let $\spath'\in \mc P$ be a special path (which is a geodesic) from $x$ to $y'$ and let $z\in \spath$ and $z'\in \spath'$ be points such that $d(x, z) = d(x, z')$ and such that $d(x, z) + d(x, z') = d(x, y) + d(x, y') -1$. Since $d(y, y')\leq 1$ and $d(x, y) + d(x, y')\geq 2$, such points always exist. Further note that $d(z, z')\leq 2$. Let $m$ be the median of the points $(z, z', x)$. Since $d(x, z) = d(x, z')$ we have that $d(m, z) = d(m, z')\leq 1$. This is depicted in Figure~\ref{fig:bounded-median}.

    Observe that we can now use the induction hypothesis to the points $z, m$ and $x$, since $d(x, z)+d(x, m)\leq d(x, z) + d(x, z')\leq d(x, y) + d(x, y') -1$. Thus, there exists a special path $\spath''\in \mc P$ from $x$ to $m$ with
    \begin{align*}
        d_{\mathrm{Haus}}(\psub{\spath}{x, z}, \spath'')\leq 2.
    \end{align*}
    Further observe that $\psub{\spath}{z, y}$ is contained in the $2$--neighbourhood of $z'$, $[m, z']$ is contained in the $2$--neighbourhood of $z$ and $[z', y']$ is contained in the $2$--neighbourhood of $y$. In particular, 
    \begin{align*}
        d_{\mathrm{Haus}}(\spath, \spath''\ast[m, z']\ast[z', y'])\leq 2.
    \end{align*}
    Since $\spath''\ast[m, z']\ast[z', y']$ is a geodesic and hence a special path, this concludes the induction step and hence the proof.
\end{proof}
In light of Proposition~\ref{prop:bounded-avoidant-implies-navigable}, to prove navigability it suffices to show that the geodesic path system is bounded. 
\begin{lemma}
    Let $X$ be a median metric space and let $\mc P$ be the path system containing all geodesics in $X$. The path system $\mc P$ is $(C, k)$--avoidant for $C = 28$ and $k = 3$.
\end{lemma}
\begin{proof}
    Let $R\geq 28 = C$. For $i = 1, 2$, let $y, z_i, m\in X$ be points with $d(m, z_i)\leq 4R$ and $\spath_i\in \mc P$ be special paths from $z_i$ to $y$ such that $d(\spath_i, m)\geq R$.

    \smallskip
        
    To show that $\mc P$ is $(C, 3)$--avoidant, we construct a $(3, \mc P)$--polygonal line $\pgot$ from $z_1$ to $z_2$ such that 
    \begin{align*}
        d(m, \pgot)\geq R \qquad \text{and} \qquad \norm{\pgot}\leq CR.
    \end{align*}

    \begin{figure}
        \centering
        \includegraphics[width=0.6\linewidth]{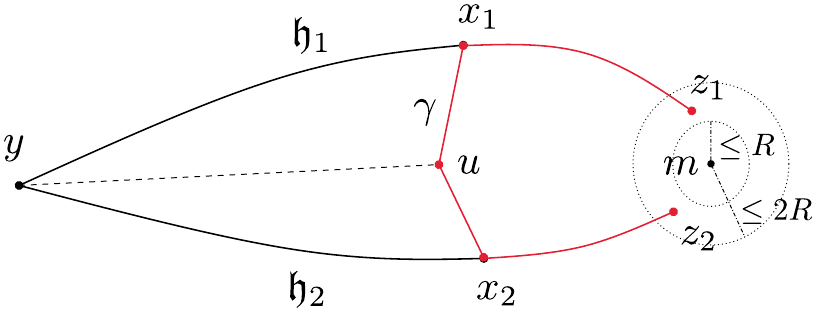}
        \caption{The red path avoids the $R$--ball around $m$.}
        \label{fig:media-avoidant}
    \end{figure}

    If $d(y, z_1)\leq 5R$ or $d(y, z_2)\leq 5R$, then the path $\pgot = \spath^ {-1}\ast \spath'$ satisfies all the criteria (by the triangle inequality, the other side has length at most $13R$). We may therefore assume that $d(y, z_i)\geq 5R$ for both $i = 1$ and $i = 2$.
    Then, for $i=1, 2$, let $x_i\in \spath_i$ be points with $d(z_i, x_i) = 5R$. Let $u$ be the median of $(x_1, x_2, y)$ and let $\gamma$ be a geodesic from $x_1$ to $x_2$ which passes through $u$. This is depicted in Figure~\ref{fig:media-avoidant}. Note that $\psub{\gamma}{x_1, u}$ lies on a geodesic from $z_1$ to $y$ and thus for any point $a\in \psub{\gamma}{x_1, u}$ we have that 
    \begin{align*}
        d(a, m)\geq d(a, z_1) - d(z_1, m) \geq d(x_1, z_1) - 4R \geq R.
    \end{align*}
    Hence $d(\psub{\gamma}{x_1, u}, m)\geq R$. Analogously, one can show that $d(\psub{\gamma}{u, x_2} ,m)\geq R$. Thus, the path
    \begin{align*}
        \pgot = \psub{\spath^{-1}}{z_1, x_1}\ast \gamma \ast \psub{\spath'}{x_2, z_2}
    \end{align*}
    satisfies that $d(\pgot, m)\geq R$. Furthermore, we have that
    \begin{align*}
        \norm{p}\leq d(z_1, x_1) + d(x_1, x_2) + d(x_2, z_2)\leq 5R + 18R + 5R\leq 28 R.
    \end{align*}
    Thus, $\pgot$ satisfies the desired properties, which concludes the proof.
\end{proof}

A combination of the two results above and Proposition~\ref{prop:bounded-avoidant-implies-navigable} yields the following.
\begin{proposition}\label{prop:median-implies-navigable}
    Let $X$ be a median geodesic metric space. The path system induced by all the geodesics in $X$ is $(C, k)$--navigable, for $C = 28$ and $k=3$.
\end{proposition}

As in the geodesic combing case, Proposition~\ref{prop:median-implies-navigable} has the following consequence. 

\begin{corollary}\label{cor:median-implies-morse-dichotomous}
    Let $X$ be a median geodesic space. Then $X$ is Morse dichotomous.
\end{corollary}

\section{The generalized contraction space}\label{sec:generalised_contraction}
The rough idea of the contraction space is to collapse every pair of points that are not along a  `negatively curved' direction. For technical reasons, it is convenient to assume that our geodesic metric space is, in fact, a connected graph. 

\textbf{Setup:} For the rest of the section, $X = (V, E)$ is a graph with induced path metric $d$ and $\mc P$ is an undirected path system on $X$.

In this section, we aim to construct a generalized version of the contraction space. More precisely, we describe a construction analogous to that of the contraction space in \cite{Zbinden:hyperbolic}, but with the paths in the path system $\mc P$ playing the part of the geodesics.

\subsection{Midthin, anti-contracting and consequences}\label{subsec:allowed-contraction-gauges}

As in \cite{Zbinden:hyperbolic}, in order to define what it means to `not be along a negatively curved direction', we need the notion of contraction gauge and midthinness. 

\begin{definition}[Contraction gauge]
    A \emph{contraction gauge} is a non-decreasing function $K: \R_{\geq 0}\to \R_{\geq 0}\cup\{ \infty\}$. Given a space $X$ endowed with a path system $\mc P$, an integer $n\geq 7$ and a contraction gauge $K$, we call $\mc K = (K, n, \mc P)$ a \emph{contraction triple}. We say that $K$, respectively $\mc K$, is \emph{full}, if $\infty$ is not in the image of $K$.
\end{definition}

Often, we only consider contraction gauges that are larger than a specific minimal function $K_0$, as this allows us to deduce nice properties of the contraction space such as hyperbolicity. Specifically, if we say that a result holds \emph{for a large enough contraction gauge $K$} or a \emph{large enough contraction triple} $\mc K = (K, n , \mc P)$, then we mean that the result holds if $K\geq K_0$, for some full contraction gauge $K$ and large enough constant $K_0$.   

\begin{definition}[Midthin]\label{def:midthin}
Let $\mc K = (K, n, \mc P)$ be a contraction triple. A special path $\spath\in \mc P$ is \emph{$(r; n, \mc P)$-midthin} if every $(n, \mc P)$-polygonal line $\pgot$ from $\spath^+$ to $\spath^-$ intersects the closed $r$-ball around the midpoint $m_\spath$. If in addition $\abs{\spath}\geq K(r)$, then we say that $\spath$ is $\mc K$--midthin and we call $r$ a \emph{neck-size} of $\spath$. 
\end{definition}


\begin{definition}[Anti-contracting]\label{def:anti-contracting}
    Let $\mc K$ be a contraction triple. We say that a special path $\spath\in \mc P$ is \emph{$\mc K$--anti-contracting} if none of its subsegments is $\mc K$--midthin. 
    
    We say that a pair of vertices $(x, y)\in V(X)\times V(X)$ is \emph{$\mc K$-anti-contracting} if there exists a special path $\spath\in \mc P$ from $x$ to $y$ which is $\mc K$--anti-contracting. We denote by $A_{\mc K} \subset V(X)\times V(X)$ the set of all $\mc K$-anti-contracting pairs.
\end{definition}

Note that the order of the quantifiers differs from the one in \cite[Definition~2.2]{Zbinden:hyperbolic}. Since $\mc P$ is undirected, $(x, y)\in \mc A_{\mc K}$ implies that $(y, x)\in \mc A_{\mc K}$.

An immediate consequence of the definition of anti-contraction is the following.

\begin{lemma}\label{lemma:diameter-of-anti-contracting-paths}
    Let $\mc K$ be a contraction triple. If $\spath \in \mc P$ is $\mc K$--anti-contracting, then all of its subsegments are $\mc K$--anti-contracting as well. 
\end{lemma}

The following two lemmas record some consequences for large enough contraction triples. They are depicted in Figure~\ref{fig:non-allowed}.

\begin{figure}
     \centering
     \begin{subfigure}[b]{0.45\textwidth}
         \centering
         \includegraphics[width=\textwidth]{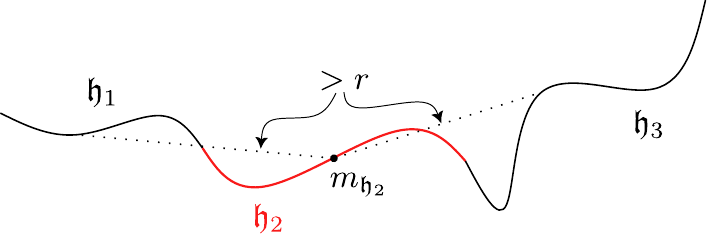}
     \end{subfigure}
     \hfill
     \begin{subfigure}[b]{0.45\textwidth}
         \centering
         \includegraphics[width=\textwidth]{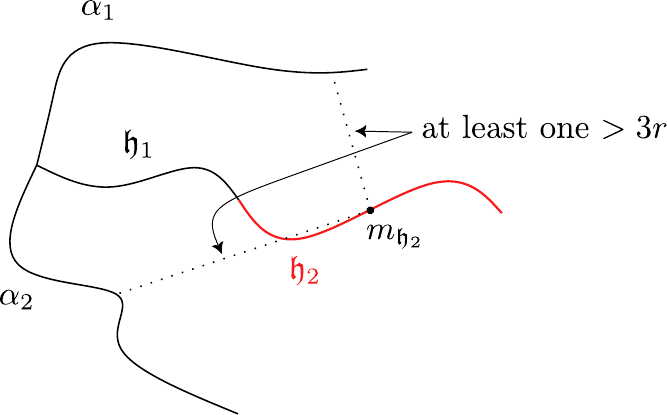}
     \end{subfigure}
     \hfill
        \caption{Depictions of Lemma~\ref{lemma:continued-special-path-distance} (left) and Lemma~\ref{lemma:T-lemma} (right). The red path $\spath_2$ is $\mc K$--midthin with neck-size $r$.}
     \label{fig:non-allowed}
\end{figure}

\begin{lemma}\label{lemma:continued-special-path-distance}
    Let $\mc K =(K, n, \mc P)$ be a large enough contraction triple. Let $\spath = \spath_1\ast \spath_2\ast \spath_3\in \mc P$ be a special path. If $\spath_2$ is $\mc K$--midthin with neck-size $r$, then $d(\spath_i, m_{\spath_2})> r$ for $i = 1, 3$. Similarly, $d(\psub{\spath}{\spath_1^-, \spath_1^+}, m_{\spath_{2}}) > r$.
\end{lemma}

\begin{proof}
    This holds if $K(r) > 2D_{\mc P}(r)$ because special paths are uniform quasi-geodesics.
\end{proof}

\begin{lemma}\label{lemma:T-lemma}
    Let $\mc K = (K, n, \mc P)$ be a large enough contraction triple. Let $\spath = \spath_1\ast \spath_2$ and $\alpha = \alpha_1\ast \alpha_2$ be special paths such that $\spath_2$ is $\mc K$--midthin and $\spath_1^- = \alpha_1^+ = \alpha_2^-$. Then either for $i=1$ or for $i=2$ we have $d(\alpha_i, m_{\spath_2})  > 3r$, where $r$ is a neck-size of $\spath_2$. 
\end{lemma}

\begin{proof}
    Assume by contradiction that for both $i= 1$ and $i = 2$ there are points $x_i\in \alpha_i$ with $d(x_i, m_{\spath_2})\leq 3r$. Since $\alpha_1^+$ lies between $x_1$ and $x_2$ on $\alpha$ and $d(x_1, x_2)\leq 6r$, we have $d(x_1, \alpha_1^+)\leq D_{\mc P}(6r)$. In particular, $d(m_{\spath_2}, \alpha^+)\leq 3r+D_{\mc P}(6r)$, a contradiction, if $K(r) > D_{\mc P}(3r+D_{\mc P}(6r))$.
\end{proof}

\begin{remark}\label{rem:lemmas-for-not-midthin}
    For the Lemmas~\ref{lemma:continued-special-path-distance} and \ref{lemma:T-lemma} it suffices to assume that $\abs{\spath_2}\geq K(r)$ instead of $\spath_2$ being $\mc K$--midthin with neck-size $r$, but we have stated them as above as this is the most frequent setting in which they will be used.
\end{remark}

\subsection{The construction of the generalised contraction space}\label{subsec:construction}

In what follows,  $\mc K = (K, n, \mc P)$ denotes a contraction triple.

\begin{definition}[The $\mc K$-contraction space]\label{def:K-contraction-space}
The $\mc K$-contraction space of $X$, denoted by $\hat{X}_{\mc K}$, is obtained from $X$ by adding an edge between every pair of $\mc K$-anti-contracting vertices $(x, y)\in \mc A_{\mc K}$.
\end{definition}

With this definition, the $\mc K$--contraction space $\hat{X}_{\mc K}$ is a graph with the same vertex set as $X$. We denote the induced path metric in $\hat{X}_{\mc K}$ by $\hat{d}_{\mc K}$. As a general rule, we use the hat notation for objects considered with respect to the metric $\hat{d}_{\mc K}$, for instance we denote $D$-neighbourhoods in $\hat{X}_{\mc K}$ by $\hat{\mc{N}}_D(\cdot, \mc K)$ and the diameter with respect to $\hat{d}_{\mc K}$ by $\widehat{\diam}_{\mc K}(\cdot)$. There is a natural distance-non-increasing inclusion $\iota_{\mc K}:X\hookrightarrow \hat{X}_{\mc K}$ and we identify $X$ with its image in $\hat{X}_{\mc K}$. For any path $\gamma: I\to X$ we denote the composition $\iota_{\mc K}\circ \gamma : I\to \hat{X}_{\mc K}$ by $\hat{\gamma}_{\mc K}$. While the images of $\gamma$ and $\hat{\gamma}_{\mc K}$ are equal, the important distinction is that if $\gamma$ is a (quasi-)geodesic, $\hat{\gamma}_{\mc K}$ need not be one. For any pair of points $x, y\in \hat{X}_{\mc K}$, $[x, y]_{\mc K}$ denotes a choice of geodesic (with respect to $\hat{d}_{\mc K}$) from $x$ to $y$. Henceforth, unless otherwise stated, the closest points, neighbourhoods, geodesics etc are always considered with respect to the metric $d$.

 We record three facts that follow immediately from the definition of anti-contraction.

\begin{lemma}\label{lemma:small-implies-anti-contracting}
      Assume that $\mc K$ is large enough. If $\spath\in \mc P$ is a special path with $d(\spath^-, \spath^+)\leq 1$, then $\spath$ is $\mc K$--anti-contracting. In particular, if $x, y\in V(X)$ and $\hat{d}_{\mc K}(x, y) \leq 1$, then $(x, y)$ is $\mc K$--anti-contracting. 
\end{lemma}
\begin{proof}
   If $x,y\in V(X)$ and $\hat{d}_{\mc K}(x, y) \leq 1$, then either there is an edge between them in $\hat{X}_{\mc K}$, or there is not. In the first case, the pair is anti-contracting by definition. In the second case, since $x,y\in V(X)$ we see that $d(x,y)\leq 1$. Thus, the pair $(x,y)$ is $\mc K$--anti-contracting as long as $K(r) > D_{\mc P}(1)$ for all $r\geq 0$.
\end{proof}

\begin{lemma}\label{lem:diam-anti-contracting}
   If $\spath\in \mc P$ is a $\mc K$--anti-contracting special path, then the set of vertices lying on $\spath$ forms a clique in $\hat{X}_{\mc K}$ and hence $\widehat{\diam}_{\mc K}(\spath)\leq 3$.
\end{lemma}
\begin{proof}
    This follows from Lemma~\ref{lemma:diameter-of-anti-contracting-paths}.
\end{proof}

\subsection{Hyperbolicity of the contraction space}\label{sec:hyperbolicity}

We want to use the guessing geodesics technique (see Theorem~\ref{thm:guessing_geodesics}) to prove that the contraction space is hyperbolic. In \cite{Zbinden:hyperbolic}, this is done as follows: the guessed paths are geodesics in $X$ and if we have a geodesic triangle $\Delta = (\gamma, \pgot_1,\pgot_2)$ in $X$ and a point $m$ on $\gamma$, we can show that $\hat{d}_K(m, x)\leq 1$, where $x$ is the closest point on $\pgot_1\cup \pgot_2$ to $m$. The proof that $\hat{d}_K(m, x)\leq 1$ has two main ingredients: firstly, if it was not, then $[m, x]$ would have to contain a midthin (called thin in \cite{Zbinden:hyperbolic}) subsegment. The geodesic $[m, x]$ `cuts' $\Delta$ into two quadrangles $\mc Q_1$ and $\mc Q_2$, so the midpoint $m'$ of the midthin segment of $[m, x]$ has to come close to points $q_1, q_2$ on $\mc Q_1$ and $\mc Q_2$. But, since $x$ was specifically chosen as the closest point, the distance between $m'$ and $\pgot_1\cup \pgot_2$ has to be at least $d(m', x)$, showing that $q_1$ and $q_2$ must lie on $\gamma$. The triangle inequality then produces a contradiction. Now, replacing geodesics by special paths we lose $d(m', \pgot_1 \cup \pgot_2) \geq d(m', x)$, in fact, at least a priori, $m'$ might even lie on $\pgot_1\cup \pgot_2$. 

To remedy this, we introduce a function $\Phi$ that we call the \emph{triangular pull} (see \eqref{eq:def-of-triangular-pull}) whose minimizer $x\in \pgot_1 \cup \pgot_2$ has the following property: even if paths in $\mc P$ are only quasi-geodesics instead of geodesics, no point $m'$ on $[m, x]_{\mc P}$ can be close to both $\gamma$ and $\pgot_1 \cup \pgot_2$. Hence, either both $q_1$ and $q_2$ lie on $\gamma$ or they both lie on $\pgot_1\cup \pgot_2$. The former still cannot happen, as it violates the triangle inequality. To ensure the latter cannot happen, we `cut off the tip' of $\pgot_1 \cup \pgot_2$ before defining $x$ as the minimizer of $\Phi$ on $\pgot_1\cup \pgot_2$. Then $q_1, q_2$ cannot both be close to $\pgot_1\cup \pgot_2$, as such points do not exist any more. 

Of course, the cutting off of the tip of $\pgot_1 \cup \pgot_2$ needs to be done carefully: If too much is cut off, then all the points left on $\pgot_1\cup \pgot_2$ might be far from $m$. If, on the other hand, not enough is cut off, one cannot conclude that $q_1$ and $q_2$ both lying on the uncut part of $\pgot_1\cup \pgot_2$ yield a contradiction. Then, to connect the stumps where the tip is cut off, with a special path, say $\alpha$, and to deal with the possibility of $q_1$ or $q_2$ lying on $\alpha$.

The proposition below uses the proof strategy described above to show that in a triangle of special paths where one side has a midthin subsegment with midpoint $m$, there exists a point on the other two sides which is close to $m$ with respect to both $d$ and $d_{\mc K}$. The proposition is then later used as the main ingredient of the proof of hyperbolicity of the contraction space (Theorem~\ref{thm:contraction_space_hyp_for_graphs}).

\begin{proposition}\label{lem:technical-properties-of-contraction-space}
    Let $\mc K =(K, n, \mc P)$ be a large enough contraction triple. Let $\gamma\in \mc P$ be a special path whose endpoints are vertices and let $\tilde{\gamma}\subset \gamma$ be a $\mc K$--midthin subsegment of $\gamma$ with neck-size $r$ and midpoint $m$. Let $\pgot = \pgot_1\ast \pgot_2$ be a $(2, \mc{P})$--polygonal line from $\gamma^+$ to $\gamma^-$. Then there exists a vertex $x\in \pgot_1\ast \pgot_2$ with $\hat{d}_{\mc K}(x, \tilde{\gamma})\leq 1$ and $d(x, m)\leq (2r+3)^2$.
\end{proposition}

Before we begin the proof, we want to recall a simple but useful fact following from the triangle inequality.

    \begin{lemma}
        \label{claim:good-fraction}
        There exists a constant $C_\ast$ depending only on the quasi-geodesic constants of $\mc P$ such that the following holds: If $\beta\in \mc P$ is a special path from $a$ to $b$, $c$ is a point on $\beta$ and $c'$ is a point at distance at most $d(a, b)$ from $c$, then for any point $m$, we have 
    \begin{align}\label{eq:fractional-bound}
        d(c', m)\leq C_\ast\max\{1, d(a, b), d(m, b)\}.
    \end{align}
    \end{lemma}
    \begin{figure}
        \centering
        \includegraphics[width=0.5\linewidth]{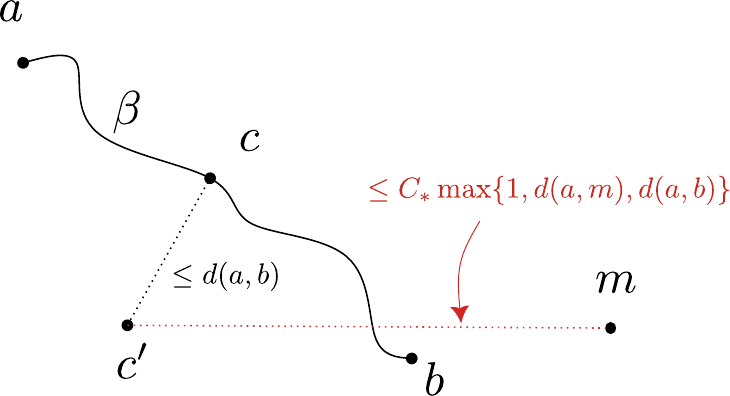}
        \caption{Upper bounding the distance $d(c', m)$.}
        \label{fig:constellation}
    \end{figure}        
    \begin{proof}
        Let $C\geq 1$ be a constant such that $\mc P$ is a $C$--quasi-geodesic path system. Since $c$ lies on a special path and hence is a $C$--quasi-geodesic from $a$ to $b$, we have $d(a,c)\leq C^2d(a, b) + 2C^3$.  This is depicted in Figure~\ref{fig:constellation}. The statement follows by the triangle inequality for example for $C_\ast = C^2 + 2C^3 + 2$.
    \end{proof}

\noindent{\textit{Proof of Proposition~\ref{lem:technical-properties-of-contraction-space}}}
    Let $C_0\geq C_*$ be a large enough constant and let $f, h, g : \R_{\geq 0}\to \R_{\geq 0}$ be functions. Morally, we should think of them as $K^{-1}(s)\ll f(s) \ll h(s) \ll g(s) \ll s$, where $K$ is the function from the contraction triple $\mc K$. In practice, there are specific inequalities that need to be satisfied, which are technical and not very illuminating. We have added them at the end of the proof for completeness.

    \smallskip
    
    Next, define the \emph{triangular pull} $\Phi : X\times X\to \R_{\geq 0}$ as follows:
    \begin{align}\label{eq:def-of-triangular-pull}
        \Phi(x, y) =  \max\{d(x, y), 1\} \cdot \max\{1, d(x, y), d(m, y)\}.
    \end{align}

    For any vertex $x$ on $\pgot$, we denote by $\Psi(x)$ a (choice of) vertex $y\in \gamma$ which minimizes $\Phi(x, y)$ among vertices on $\gamma$. 

    \begin{claim}\label{claim:dist1-enough}
        Let $x\in \pgot$ be a vertex with $\Phi(x, \Psi(x))\leq (2r+2)^2$. We have $d(x, m)\leq (2r+3)^2$.
    \end{claim}

    \textit{Proof of Claim \ref{claim:dist1-enough}.}
    By the definition of $\Phi$, we have $d(x, \Psi(x))\leq \sqrt{\Phi(x, \Psi(x))}\leq 2r+2$. We also have $d(\Psi(x), m)\leq \Phi(x,\Psi(x))\leq (2r+2)^2$. The statement follows from the triangle inequality.
    \hfill$\blacksquare$

    \smallskip
 
    We want to find subpaths $\qgot_i$ of $\pgot_i$, that avoid so called \emph{crosses}, which we now define. 
    
    Let $x\in \pgot$ be a point such that $\Phi(x, \Psi(x))\leq (2r+2)^2$ and let $\spath\in \mc P$ be a special path from $x$ to $\Psi(x)$. Let $z_1\in \pgot_1$, $z_2\in \pgot_2$, $m_x\in \spath$, and $\ell = d(m_x, \Psi(x))$. We say that the data $(z_1, z_2, x, \spath, m_x, \ell)$ is a \emph{cross} and $(z_1, z_2)$ are \emph{a pair of cross points} if
    \begin{align}
        d(m_x, z_i)&\leq f(\ell) \qquad \text{for $i = 1, 2$, and}&\label{eq:close-to-side}\\
        d(m_x, \gamma)&\geq g(\ell).\label{eq:far-from-gamma}
    \end{align}

    \smallskip

    As mentioned, we want to find subpaths $\qgot_i$ of $\pgot_i$, and a path $\alpha$ between them, so that $\qgot_i$ avoid crosses.    
    If there does not exist a cross we define $\qgot_1 = \pgot_1$ and $\qgot_2 = \pgot_2$ and $\alpha$ as the trivial path from $\pgot_1^+$ to itself. In this case, we define $\ell  = -1$. Otherwise, let $(\pgot_1(t_1), \pgot_2(t_2), x, \spath, m_x, \ell)$ be a cross which minimizes $t_1$ (and which maximizes $t_2$ among crosses that minimize $t_1$). Finally, let $\qgot_1 = \pref{\pgot_1}{\pgot_1(t_1)}$, let $\qgot_2 = \suf{\pgot_2}{\pgot(t_2)}$ and let $\alpha\in \mc P$ be a special path from $\qgot_1^+$ to $\qgot_2^-$. Define $\qgot = \qgot_1\ast \alpha\ast \qgot_2$, $z_1 = \qgot_1^+$ and $z_2 = \qgot_2^-$. This is depicted in Figure~\ref{fig:constellation}.

    Note that the cross we found must be minimal: 

    \begin{claim}\label{claim:no-cross-points}
        If $(v_1, v_2)$ is a pair of cross points with $v_1\in \qgot_1$ and $v_2\in \qgot_2$, then $v_1 = z_1$ and $v_2 = z_2$. 
    \end{claim}

    \textit{Proof of Claim \ref{claim:no-cross-points}.}
        This is a consequence of the definition of $z_1$ and $z_2$.
    \hfill$\blacksquare$

     \smallskip

    \begin{claim}\label{claim:cross-facts2}
     We have that $d(\alpha, \gamma) \geq h(\ell)$.
    \end{claim}
    
    \textit{Proof of Claim \ref{claim:cross-facts2}.}
    If there was no cross, then $h(\ell)  = 0$ by definition of $h$ and $\ell$ and the statement would follow. For $i = 1, 2$, by the triangle inequality \eqref{eq:far-from-gamma} and \eqref{eq:close-to-side}, we have $d(\alpha, \gamma)\geq  d(m_x, \gamma) - d(z_i, m_x)  - \diam(\alpha) \geq g(\ell) - f(\ell)  -D_{\mc P}(2f(\ell)) \geq h(\ell)$, where the last step follows from $g\gg h\gg f$ or, more specifically, from \ref{2a}.
    \hfill $\blacksquare$

    \smallskip
    
    Let $\qgot_1 = \qgot_1''\ast \qgot_1'$ and $\qgot_2 = \qgot_2 ' \ast \qgot_2''$ be subdivisions of $\qgot_1$ and $\qgot_2$, respectively, such that $d(\alpha, \qgot_1'')\geq h(\ell)$ and $d(\alpha, \qgot_2'')\geq h(\ell)$ and such that the domain length of $\qgot_1'$, respectively $\qgot_2'$, are minimal amongst such possible subdivisions. By Claim~\ref{claim:cross-facts2} and continuity, such subdivisions exist and $d(\qgot_1''^+, \alpha) = h(\ell) = d(\alpha, \qgot_2''^-)$. Since the endpoints of $\gamma$ are vertices, both $\qgot_2''$ and $\qgot_1''$ contain at least one vertex.

    \smallskip 

    \begin{claim}\label{claim:diameter-alpha}
    The path $\qgot_1'\ast \alpha\ast \qgot_2'$  has diameter at most $3D_{\mc P}(h(\ell))$.  
    \end{claim}

    \textit{Proof of Claim \ref{claim:diameter-alpha}.}
        By \eqref{eq:close-to-side}, $d(z_1, z_2)\leq 2f(\ell)$ and hence $\diam(\alpha)\leq D_{\mc P}(h(\ell))$ since $f\ll h$ or, more specifically, by \eqref{1}. Since $d(\qgot_1'^-, \alpha) = h(\ell)$, we have $\diam({\qgot_1'}^{-})\leq D_{\mc P}(\qgot_1^-, z_1)\leq D_{\mc P}(h(\ell))$.  Similarly, $\diam(\qgot_2')\leq D_{\mc P}(h(\ell))$. The claim follows.
    \hfill$\blacksquare$

    \smallskip 
    
    \begin{claim}\label{claim:cross-facts}
        If $d(\qgot_1'\ast\alpha\ast \qgot_2', m)\leq r$, then $d(\qgot_1'^-, m)\leq 2r$.
    \end{claim}

    \textit{Proof of Claim \ref{claim:cross-facts}.}
        If there was no cross, the claim would hold trivially. Assume there was a cross and $d(y, m)\leq r$ for some point $y$ on $\qgot_1'\ast\alpha\ast \qgot_2'$. We have on one hand by Claim~\ref{claim:diameter-alpha},
        \begin{align*}
            d(m_x, m)&\leq d(m_x, z_1) + 3D_{\mc P}(h(\ell)) + d(y, m)\\
            &\leq f(\ell) + 3D_{\mc P}(h(\ell)) + r.
        \end{align*}
        Similarly, we obtain $d(\qgot_1'^-, m)\leq 3D_{\mc P}(h(\ell))+r$. On the other hand, by \eqref{eq:far-from-gamma}, we have $g(\ell) \leq d(m_x, \gamma)\leq d(m_x, m)$. We have $f(\ell) + 3D_{\mc P}(h(\ell))\leq g(\ell)/2$ by \eqref{2a}, hence $g(\ell)\leq 2r$. Using $3D_{\mc P}(h(\ell))\leq g(\ell)/2$ once more, we obtain $d(\qgot_1'^-, m)\leq 2r$ as desired.
    \hfill$\blacksquare$

    \medskip
    
    Let $x_0$ be a vertex in $\qgot_1''\cup \qgot_2''$ that minimizes $\Phi(\cdot, \Psi(\cdot))$ among vertices on $\qgot_1''\cup \qgot_2''$. Without loss of generality, we may assume that $x_0$ lies on $\qgot_2''$.

    \begin{claim}\label{claim:still-point-close}
        We have $\Phi(x_0, \Psi(x_0))\leq (2r+2)^2$.
    \end{claim}

    \textit{Proof of Claim \ref{claim:still-point-close}.}
    Let $\tilde{m}$ be a vertex on $\gamma$ at distance at most $1$ from $m$. For any vertex $x'$ on $\qgot_1''\cup \qgot_2''$ we have $\Phi(x', \Psi(x'))\leq \Phi(x', \tilde{m}) \leq (d(x', m) + 1)^2$ because $\Psi(x')$ is a minimum. Hence it suffices to show that there exists a vertex $x'\in \qgot_1'' \cup \qgot_2''$ with $d(x', m)\leq 2r+1$, as it implies $\Phi(x_0, \Psi(x_0))\leq \Phi(x', \tilde{m})\leq (2r+2)^2$.

    Since $\tilde{\gamma}$ is $\mc K$--midthin with neck-size $r$, there has to exist a point $\tilde{x}$ on the $(5, \mc P)$--polygonal line $\suf{\gamma}{\tilde{\gamma}^+}\ast \qgot_1 \ast \alpha \ast \qgot_2 \ast \pref{\gamma}{\tilde{\gamma}^-}$ with $d(\tilde{x}, m)\leq r$.

    If $\tilde{x}$ lies on $\qgot_2''$ respectively $\qgot_1''$, we can choose $x'$ as a vertex on $\qgot_2''$ respectively $\qgot_1''$ at distance at most $1$ from $\tilde{x}$ and get $d(x', m)\leq r+1$.

    If $\tilde{x}$ lies on $\qgot_1'\cup \alpha \cup \qgot_2'$, we have by Claim~\ref{claim:cross-facts} that $d(m, \qgot_1'^-)\leq 2r$. So we can choose $x'$ as a vertex on $\qgot_1''$ at distance at most $1$ from $\qgot_1'^-$ and obtain $d(x', m)\leq 2r+1$.

    Lastly, $\tilde{x}$ cannot lie on $\suf{\gamma}{\tilde{\gamma}^+}$ or $\pref{\gamma}{\tilde{\gamma}^-}$ by Lemma~\ref{lemma:continued-special-path-distance}.
    \hfill $\blacksquare$

    \smallskip

    \begin{figure}
        \centering
        \includegraphics[width=1.0\linewidth]{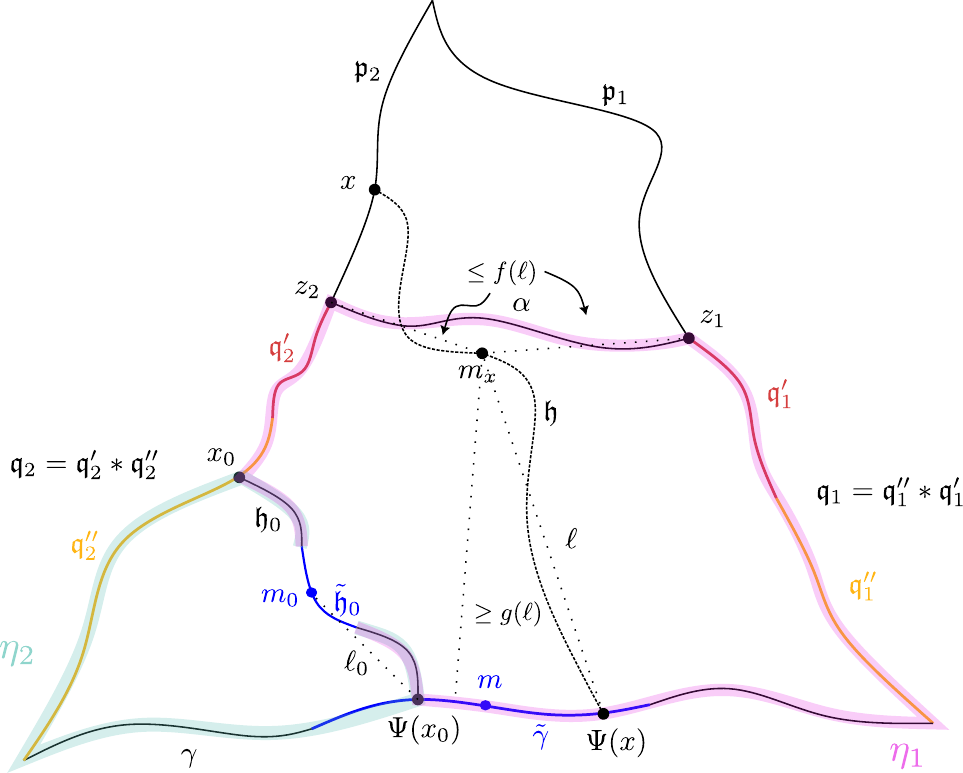}
        \caption{Depiction of the constructed paths. The point $x$ might also lie on $\pgot_1$ instead of $\pgot_2$.}
        \label{fig:two-quads}
    \end{figure}
    \begin{table}[]
\begin{tabular}{l|l|l|l|l|l|l}
                         & A: $\psub{\gamma}{\Psi(x_0), \gamma^+}$ & B: $\qgot_1''$ & C: $\qgot_1'$      & D: $\alpha$       & E: $\qgot_2'$             & F: $\pref{\qgot_2''}{x_0}$ \\ \hline
1: $\pref{\gamma}{\Psi(x_0)}$ & Lemma~\ref{lemma:T-lemma}     & $\Phi$    & $\Phi \oplus$ & $\Phi\oplus$ & $\Phi \oplus$        & $\Phi$                \\ \hline
2: $\suf{\qgot_2''}{x_0}$    & $\Phi$                   & cross     & cross         & $d(x_0, z_2)$  & Lemma~\ref{lemma:T-lemma}& Lemma~\ref{lemma:T-lemma}
\end{tabular}
\caption{Case distinction: how do we reach a contradiction in case different parts of $\eta_1$ and $\eta_2$ come close to $m_0$.}
\label{table:1}
\end{table}

    If $\hat{d}_{\mc K}(x_0, \Psi(x_0)) \leq 1$, then, by Claim~\ref{claim:dist1-enough}, $x_0$ satisfies all the conclusions of the statement and hence we are done ($x_0$ satisfies the assumptions of Claim~\ref{claim:dist1-enough} by Claim~\ref{claim:still-point-close}). So from now on we assume that $\hat{d}_{\mc K}(x_0, \Psi(x_0)) > 1$, and show that this leads to a contradiction.

    \smallskip

    Let $\spath_0 \in \mc P$ be a special path from $x_0$ to $\Psi(x_0)$. Having $\hat{d}_{\mc K}(x_0,\Psi(x_0)) > 1$ implies that $\spath_0$ is not $\mc K$--anti-contracting and hence has a $\mc K$--midthin subpath $\tilde{\spath}_0$ with midpoint $m_0$ and neck-size $r'$ for some $r'\geq 0$. Let $\ell_0 = d(m_0, \Psi(x_0))$. This is depicted in Figure~\ref{fig:two-quads}. Since $\mc K$ is large enough, we may assume $r'\leq f(\ell_0)\leq g(\ell_0)$ and $r'\leq \ell_0$.

    Using that $\spath_0$ is $(r', n, \mc P)$--midthin on the $(6, \mc P)$--polygonal lines $\eta_1  =  \suf{\spath_0}{\tilde{\spath}_0^+}\ast\psub{\gamma}{\Psi(x_0), \gamma^+}\ast \pref{\qgot}{x_0} \ast \pref{\spath_0}{\tilde{\spath}_0^-}$ and $\eta_2 = \suf{\spath_0}{\tilde{\spath}_0^+}\ast(\suf{\qgot}{x_0}\ast \pref{\gamma}{\Psi(x_0)})^{-1} \ast \pref{\spath_0}{\tilde{\spath}_0^-}$ (shaded pink and turquoise in Figure~\ref{fig:two-quads}) yields points $v_1, v_2$ on $\eta_1$ and $\eta_2$ which are in the $r'$--neighbourhood of $m_0$. Lemma~\ref{lemma:continued-special-path-distance} shows that $v_1$ and $v_2$ do not lie on $\spath_0 - \tilde{\spath}_0$. Hence, $m_0$ either has to be in the $r'$ and hence $g(\ell_0)$--neighbourhood of $\psub{\gamma}{\Psi(x_0), \gamma^+}$ or, if it does not, then $v_1$ has to lie on one of the $5$ special paths $\qgot_1'', \qgot_1', \alpha, \qgot_2'$ or $\pref{\qgot_2''}{x_0}$ (we denote these cases as Cases $A, B, C$ etc. ). Moreover, either $m_0$ has distance at most $g(\ell_0)$ to $\pref{\gamma}{\Psi(x_0)}$ or, if it does not, $v_2$ lies on $\suf{\qgot_2''}{x_0}$ (we denote these as Case 1 and Case 2, respectively).

    We will now perform a case distinction on how to get a contradiction in each of the twelve cases, see Table~\ref{table:1}. In the cases where it says `cross', we will show that $(v_1, v_2)$ are cross points on $\qgot_1$ and $\qgot_2$. In the case where it says $\Phi$ or $\Phi \oplus$, we will contradict the minimality of $\Phi(x_0, \Psi(x_0))$ and in the cases where it says Lemma~\ref{lemma:T-lemma}, we will find a contradiction to Lemma~\ref{lemma:T-lemma}. Finally, in the case 2D where it says $d(x_0, z_2)$, we will find a contradiction to the triangle inequality based on lower bounds on $d(x_0, z_2)$.

    \smallskip
    
    \textbf{Case `$\Phi$' and `$\Phi\oplus$', that is, Case 1B, 1C, 1D, 1E, 1F and 2A:} Recall that, by assumption, $\hat{d}_{\mc K} (x_0, \Psi(x_0)) > 1$. Hence, since $\mc K$ is large enough and in particular large enough compared to $C_0$, we can assume $d(x_0, \Psi(x_0)) \geq C_0$.
    
    Now observe the following: in each of these cases, by assumption, we have a vertex $w$ on $\gamma$ (namely a vertex at distance at most 1 from either $v_1$ (Case 2A) or $v_2$ (Case 1B, 1C, 1D, 1E, 1F)) which is at distance at most $g(\ell_0)+1\leq \frac{d(x_0, \Psi(x_0))}{2C_0}$ from $m_0$ (the latter inequality holds by \ref{3a}). Hence, by Lemma~\ref{claim:good-fraction}, we have 
    \begin{align}
        d(w, m)\leq C_0\max\{1,d(x_0, \Psi(x_0)), d(\Psi(x_0), m) \}.
    \end{align}

    If there exists a vertex $u$ in $\qgot_1''\cup \qgot_2''$ with $d(u, w) <  \frac{d(x_0, \Psi(x_0))}{C_0}$, then $\Phi(u, \Psi(u))\leq \Phi(u, w) < \Phi(x_0, \Psi(x_0))$, a contradiction to the minimality of $\Phi(x_0, \Psi(x_0))$.
    
    \begin{claim}\label{claim:alpha-close-to-m_0}
        There exist a vertex $u$ in $ \qgot_1'' \cup \qgot_2''$ with $d(u, m_0) < \frac{d(x_0, \Psi(x_0))}{2C_0}$.
    \end{claim}

    Since $d(m_0, w)\leq \frac{d(x_0, \Psi(x_0))}{2C_0}$ (see the definition of $w$), Claim~\ref{claim:alpha-close-to-m_0} concludes the proof.
    
    \textit{Proof of Claim~\ref{claim:alpha-close-to-m_0}.}
        If we are in Case $\Phi$ then this is true for $u = v_1$ (Cases 1B and 1F) or $u = v_2$ (Case 2A), as $u$ lies on $\qgot_1''\cup \qgot_2''$ and $d(u, m_0)\leq r'\leq g(\ell_0)\leq \frac{d(x_0, \Psi(x_0))}{2C_0}$. We now aim to prove this in the Case $\Phi\oplus$.
        
        Now, if we are in Case $\Phi\oplus$ (that is Case 1C, 1D or 1E), $v_1$ lies on $\qgot_1'\ast \alpha\ast \qgot_2'$. In particular, by the triangle inequality, Claim~\ref{claim:diameter-alpha} and $f\ll h\ll g$, more specifically \ref{2a}, we have 
        \begin{align*}
            d(m_0, m_x)&\leq d(m_0, v_1) + d(v_1, z_2) + d(z_2, m_x) \leq r' + 3D_{\mc P}(h(\ell)) + f(\ell) \leq r' + \frac{g(\ell)}{2}.
        \end{align*}
        By the triangle inequality and since $r'\leq \ell_0$, we have 
        \begin{align*}
            d(\Psi(x_0), m_x)& \leq d(\Phi(x_0), m_0) + d(m_0, m_x)\leq \ell_0 + r' + \frac{g(\ell)}{2} \leq 2\ell_0 + \frac{g(\ell)}{2}.
        \end{align*}

        Using \eqref{eq:far-from-gamma}, which implies that $g(\ell)\leq d(\Psi(x_0), m_x)$, we obtain
        \begin{align}\label{eq:bound-on-ell}
             \ell\leq g^{-1}(4\ell_0).
        \end{align}

        We choose $u$ as the first vertex on $\qgot_2''$. Using $r'\leq f(\ell_0)$, \eqref{eq:bound-on-ell}, Claim~\ref{claim:diameter-alpha} and \ref{3a}, we obtain 
        \begin{align*}
            d(u, m_0)&\leq d(m_0, v_1) + d(v_1, z_2) + d(z_2, u) \leq f(\ell_0) + 3D_{\mc P}(h(\ell)) + 1\leq \frac{d(x_0, \Psi(x_0))}{2C_0}.
        \end{align*} 
    \hfill$\blacksquare$

    \smallskip
    
    \textbf{Case `cross', that is, Case 2B and 2C:} In this case $(v_1, v_2, x_0, \spath_0, m_0, \ell_0)$ is a cross. Indeed, inequality \eqref{eq:far-from-gamma}, holds in all cases except 1A and 2X (for X $= $ A, B, C, D , E or F). In particular, it holds for Cases 2B and 2C. Moreover, we have $d(m_0, v_i)\leq r'\leq f(\ell_0)$ for $i = 1,2$, yielding \eqref{eq:close-to-side}. However, since $v_2\neq z_2$ (this follows from $v_2\in \qgot_2''$), $(v_1, v_2, x_0, \spath_0, m_0, \ell_0)$ cannot be a cross (see Claim~\ref{claim:no-cross-points}), a contradiction.

    \smallskip
    
    \textbf{Case `Lemma~\ref{lemma:T-lemma}', that is, 1A, 2E and 2F:} In Case 1A, both $v_1$ and $v_2$ lie on $\gamma$, in Cases 2E and 2F both $v_1$ and $v_2$ lie on $\qgot_2$. Lemma~\ref{lemma:T-lemma} leads to a contradiction.
    
    \smallskip
    
    \textbf{Case `$d(x_0, z_2)$', that is, Case 2D:} In this case, $v_1$ lies on $\alpha$ and $v_2$ lies on $\qgot_2''$. By the triangle inequality, we have 
        \begin{align*}
            d(v_2, z_2)&\leq d(v_2, m_0) + d(m_0, v_1) + d(v_1, z_2) \leq r' + r' + \diam(\alpha)\leq 2r'+ D_{\mc P}(2f(\ell)).
        \end{align*}
        Now, if $r' < D_{\mc P}(2f(\ell))$, then $d(v_2, z_2)< 3D_{\mc P}(2f(\ell)) $. This cannot be the case, since $v_2\in \qgot_2''$ and $d(\qgot_2'', z_2)\geq h(\ell)$ by definition. Hence $r'\geq D_{\mc P}(2f(\ell))$ and consequently $d(v_2, z_2)\leq 3r'$. This is a contradiction to Lemma~\ref{lemma:T-lemma}.

    \smallskip
    
    \textbf{Requirements on $f, g$ and $h$.}
    
    \begin{enumerate}
        \item $f, g$ and $h$ are non-decreasing functions tending to infinity that satisfy $2f(s)\leq h(s) \leq g(s)$ for all $s\geq 0$. Moreover, assume that $h(-1) = 0$; \label{1}
        \item \label{2a} $3D_{\mc P}(2f(s))\leq h(s)$ and $8D_{\mc P}(h(s))\leq g(s)$ for all $s\geq 0$;
        \item If $a$ and $b$ are points connected by a special path whose diameter is at least $s\geq C_0^2$, then $f(s) + 3D_{\mc P}(h(g^{-1}(4s)))+1 \leq g(s)+1\leq \frac{d(a, b)}{2C_0}$.\label{3a}
    \end{enumerate}

    For large enough $C_0$, we can find such functions: first choose $g$ so that the second inequality of \ref{3a} holds for $s\geq C_0^2$ and for $s\leq C_0^2$ define $g(s) = 10D_{\mc P}(10D_{\mc P}(s))$. Then we can choose $h$ small enough compared to $g$ and $f$ small enough compared to $h$ such that all the others hold.

We use the following proposition of \cite{DowdallDurhamLeiningerSisto:extensionsI}, which is a version of the Guessing Geodesic criterion \cite[Proposition~3.1]{Bowditch(2006)}.

\begin{theorem}[Guessing Geodesics, Proposition 2.2 of \cite{DowdallDurhamLeiningerSisto:extensionsI}] \label{thm:guessing_geodesics}Let $Y$ be a path-connected metric space, let $S\subset Y$ be an $R$-dense subset for some $R > 0$, and suppose there is $\delta\geq 0$ such that for all pairs $x, y\in Y$ there are rectifiable path-connected sets $\eta(x, y)\subset Y$ containing $x, y$ and satisfying:
\begin{enumerate}
    \item for all $x, y\in S$ with $d(x, y)\leq 3R$ we have $\diam(\eta(x, y))\leq \delta$,\label{guess_geo1}
    \item for all $x, y, z\in S$ we have $\eta (x, y)\subset \mc N_{\delta}(\eta(x, z)\cup \eta(z, y))$.\label{guess_geo2}
\end{enumerate}
Then there exists a constant $\delta'$ depending only on $\delta$ and $R$ such that $Y$ is $\delta'$-hyperbolic and the Hausdorff distance between $\eta(x, y)$ and any geodesic from $x$ to $y$ is at most $\delta'$.
\end{theorem}

\begin{theorem}\label{thm:contraction_space_hyp_for_graphs}
    Let $\mc K = (K, n, \mc P)$ be a large enough contraction triple. Then the contraction space $\hat{X}_{\mc K}$ is $\delta_0$-hyperbolic, for a fixed constant $\delta_0$ not depending on $X$ or $\mc K$. 
\end{theorem}

\begin{proof}
    We use the Guessing Geodesic criterion (Theorem~\ref{thm:guessing_geodesics}) for $S = V(X) = V(\hat{X}_{\mc K})$ and define $\eta(x, y)= \hat{\spath}_{\mc K}$ for a special path $\spath\in \mc P$ from $x$ to $y$. In the case where the pair $(x, y)$ is $\mc K$--anti-contracting, we require that the special path $\spath$ is $\mc K$--anti-contracting. The set $S$ is $\frac{1}{2}$-dense in $\hat{X}_\mc{K}$. 
    Clearly, $x, y\in \eta(x, y)$, and since special paths are rectifiable, so are the paths $\eta(x, y)$. We now prove \eqref{guess_geo1} and \eqref{guess_geo2} of Theorem~\ref{thm:guessing_geodesics} for $\delta=4$.

    \eqref{guess_geo1}: Let $x, y\in S$ with $\hat{d}_{\mc K}(x, y)\leq 3/2$. Since $\hat{X}_{\mc K}$ is a graph and $x, y$ are vertices, we know that $\hat{d}_{\mc K}(x, y)\leq 1$. By Lemma~\ref{lemma:small-implies-anti-contracting}, $(x, y)$ and hence $\eta(x, y)$ is $\mc K$--anti-contracting. Lemma~\ref{lem:diam-anti-contracting} shows that $\widehat{\diam}_{\mc K}(\eta(x, y))\leq 3 \leq \delta$.\\ 

    We next prove the following claim, which directly implies \eqref{guess_geo2} of Theorem~\ref{thm:guessing_geodesics}. 
    
    \begin{claim}\label{claim:triangles-in-contraction-space}
    Let $\gamma, \gamma',\gamma''\in \mc P$ be special paths whose endpoints are vertices and such that $\gamma\ast \gamma'\ast \gamma''$ is a $(3, \mc P)$--polygonal line and $\gamma^- = {\gamma''}^+$. We have that $\gamma\subset \hat{\mc N}_{4}(\gamma'\cup\gamma'')$.     
    \end{claim}
    
    \textit{Proof of Claim \ref{claim:triangles-in-contraction-space}.} Given $x = \gamma(t)$ an arbitrary vertex on $\gamma$, we will prove that $\hat{d}_{\mc K}(x, \gamma'\cup \gamma'')\leq 3$. If $\hat{d}_{\mc K}(x, \gamma^+) < 2$, then this trivially holds. If $\hat{d}_{\mc K}(x, \gamma^+)\geq 2$, then there exists at least one vertex $y\in \suf{\gamma}{x}$ with $\hat{d}_{\mc K}(x, y) = 2$. We choose $y = \gamma(s)$ as the first such vertex. By minimality of $s$, we have that $\hat{d}_{\mc K}(x, x')\leq 2$ for all $x'\in \isub{\gamma}{t, s}$. Since $\hat{d}_{\mc K}(x, y)>1$, $(x, y)$ is not $\mc K$--anti-contracting. In particular, there exists a subsegment $\spath$ of $\isub{\gamma}{ t, s}$ which is $\mc K$--midthin. Proposition~\ref{lem:technical-properties-of-contraction-space} yields a vertex $z\in \gamma'\ast \gamma''$ with $\hat{d}_{\mc K}(\spath, z)\leq 1$. By the triangle inequality, $\hat{d}_{\mc K}(z, x)\leq 3$, concluding the proof of the claim.
    \hfill$\blacksquare$
\end{proof}

\begin{remark}\label{rem:delta_0}
    We denote by $\delta_0$ the hyperbolicity constant obtained by applying Theorem~\ref{thm:guessing_geodesics} to $\delta = 4$. With this notation, the $\mc K$--contraction space is $\delta_0$-hyperbolic for large enough $\mc K$.
\end{remark}

\begin{remark}\label{rem:on_geodesic}
    The proof of Theorem~\ref{thm:contraction_space_hyp_for_graphs} shows that any special path $\spath\in \mc P$ between two vertices $x$ and $y$ has Hausdorff distance at most $[x, y]_{\mc K}$, unless $\spath$ is not $\mc K$--anti-contracting while $(x, y)$ is $\mc K$--anti-contracting. In particular, if $z$ lies on such a special path $\spath$, then $\hat{d}_{\mc K}(x, z)\leq \hat{d}_{\mc K}(x, y) + \delta_0$. 
\end{remark}

\subsection{Geometric connections}\label{sec:quasi-geodesics}

In this section, we highlight further geometric connections between a graph $X$ and its contraction space $\hat{X}_{\mc K}$. In particular, we show that if $\mc K$ is large enough and full, then $\hat{\gamma}_{\mc K}$, the image of a special path $\gamma\in \mc P$ in the contraction space $\hat{X}_{\mc K}$, is a parametrized quasi-geodesic if and only if $\gamma$ is $\mc P$--contracting. In the case where $\mc P$ is navigable, we can use Corollary~\ref{cor:morse-implies-mccontracting} to get the even stronger statement that $\hat{\gamma}_{\mc K}$ is a quasi-geodesic if and only if $\gamma$ is Morse.

We start by showing that midthinness is a consequence of $\mc P$--contraction.

\begin{lemma}\label{lem:p-contracting-implies-midthin}
     For every $C$ and $n$ there exists $r$ such that all special paths $\gamma\in \mc P$ which are $\mc P$--contracting with constant $C$ are $(r; n, \mc P)$--midthin.
\end{lemma}
\begin{proof}
    Let $\gamma$ be a special path which is $\mc P$--contracting with constant $C$ and let $m$ be the midpoint of $\gamma$. Let $\pi_\gamma : X\to \gamma$ be the associated projection. Observe that by Theorem~\ref{theorem:contracting-implies-morse}, $\gamma$ is $M$--Morse for some Morse gauge $M$ depending only on $C$. Let $\alpha = \alpha_1\ast\ldots \ast \alpha_n$ be an $(n, \mc P)$--polygonal line from $\gamma^+$ to $\gamma^-$. We have to show that $d(\alpha, m)\leq r$ for some $r$ depending only on $C$. 

    Temporarily, fix an index $i$ and denote $\pi_{\gamma}(\alpha_i^-)$ and $\pi_{\gamma}(\alpha_i^+)$ by $x$ and $y$, respectively. If $d(x, y)\geq C$, then since $\gamma$ is $\mc P$--contracting, there exist $x', y'\in \alpha_i$ such that $d(x',x)\leq C$ and $d(y', y)\leq C$. Thus, $\gamma$ being $M$--Morse,  yields an upper bound $r'$ depending only on $C$ of the Hausdorff distance between $\psub{\gamma}{x, y}$ and $\psub{\alpha}{x', y'}$ (see \cite[Lemma~2.1]{Cordes:Morse}). 

    Next define $r = r'+D_{\mc P}((n+2)C)$. If $d(\pi_\gamma(\alpha_i^-), \pi_\gamma(\alpha_i^+))\leq C$ for all $i$, then $d(\gamma^-, \gamma^+)\leq (n+2)C$. Hence $\diam(\gamma)\leq r$ yielding $d(m, \alpha)\leq d(m, \gamma^-)\leq r$. Assume now that $d(\pi_\gamma(\alpha_i^-), \pi_\gamma(\alpha_i^+))> C$ for at least one index $i$. If $d(m, \gamma^-\cup \gamma^+)\geq D_{\mc P}(C)$, then $m$ has to lie between $\pi_\gamma(\alpha_j^+)$ and $\pi_\gamma(\alpha_j^-)$ for at least one index $j$. Let $k$ be an index minimizing $\abs{j - k}$ for which $d(\pi_\gamma(\alpha_k^-), \pi_\gamma(\alpha_k^+))>C$. With this, $d(m, \alpha_k)\leq r'+\abs{j-k}C\leq r$, concluding the proof.
\end{proof}

\begin{proposition}\label{prop:contr-to-quasi-geo}
    Assume that $\mc K = (K, n, \mc P)$ is large enough and full. For every $C\geq 0$ there exists $Q\geq 1$ such that the following holds. If $\gamma\in \mc P$ is $\mc P$--contracting with constant $C$, then its image $\hat{\gamma}_{\mc K}$ is a parametrized $Q$--quasi-geodesic.
\end{proposition}

\begin{proof}
Let $\gamma$ be $\mc P$--contracting with constant $C$ and let $\pi_{\gamma} : X \to \gamma$ be the associated projection that makes $\gamma$ $\mc P$--contracting. Since $V(X) = V(\hat{X}_{\mc K})$, the map $\pi_\gamma$ induces a map $\tau : (V(\hat{X}_{\mc K}), \hat{d}_{\mc K})\to (\gamma, d)$. Note that the metrics on the two spaces differ. We show that $\tau$ is $\ka$--coarsely Lipschitz for some constant $\ka\geq C$ depending only on $C$ and the contraction triple $\mc K$, which would imply that $\hat{\gamma}_{\mc K}$ is a $\ka$--quasi-geodesic. 

\begin{figure}
    \centering
    \includegraphics[width=0.5\linewidth]{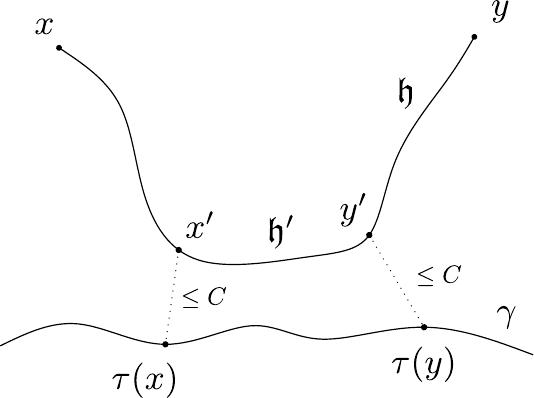}
    \caption{In the proof of Proposition~\ref{prop:contr-to-quasi-geo}, the path $\spath'\subset \spath$ has endpoints close to $\gamma$ and thus inherits $\mc P$--contraction from $\gamma$.}
    \label{fig:p-contracting-quasi-geodesic}
\end{figure}

Let $x, y \in V(\hat{X}_{\mc K})$ be vertices with $\hat{d}_{\mc K}(x, y)\leq 1$. In particular, there exists a special path $\spath\in \mc P$ from $x$ to $y$ which is $\mc K$--anti-contracting, see Lemma~\ref{lemma:small-implies-anti-contracting}. Now, it remains to show that $d(\tau(x), \tau(y))\leq \ka$. Assume $d(\tau(x), \tau(y)) > \ka$. In particular, $d(\tau(x), \tau(y)) > C$, which, since $\gamma$ is $\mc P$--contracting with constant $C$ implies that there exist points $x', y'$ on $\spath$ in the $C$-neighbourhood of $\tau(x)$ and $\tau(y)$, respectively. This is depicted in Figure~\ref{fig:p-contracting-quasi-geodesic}. Let $\spath'\subset\spath$ be a subpath whose endpoints are $x'$ and $y'$, respectively. By Lemma~\ref{lem:contracting_preserved_Hdistance}, $\spath'$ is $\mc P$--contracting with constant $C'$ depending only on $C$. Consequently, by Lemma~\ref{lem:p-contracting-implies-midthin}, $\spath'$ is $(r; n, \mc P)$--midthin for some $r\geq 0$ depending only on $C'$ and $n$. Now set $\ka = D_{\mc P}(K(r)+1)+2C$. With this, $\abs{\spath'} > K(r)$, in particular, $\spath'$ is $\mc K$--midthin. This contradicts $\spath$ being $\mc K$--anti-contracting. Hence $\tau$ is $\ka$--coarsely Lipschitz, concluding the proof.
\end{proof}

\begin{proposition}\label{prop:quasi-geodesic-implies-contracting}
     Assume that $\mc K$ is large enough. For every $Q\geq 0$ there exists $C\geq 0$ such that the following holds. If $\gamma\in \mc P$ is a special path whose image $\hat{\gamma}_{\mc K}$ is a parametrized $Q$--quasi-geodesic, then $\gamma$ is $\mc P$--contracting with constant $Q'$ depending only on $\mc K$ and $Q$.
\end{proposition}

\begin{proof}
    Let $\gamma \in \mc P$ be a special path such that $\hat{\gamma}_{\mc K}$ is a $Q$--quasi-geodesic. We show that $\gamma$ is $\mc P$--contracting with constant $Q' = 100D_{\mc P}(D_{Q, Q}(2))$. Let $D = D_{Q, Q}(2)$.

    Define the projection $\pi$ as in the proof of Proposition~\ref{prop:local_poly_P_contracting}, that is, for $x\in X$ define $\pi(x)$ as a $D$--lower projection point of some special path $\spath_x$ from $x$ to $\gamma$. Further, define $\tau(x)$ as the $D$--upper projection point of $\spath_x$ onto $\gamma$. 

    Assume that there exist points $x, y\in X$ with $d(\pi(x), \pi(y)) > Q'$ but such that there exists a special path $ \beta = [y, x]_{\mc P}$ with $d(\pi(x), \beta) > Q'$.

    Let $z_1$ be a vertex on $\psub{\gamma}{x, y}$ with $d(\pi(x), z_1) = Q'/4$ and let $z_2$ be a vertex on $\psub{\gamma}{z_1, y}$ with $\hat{d}_{\mc K}(z_1, z_2) = 2$. Such vertices exist since $\hat{\gamma}_{\mc K}$ is a $Q$--quasi-geodesic and $Q'\gg Q$. Further, there has to exist a $\mc K$--midthin subpath $\tilde{\gamma}$ of $\psub{\gamma}{z_1,z_2}$ with midpoint $m$ and neck-size $r$ for some $r$ with $K(r)\leq \abs{\tilde{\gamma}}\leq D_{Q, Q}(2)$. Also observe that $Q'/2\geq d(\pi(x), m) \geq Q'/4$ and $d(\pi(y), m)\geq Q'/4$ and for large enough $\mc K$, $r\leq Q'/8$ and $r\leq D$.

    \begin{figure}
        \centering
        \includegraphics[width=0.7\linewidth]{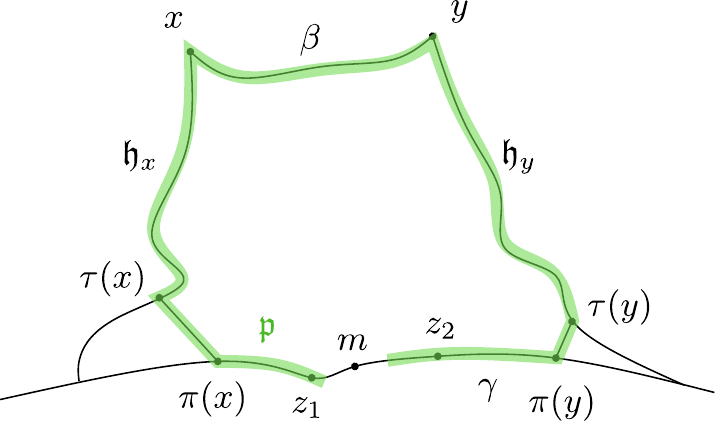}
        \caption{In the proof of Proposition~\ref{prop:quasi-geodesic-implies-contracting}, the path $\pgot$ has to intersect the $r$--neighbourhood of $m$.}
        \label{fig:quasi-implies-contracting-space}
    \end{figure}

    Next consider the $(7, \mc P)$--polygonal line 
    $$\pgot = \psub{\gamma}{\tilde{\gamma}^+, \pi(y)}\ast[\pi(y), \tau(y)]_{\mc P}\ast \pref{\spath_y}{\tau(y)}^{-1}\ast \beta\ast \pref{\spath_x}{\tau(x)}\ast[\tau(x), \pi(x)]_{\mc P} \ast \psub{\gamma}{\pi(x), \tilde{\gamma}^-},$$
    as depicted in Figure~\ref{fig:quasi-implies-contracting-space}.

    Since $\tilde{\gamma}$ is $\mc K$--midthin with neck-size $r$, we have $d(\pgot, m)\leq r$. We can thus obtain a contradiction by showing that $d(\qgot, m) > r$ for all of the $7$ special paths $\qgot$ composing $\pgot$. We have $d(m, [\tau(x), \pi(x)]_{\mc P}) \geq d(m,\pi(x)) - \diam([\tau(x), \pi(x)]_{\mc P})> r$. Likewise, $d(m, [\pi(y), \tau(y)]_{\mc P})\geq d(m, \pi(y)) - \diam([\pi(y), \tau(y)]_{\mc P}) >  r$. Also $d(\beta, m) > d(\beta, \pi(x)) - d(\pi(x), m) >r$ by the choice of $\beta, x$ and $y$. If $\qgot = \pref{\spath_y}{\tau(y)}^{-1}$ or $\qgot = \pref{\spath_x}{\tau(x)}$, then $d(m, \qgot)> D\geq r$ by the definition of a $D$--upper projection point. Finally, if $\qgot \subset \gamma$, then for large enough $K$, $d(\qgot, m) > r$ by Lemma~\ref{lemma:continued-special-path-distance}.
\end{proof}

\begin{theorem}\label{prop:morse-quasi-geo-equivalence}
    Assume that $\mc P$ is navigable and $\mc K$ is large enough and full. Let $\gamma\in \mc P$ be a special path. Its image $\hat{\gamma}_{\mc K}$ is a parametrized $C$-quasi-geodesic if and only if $\gamma$ is $M$--Morse. Moreover, this equivalence is quantitative in the sense that $M$ depends only on $C$ and vice versa. 
\end{theorem}

\begin{proof}
    This follows directly by combining Corollary~\ref{cor:morse-implies-mccontracting} with Propositiona~\ref{prop:contr-to-quasi-geo} and~\ref{prop:quasi-geodesic-implies-contracting}.
\end{proof}

\subsection{Group actions on the contraction space}\label{sec:acylindricity}

Let $\mc K$ be a contraction triple and let $G$ be a group acting on $X$ by graph isomorphisms. If the path system $\mc P$ is $G$--invariant, then the action of $G$ on $X$ induces an action by graph isomorphisms of $G$ on $\hat{X}_{\mc K}$, because in that case the construction of $\hat{X}_{\mc K}$ is $G$--invariant. Moreover, if $G$ acts \emph{$\rho$--coboundedly} on $X$, that is, if for some $x_0\in X$, the set $G\cdot x_0$ is $\rho$-dense in $X$, then it acts $(\rho+1)$--coboundedly on $\hat{X}_{\mc K}$. We next show that if $G$ acts properly on $X$, it acts non-uniformly acylindrically on $\hat{X}_{\mc K}$.

\begin{theorem}\label{prop:non-uniformly-acylindrical}
    Let $G$ be a group which acts properly on $X$ and such that $\mc P$ is $G$--invariant. If $\mc K$ is large enough, then the induced action of $G$ on $\hat{X}_{\mc K}$ is non-uniformly acylindrical.
\end{theorem}
\begin{proof}
    Since $V(X)$ is $1$-dense in $\hat{X}_{\mc K}$, it is enough to show that for all $\eps > 0$ there exists $R>0$ such that for every pair of vertices $x, y\in V(X)$ with $\hat{d}_{\mc K}(x, y)\geq R$ we have 
    \begin{align}\label{eq:non-acy-hyp}
        \abs{\{g\in G \mid \hat{d}_{\mc K}(x, gx) <\eps, \hat{d}_{\mc K}(y, gy)<\eps\}}<\infty.
    \end{align}
    Let $\eps > 0$ and define $R= 2\ceil{\eps+\delta_0} + 10$. Let $x, y\in X$ be vertices with $\hat{d}_{\mc K}(x, y)\geq R$ and let $g\in G$ be such that $\hat{d}_{\mc K}(x, gx)<\eps$ and $\hat{d}_{\mc K}(y, gy)<\eps$. Denote by $\spath, \spath_x$ and $\spath_y$ special paths from $x$ to $y$, from $gx$ to $x$, and from $y$ to $gy$, respectively. Let $z_x, z_y$ be the first respectively last vertex on $\spath$ at distance $\ceil{\eps+\delta_0}+4$ from $x$, respectively $y$. With this, $\hat{d}_{\mc K}(z_x, z_y) \geq 2$. Thus, there exists a $\mc K$--midthin subpath $\gamma$ of $\psub{\spath}{z_x, z_y}$ with neck-size $r$ for some $r\geq 0$.

    \begin{claim}\label{claim:distr}
       We have $d(m_\gamma, g\spath)\leq r$. 
    \end{claim}

    \begin{figure}
        \centering
        \includegraphics[width=0.9\linewidth]{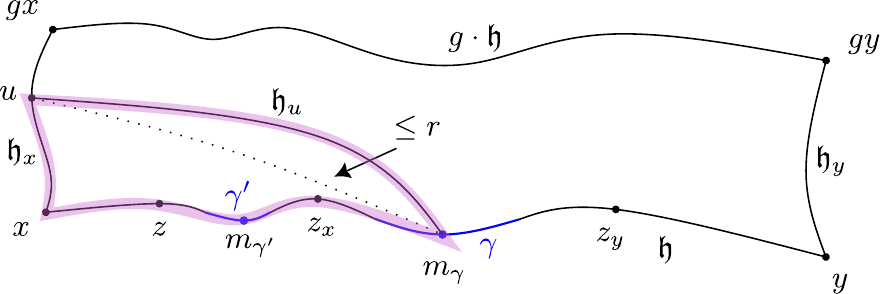}
        \caption{The blue paths $\gamma$ and $\gamma'$ are $\mc K$--midthin. We contradict this constellation by applying Proposition~\ref{lem:technical-properties-of-contraction-space} to the triangle shaded in pink.}
        \label{fig:non-uniform-acylindricity}
    \end{figure}
    Since the action of $G$ on $X$ is proper, and the length of $\spath$ is finite, there can only be finitely many elements $g\in G$ with $d(m_\gamma, g\spath)\leq r$. Hence the claim concludes the proof. 

    \smallskip
    
    \textit{Proof of Claim \ref{claim:distr}.}
        Since $\gamma$ is $\mc K$--midthin with neck-size $r$, $d(\pgot, m_\gamma)\leq r$ for the $(5, \mc P)$--polygonal line $\pgot = \psub{\spath}{\gamma^+, y}\ast \spath_y\ast (g\spath)^{-1}\ast \spath_x\ast \psub{\spath}{x, \gamma^-}$. By Lemma~\ref{lemma:continued-special-path-distance}, $d(\psub{\spath}{\gamma^+, y}, m_{\gamma}) > r$ and $d(\psub{\spath}{x, \gamma^-}) > r$. Thus, to conclude the proof, it suffices to show that $d(m_{\gamma}, \spath_x)>r$ and $d(m_{\gamma}, \spath_y)> r$. 

        We will show that $d(m_{\gamma}, \spath_x) > r$, the proof for $\spath_y$ is similar. Assume, by contradiction, that there exists $u\in \spath_x$ with $d(m_\gamma, u)\leq r$ and let $\spath_u\in \mc P$ be a special path from $m_{\gamma}$ to $u$. This is depicted in Figure~\ref{fig:non-uniform-acylindricity}. Denote $\psub{\spath}{x, m_{\gamma}}$ by $\spath'$ and let $z\in \spath'$ be the first vertex with $d(x, z) = \ceil{\eps+\delta_0}+2$. Note that $z_x$ lies on $\spath'$. By the triangle inequality, $\hat{d}_{\mc K}(z_x, z)\geq 2$, and hence there exists a special path $\gamma'\subset \psub{\spath'}{z, z_x}$ which is $\mc K$--midthin with neck-size $r'$.

        Now we apply Proposition~\ref{lem:technical-properties-of-contraction-space} to $\pgot = \spath_u\ast \psub{\spath_x}{u, x}$, $\gamma = \spath'$ and $\tilde{\gamma} = \gamma'$ to get a vertex $w$ on $\spath_u\ast\psub{\spath_x}{u, x}$ with $\hat{d}_{\mc K}(w, \gamma')\leq 1$ and $d(w, m_{\gamma'})\leq r'$. Using Remark~\ref{rem:on_geodesic}, we obtain that any point $w'\in \spath_x$ satisfies $\hat{d}_{\mc K}(w', x)\leq \eps +\delta_0$, while $\hat{d}_{\mc K}(\gamma', x)\geq \ceil{\eps+2\delta_0}+2 - \delta_0$. Hence, by the triangle inequality, $\hat{d}_{\mc K}(w, \gamma') \geq 2$, implying that $w$ has to lie on $\spath_u$. 

        Now, if $w$ lies on $\spath_u$, we have $d(m_{\gamma'}, m_{\gamma})\leq d(m_{\gamma'}, w)+d(w, m_{\gamma})\leq r' + D_{\mc P}(r)$. Therefore $\abs{\psub{\spath'}{m_{\gamma'}, m_{\gamma}}}\leq D_{\mc P}(r'+D_{\mc P}(r))$. However, $\abs{\psub{\spath'}{m_{\gamma'}, m_{\gamma}}}\geq K(r)/2 + K(r')/2$. For large enough $K$, this yields the desired contradiction.
    \hfill$\blacksquare$

\end{proof}

The following corollary follows immediately from Theorem~\ref{prop:non-uniformly-acylindrical}.

\begin{corollary}\label{cor:wpd-elements}
    Let $\mc K$ be a large enough contraction triple. Let $G$ be a group which acts properly on $X$ and such that $\mc P$ is $G$--invariant. If $g\in G$ is an element such that $g\acts \hat{X}_{\mc K}$ is loxodromic then $g\acts \hat{X}_{\mc K}$ is WPD.
\end{corollary}

We also obtain that, in the presence of a nice enough group action, having an unbounded contraction space implies having a $\mc P$--contracting ray.

\begin{corollary}\label{cor:contracting-element}
    Let $G$ be a group acting properly and coboundedly on $X$. Assume that $\mc P$ is $G$--invariant, $\mc K$ is large enough and $\widehat{\diam}_{\mc K}(\hat{X}_\mc{K}) = \infty$. Then $G$ has a $\mc P$--contracting element. In particular, $X$ has a $\mc P$--contracting quasi-geodesic ray. Moreover, if $G$ is not virtually cyclic, then $G$ is acylindrically hyperbolic.
\end{corollary}
\begin{proof}
    Since $G$ acts coboundedly on $X$ and $\mc P$ is $G$--invariant, $G$ acts coboundedly on $\hat{X}_{\mc K}$, which is $\delta_0$-hyperbolic for large enough $\mc K$. By \cite{Gromov:hyperbolic}, there exists an element $g\in G$ which is loxodromic. Since $G$ acts properly on $X$, $g$ is a WPD by Corollary~\ref{cor:wpd-elements}. Hence, if $G$ is not virtually cyclic, it is acylindrically hyperbolic by \cite[Theorem~1.2]{Osin:acylindrically}. By \cite[Theorem~1]{Sisto:quasi-convexity}, the axis $\gamma$ of $g$ in $X$ is $M$--Morse for some Morse gauge $M$. Let $\spath$ be any special path connecting two points $x, y$ on $\gamma$. Since $\gamma$ is $M$--Morse, $\spath$ is contained in a uniform neighbourhood of $\gamma$. In particular, $\hat{\gamma}_{\mc K}$ is a $Q$--quasi-geodesic for some uniform $Q$. Proposition~\ref{prop:quasi-geodesic-implies-contracting} shows that $\spath$ is $\mc P$--contracting with some uniform constant $Q'$. Lemma~\ref{lem:contracting_preserved_Hdistance} shows that $\psub{\gamma}{x, y}$ is $\mc P$--contracting, for some uniform constant $C$. Since this holds for any choice of $x$ and $y$, it follows that $\gamma$ is $\mc P$--contracting, concluding the proof.
\end{proof}

Next we prove an analogue of \cite[Proposition~4.2]{Zbinden:hyperbolic}, which shows that if the action of $G$ on $X$ is cobounded, then the $\mc K$--contraction space satisfies a weak diameter dichotomy. We do this by adapting the proof of \cite[Proposition~4.2]{Zbinden:hyperbolic} to the more general setting of a group acting coboundedly on a hyperbolic space.

\begin{lemma}[Weak diameter dichotomy]\label{lemma:actions-on-finite-hyperbolic-spaces}
    Let $G$ be a group acting (by isometries) on a $\delta$-hyperbolic geodesic metric space $Y$. For every $\rho > 0$ there exists a constant $\Delta = \Delta(\delta, \rho) = 4\delta + 4\rho +3$ such that the following holds. If any orbit of $G$ is $\rho$--dense in $Y$, then one of the following holds
    \begin{itemize}
        \item $\diam(Y) \leq \Delta(\rho, \delta)$,
        \item $\diam(Y) = \infty$.
    \end{itemize}
\end{lemma}

\begin{figure}
    \centering
    \includegraphics{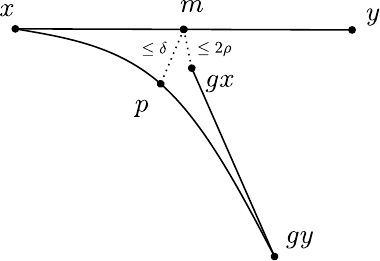}
    \caption{Proof of the weak diameter dichotomy.}
    \label{fig:diam}
\end{figure}

\begin{proof}
 Assume by contradiction that $\infty > \diam(Y) = T > \Delta = 4\delta+2\rho +3$ and let $x, y\in Y$ be points with $d(x, y) = T-1$. Let $m$ be the midpoint of a geodesic $[x, y]$. Since there exists a $\rho$-dense orbit of $G$, every orbit of $G$ is $2\rho$-dense. As depicted in Figure \ref{fig:diam}, there exists therefore $g\in G$ such that $d(gx, m)\leq 2\rho$. Consider the geodesic triangle with vertices $x, y$ and $gy$. Since $Y$ is $\delta$-hyperbolic, there exists $p\in [x, gy]\cup[gy, y]$ with $d(m, p)\leq \delta$. If $p\in [x, gy]$, we have that 
\begin{align*}
    d(x, gy) &= d(x, p) + d(p, gy) \geq d(x, m) - \delta + d(gy, gx) - \delta - 2\rho \geq \\ 
    &\geq \frac{3}{2}(T-1) - 2\delta - 2\rho > T, 
\end{align*}
a contradiction. If $p\in[gy, y]$, we get a contradiction similarly, which concludes the proof.
\end{proof}

\begin{corollary}[Weak diameter dichotomy]\label{cor:weak-diameter}
     Let $X$ and $\mc P$ be as in Theorem \ref{prop:non-uniformly-acylindrical}, and let $\mc K$ be large enough. If there exists a group $G$ acting $\rho$--coboundedly on $X$ such that $\mc P$ is $G$--invariant, then one of the following holds
     \begin{itemize}
         \item $\widehat{\diam}(\hat{X}_{\mc K})\leq \Delta(\rho+1, \delta_0) = 4\delta_0+4\rho+5$,
         \item $\widehat{\diam}(\hat{X}_{\mc K}) = \infty$.
     \end{itemize}
\end{corollary}
\begin{proof}
    If $G$ acts $\rho$--coboundedly on $X$ then it acts $(\rho+1)$--coboundedly on $\hat{X}_{\mc K}$, which is $\delta_0$-hyperbolic for large enough $\mc K$. Lemma~\ref{lemma:actions-on-finite-hyperbolic-spaces} concludes the proof.
\end{proof}

\begin{definition}\label{def:allowed-contraction-triple}
    We say that a contraction triple $\mc K = (K, n, \mc P)$ is \emph{allowed} if $K$ is full, $n\geq 7$ and $\mc K$ is large enough to satisfy all the results in this Section.
\end{definition}

\section{Comparing contraction spaces}\label{sect:comparing_contraction_spaces}

\textbf{Standing assumptions:} Let $X = (V, E)$ be a graph with induced path metric $d$ and let $G$ be a group acting on $X$ by graph isomorphisms. Moreover, let $\mc P$ be an undirected $G$--invariant path system of $X$.

We will show that the contraction space satisfies a strong diameter dichotomy. 

\begin{restatable}[Strong Diameter dichotomy]{theorem}{strongdiameter}\label{thm:strict-diameter-dichotomy}
   Let $G$ be a non-virtually cyclic group which acts properly, coboundedly and by graph isomorphisms on a graph $X$ equipped with a $G$--invariant undirected quasi-geodesic path system $\mc P$. There exists an allowed contraction triple $\mc K = (n, K, \mc P)$ such that exactly one of the following holds:
   \begin{enumerate}
       \item\label{item:diam-one} $\widehat{\diam}_{\mc K}(V(X)) = 1$, that is, every pair of vertices $(x, y)\in V(X)\times V(X) $ is $\mc K$--anti-contracting;
       \item\label{item:diam-inf} $\widehat{\diam}_{\mc K}(V(X)) = \infty$.
   \end{enumerate}
\end{restatable}

\subsection*{Proof Strategy.} We fix two allowed contraction triples $\mc K = (K, n, \mc P)$ and $\mc K' = (K', n', \mc P)$ such that $K'\gg K$ and $n' \gg n$. If $\hat{X}_{\mc K}$ satisfies (\ref{item:diam-inf}) or (\ref{item:diam-one}), we are done, so we can assume that $2 \leq  \widehat{\diam}_{\mc K}(V(X))$ and $\widehat{\diam}_{\mc K'}(V(X)) < \infty$. In particular, $\widehat{\diam}_{\mc K'}(V(X))< \Delta = \Delta(\rho+1, \delta_0)$ by the weak diameter dichotomy. Further, we can assume that there exists a $\mc K$--midthin special path $\gamma$. We want to use $\gamma$ and translates thereof to construct a polygonal line $\Lambda$ whose diameter with respect to $\hat{d}_{\mc K'}$ is larger than $\Delta$. The weak diameter dichotomy then immediately yields $\widehat{\diam}_{\mc K'}(V(x)) = \infty$. 

\subsubsection*{Construction of $\Lambda$.} This is done in Section~\ref{sec:construction-of-lambda}. The hardest part in the construction of $\Lambda$ is to find a translate $g\gamma$ of $\gamma$ such that $\gamma^+$ and $g\gamma^-$ can be joined by a $(5, \mc P)$--polygonal line $\eta$ staying away from both $m_\gamma$ and $gm_\gamma$. Then we can define $\Lambda$ as the periodic polygonal line that repeats $\gamma\ast \eta$ a fixed, but large ($\approx 3 \Delta$), number of times, and one can show (see Proposition~\ref{prop:no-skips}) that $\Lambda - g^k\gamma$ stays away from $g^km_\gamma$ for all $k$.

The construction of the $(5, \mc P)$-polygonal line $\eta$ goes as follows: we first show that there are arbitrarily long special paths $\spath, \spath'$ such that $\spath^- \approx\gamma^+$ and $\spath'^+ \approx\gamma^-$ and such that both $\spath$ and $\spath'$ stay away from $m_\gamma$. This is done in Lemma~\ref{lemma:attached-quasi-geo}. Then we use that $G$ has sublinear growth function to show in a combinatorial argument that there exists an element $g\in G$ such that $g\cdot(\spath'\ast \gamma)$ stays away from $m_\gamma$, $\gamma \ast \spath$ stays away from $gm_\gamma$ and $d(\spath^+, g\spath'^-)\ll\min(\abs{\spath}, \abs{\spath'})$. This is done in Lemma~\ref{lemma:good-connector}. Defining $\eta \approx \spath\ast [\spath^+, g\spath'^-]_{\mc P} \ast g\spath'$ yields the desired properties.

\subsubsection*{Lower bound of the diameter of $\Lambda$.} Since $\widehat{\diam}_{\mc K'}(V(X)) < \Delta$, there exists a $(\Delta+2, \mc P)$--polygonal line $\beta = \beta_1 \ast\ldots \ast \beta_{\Delta+2}$ from $\Lambda^+$ to $\Lambda^-$ such that all $\beta_i$ are $\mc K'$--anti-contracting. For $n\gg \Delta$ large enough, $\Lambda \ast \beta$ is a $(n, \mc P)$--polygonal line. In particular, we can use $\mc K$--midthinness of $\gamma$ and its translates to obtain that $g^jm_\gamma$ has to come close to $\pgot\ast \beta - g^j\gamma$ for all $j$. Recall that by construction $\pgot - g^j\gamma$ is far from $g^j\gamma$. Next, we use the pigeonhole principle to obtain a single $\beta_i$ that comes close to $g^jm_\gamma$ for at least 3 different indices $j$. Finally, we argue that this contradicts $\beta_i$ being $\mc K'$--anti-contracting, as it `passes too close to a $\mc K$--midthin' path. This is made precise in Lemma~\ref{lemma:close-to-thin-implies-thin-contraction-gauge-version}.

\subsection{Paths close to midthin paths}

In this section, we prove Lemma~\ref{lemma:close-to-thin-implies-thin-contraction-gauge-version}, which states that a special path whose `endpoints are close to a polygonal line suitably containing a midthin path' is midthin. This should be thought of as an analogue of the fact that quasi-geodesics are Morse ($\mc P$--contracting) if their endpoints are close to a Morse ($\mc P$--contracting) quasi-geodesic. In the midthin setting, the statement becomes more technical, as we, amongst other things, need to line up the midpoints.

\begin{lemma}\label{lemma:close-to-thin-implies-thin}
    Let $\mc P$ be an undirected path system. Let $n', k$ be positive integers and let $r, C, C', \ell\geq 0$ be constants. Let $\gamma\in \mc P$ be a special path and let $\pgot = \pgot_1\ast \spath \ast \pgot_2$ be a $(2k+1, \mc P)$--polygonal line satisfying the following. 
    \begin{enumerate}[label =\roman*)]
        \item $\spath\in \mc P$ is a special path that is $(r; 2k+4+n', \mc P)$--midthin. We denote its midpoint by $m$.\label{prop:pk-is-thin}
        \item \label{prop:pi-far-from-pk} For $i\in \{ 1, 2\}$, the path $\pgot_i$ is a $(k, \mc P)$--polygonal line with $d(\gamma, \pgot_i - \mc N_{C}(m))\leq \ell$ and $d(\pgot_i, m) > r$.
    \end{enumerate}
    If $C' > 2D_{\mc P}(2r)$ and $C\geq D_{\mc P}(C')+D_{\mc P}(\ell)+r+1$, then $\gamma$ contains a $(2r; n', \mc P)$--midthin subpath $\gamma'$ with $\abs{\gamma'}\geq C'$. 
\end{lemma}

The idea of the proof is as follows. Find a subpath $\gamma'$ of $\gamma$ whose midpoint is $r$--close to $m$. Then, for any $(n', \mc P)$--polygonal line $\qgot'$ between the endpoints of $\gamma'$, build an $(n, \mc P)$--polygonal line $\qgot$ between the endpoints of $\spath$ that consists of special paths of $\qgot'$ plus some auxiliary paths. Then, use the setup to show that the auxiliary paths are far away from $m$, implying that a point on $\qgot'$ is in the $r$--neighbourhood of $m$ and hence in the $2r$--neighbourhood of the midpoint of $\gamma'$. 

\begin{proof}
    \begin{figure}
        \centering
        \includegraphics[width=0.8\linewidth]{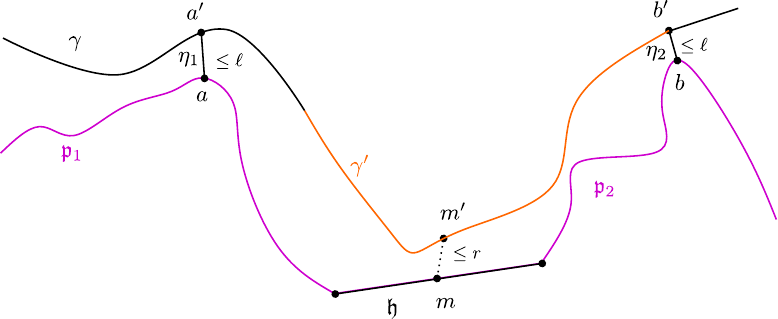}
        \caption{Showing that $\gamma$ has a subpath $\gamma'$ whose midpoint is close to $m$.}
        \label{fig:close-to-midthin1}
    \end{figure}
    \begin{figure}
        \centering
        \includegraphics[width=0.8\linewidth]{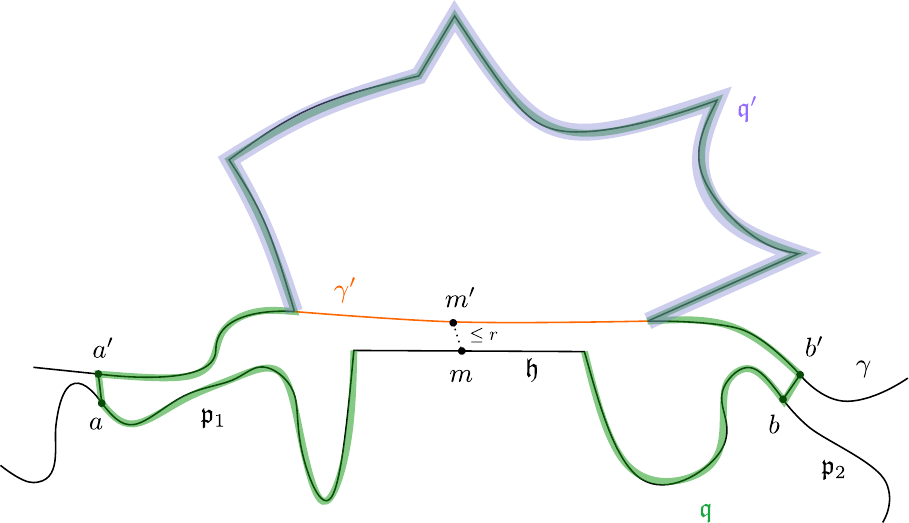}
        \caption{The path $\gamma'$ is midthin, as the polygonal path $\qgot'$ can be extended, which allows the use of $\spath$ being midthin.}
        \label{fig:close-to-midthin2}
    \end{figure}
    We first show that there is a subpath $\gamma'$ of $\gamma$ whose midpoint $m' = m_{\gamma'}$ is in the $r$--neighbourhood of $m$. 
    
    Let $a', b'$ be points on $\gamma$ and let $a, b$ be points on $\pgot_1 - \mc N_{C}(m)$, respectively $\pgot_2 - \mc N_C(m)$, with $d(a, a')\leq \ell$ and $d(b, b')\leq \ell$. Let $\eta_1$ and $\eta_2$ be special paths from $a'$ to $a$ and from $b$ to $b'$ respectively. This is depicted in Figure~\ref{fig:close-to-midthin1}. The path
    \begin{align*}
        \qgot = \psub{\pgot_2}{\pgot_2^-,b}\ast \eta_2 \ast \psub{\gamma}{b', a'}\ast \eta_1 \ast \psub{\pgot_1}{a, \pgot_1^+}
    \end{align*}
    is a $(2k+3, \mc P)$--polygonal line from $\spath^+$ to $\spath^-$. Since $\spath$ is $(r; 2k+3, \mc P)$--midthin, there exists a point $m'\in \qgot$ with $d(m',m)\leq r$. By \ref{prop:pi-far-from-pk}, none of the paths $\pgot_i$ intersect the $r$--neighbourhood of $m$. Since $d(a, m)\geq C$, but $d(a, a')\leq \ell$, the conditions on $C$ imply that $d(m, \eta_i) > r$ for $i = 1, 2$. Thus, there has to be a point $m'$ on $[a', b']_{\gamma}$ which is in the $r$--neighbourhood of $m$. Define $\gamma'$ as a subsegment of $[a', b']_{\gamma}$ whose midpoint is $m'$ and which has maximal domain length amongst all such subsegments. In particular, at least one of $\psub{\gamma}{a', m'}$ and $\psub{\gamma}{m', b'}$ is a subsegment of $\gamma'$, yielding $\abs{\gamma'}\geq C'$. Indeed, $d(a, m)\geq C$ and the conditions on $C$ and $C'$ yield $d(a', m')\geq D_{\mc P}(C')$ and hence $\abs{\psub{\gamma}{a', m'}}\geq C'$. Similarly, $\abs{\psub{\gamma}{m', b'}}\geq C'$.

    \smallskip
    
    It remains to show that $\gamma'$ is $(2r; n', \mc P)$--midthin. Let $\qgot'$ be an $(n', \mc P)$--polygonal line from $\gamma'^+$ to $\gamma'^-$. Denote by $\qgot $ the $(n'+2k+4, \mc P)$--polygonal line
    \begin{align*}
        \qgot  = \psub{\pgot_2}{\spath^+, b} \ast \eta_2\ast \psub{\gamma}{b', \gamma'^+} \ast \qgot' \ast \psub{\gamma}{\gamma'^-, a'}\ast \eta_1\ast \psub{\pgot_1}{a, \gamma^-}.
    \end{align*}
    This is depicted in Figure~\ref{fig:close-to-midthin2}. Since $\spath$ is $(r; n'+2k+4, \mc P)$--midthin, $\qgot $ has to intersect the $r$--neighbourhood of $m$. As argued above, $\eta_i$ and $\pgot_i$ do not intersect the $r$--neighbourhood of $m$. Further, $d(m',  \psub{\gamma}{b', \gamma'^+})\geq D_{\mc P}^{-1}(d(m', b')) \geq C'\geq 2r$, implying $d(\psub{\gamma}{b', \gamma'^+}, m) > r$. Similarly, $d( \psub{\gamma}{\gamma'^-, a'}, m) > r$.
 
    Therefore $d(\qgot', m)\leq r$, implying $d(\qgot', m')\leq 2r$ as desired. 
\end{proof}

\begin{lemma}\label{lemma:close-to-thin-implies-thin-contraction-gauge-version}
    Let $n', k$ be positive integers, let $n = n'+2k+4$ and let $\mc K' = (K', n, \mc P)$ be a contraction triple. There exists a non-decreasing function $C:\R_{\geq 0}\to \R_{\geq 0}$ such that the following holds. For all $r\geq 0$, given $\gamma\in \mc P$ a special path and $\pgot = \pgot_1\ast \spath \ast \pgot_2$ a $(2k+1, \mc P)$--polygonal line satisfying 
    \begin{enumerate}[label =\roman*)]
        \item \label{prop2:pk-is-thin} $\spath\in \mc P$ is a special path that is $(r; n, \mc P)$--midthin, and $m$ is its midpoint;
        \item \label{prop2:pi-far-from-pk} for $i\in \{ 1, 2\}$, the path $\pgot_i$ is a $(k, \mc P)$--polygonal line with $d(\gamma, \pgot_i - \mc N_{C(r)}(m))\leq r$ and $d(\pgot_i, m) > r$;
    \end{enumerate}
   there exists a subpath $\gamma'$ of $\gamma$ that is $\mc K'$--midthin with neck-size $2r$.
\end{lemma}

\begin{proof}
    By potentially increasing $K'$, we may assume that $K'(r)> r$ for all $r\geq 0$. Define $C(r) =D_{\mc P}(K'(2r)) + D_{\mc P}(r) + r+1$. The statement follows from Lemma~\ref{lemma:close-to-thin-implies-thin} applied to $\ell = r,C' = K'(2r)$ and $C = C(r)$.
\end{proof}

\subsection{Constructing $\Lambda$}\label{sec:construction-of-lambda}

In this section, we construct the path $\Lambda = \Lambda (\gamma, k)$ as described in the beginning of Section~\ref{sect:comparing_contraction_spaces}. 

\begin{lemma}\label{lemma:attached-quasi-geo}
    Assume that $\diam(X)$ is infinite and that $G$ acts $\rho$--coboundedly on $X$. Let $Q'$ be as in Lemma~\ref{lemma:quasi-geodesic-concatenation} applied to the quasi-geodesic constants $\lao, \kao$ of $\mc P$ and $2\rho$. Further, let $Q'' = 2Q'$, let $\gamma\in \mc P$ be a special path and let $M\geq 0$ be a constant. If $\abs{\gamma}\geq (Q'')^2D_{\mc P}(2\rho)$, then there exists a special path $\spath\in \mc P$ with $\abs{\spath} = M$ and such that
    \begin{align}
        d(\spath^-, \gamma^+)&\leq 2\rho,\label{eq:close-to-x}\quad \text{and}\\
        d(\spath\cup [\spath^-, \gamma^+]_{\mc P}, m_{\gamma}) &\geq \abs{\gamma}/Q''.\label{eq:far-from-mid}
    \end{align}
\end{lemma}

\begin{proof}
     Let $M\geq 0$. Let $\spath' :  [a, b] \to X$ be a special path with $\abs{\spath'} >  2M$ (this exists, since $\diam(X)$ is infinite), and let $t = a+M$. Let $h\in G$ be an element such that $d(hx, \spath'(t))\leq 2\rho$. Applying Lemma~\ref{lemma:quasi-geodesic-concatenation}, yields
    \begin{align}
         d(\isub{\spath'}{t, b}, hm_\gamma)\geq \abs{\gamma}/Q' \qquad \text{or} \qquad d(\isub{\spath'}{a, t}, hm_\gamma)\geq \abs{\gamma}/Q'.\label{eq:one-stays-close}
    \end{align}
    If the former holds, define $\spath = h^{-1}\cdot\isub{\spath'}{t, t+M}$, if the latter holds, define $\spath = h^{-1}\cdot\left(\isub{\spath'}{t-M, t}\right)^{-1}$. Since $\mc P$ is $G$--invariant and undirected, we have $\spath\in \mc P$. With this, \eqref{eq:close-to-x} holds. Observe that $\diam([\spath^-, x]_{\mc P})\leq D_{\mc P}(2\rho) \leq \abs{\gamma}/Q''$. Hence \eqref{eq:one-stays-close} combined with the triangle inequality yields \eqref{eq:far-from-mid} and concludes the proof.
\end{proof}

\begin{definition}
    Let $x_0\in X$ be a point. We denote by $S(k, x_0)$ the number of elements $g\in G$ such that $d(x_0, gx_0)\leq k$. In addition, we define
    \[
    S(k)  = \min\{S(k, x) \mid x\in X\},
    \]
    and call $S$ the {\emph{growth function of}} $G$. 
\end{definition}

If the action of $G$ is proper, then the above function only takes finite values. If the action is $\rho$--cobounded, then $S(k, x) \leq S(k + 2\rho, y)$ and hence $S(k,x)$ is linear, for some $x$, if and only if $S(k)$ is.
We recall the following fact due to Justin (see \cite[Corollary~7.27]{CheccheriniDadderio:Topics} for a modern proof).

\begin{lemma}[{\cite{Justin:Groupes}}]\label{lemma:linear-growth-implies-virtually-cyclic}
    If $G$ acts properly and coboundedly on $X$ and is not virtually cyclic, then the growth function $S$ of $G$ is superlinear. 
\end{lemma}

In fact, if $G$ is not virtually cyclic, then $S$ is at least quadratic, but we only need the weaker result stated above. 

\begin{lemma}\label{lemma:good-connector}
 Assume that $G$ is not virtually cyclic and acts properly and $\rho$--coboundedly on $X$. Let $Q''$ be as in Lemma~\ref{lemma:attached-quasi-geo}. Let $\gamma\in \mc P$ be a special path with $\abs{\gamma}\geq (Q'')^2D_{\mc P}(2\rho)$ and midpoint $m=m_\gamma$. There exists an element $g\in G$ and a $(5, \mc P)$--polygonal line $\eta$ from $\gamma^+$ to $g\gamma^-$ satisfying 
    \begin{align}
        d(m, g\gamma \cup \eta)&\geq \frac{\abs{\gamma}}{Q''}, \label{eq:eta-far1}\quad\text{and}\\
        d(gm, \gamma \cup \eta) &\geq \frac{\abs{\gamma}}{Q''}.\label{eq:eta-far2}
    \end{align}
\end{lemma}

\begin{proof}
    Let $Q$ be the maximum of the quasi-geodesic constants of $\mc P$. Let $M\geq 1$ be an integer which we determine later and which will depend on $\abs{\gamma}, Q$ and $\rho$. Let $\spath \in \mc P$, $\spath : [t, t +M]\to X$, be the special path we get if we apply Lemma~\ref{lemma:attached-quasi-geo} to $\gamma$ and $M$. Likewise, let $\spath' \in \mc P$,  $\spath' : [t', t' +M]\to X$ be the special path we get if we apply Lemma~\ref{lemma:attached-quasi-geo} to $\gamma^{-1}$. Note that, by potentially reparametrizing $\spath$ and $\spath'$, we can assume that $t$ and $t'$ are integers.
    
    The idea of the proof is the following: if we can find an element $g\in G$ such that $d(g\spath, m)$ and $d(gm, \spath')$ are both sufficiently large while $d(g\spath^+, \spath'^+)$ is small compared to $M$, then
    \[
    [\gamma^+, \spath'^-]_{\mc P}\ast \spath' \ast[\spath'^+, g\cdot\spath^+]_{\mc P} \ast g\cdot \spath^{-1}\ast [g\cdot\spath^-, g\cdot \gamma^-]_{\mc P}
    \]
    satisfies all the requirements of the statement. Hence, we now aim to find such a $g\in G$ using that the growth function is superlinear. 

    Let $\eps = \abs{\gamma}/Q'' + D_{\mc P}(2\rho) + D_{\mc P}(\abs{\gamma})$. There exist $K_1\geq 1$ depending only on $\eps$ and $Q$ such that 
    \begin{align}
        \diam([\spath'^+, g\cdot\spath^+]_{\mc P})+\eps < \min\{d(\spath^-, \spath^+), d(\spath'^-, \spath'^+)\},\label{eq:upper-path-bounded} 
    \end{align}
    whenever $M\geq K_1$ and $d(\spath'^+, g\cdot\spath^+)\leq M/K_1$.

    By the definition of $S(\cdot, \cdot)$, for any particular $s'\in [t', t'+M]$, there exists at most $S(\eps, \gamma^-)$ many elements $g\in G$ such that $d(\spath'(s'), g\gamma^-) \leq \eps$. In particular, there exist at most $(M+1)S(\eps, \gamma^-)$ elements $g\in G$ such that $d(\spath'(s'), g\gamma^-)\leq \eps$ for some $s'\in [t', t'+M]\cap \Z$. Consequently, there exist at most $\left ((M+1)S(\eps, \gamma^-)+(M+1)S(\eps, \gamma^+)\right )$ many elements $g\in G$ such that $d(\spath'(s'), g\gamma^-)\leq \eps$ for some $s'\in [t', t'+M]\cap \Z$ or $d(g\spath(s), \gamma^+)\leq \eps$ for some $s\in [t, t+M]\cap \Z$.

    Since $G$ is not virtually cyclic and hence its growth function $S$ is superlinear there exists $K_2\geq K_1$ such that for all $k\geq K_2$,
    \begin{align}
        S(k/K_1) &> (k+1)S(\eps, \gamma^-)+ (k+1)S(\eps, \gamma^+).\label{eq2}
    \end{align}

    Choose $M = K_2$. The inequalities above show that there exists an element $g\in G$ such that 
    \begin{align}
        d(\spath'^+, g\spath^+)&< M/K_1,\label{eq:upper-distance-bounded}\\
        d(\spath'(s'), gm) &> \eps,\quad &\text{for all $s'\in [t', t'+M]\cap \Z$.}\label{eq:h-bound}\\ 
        d(g\spath(s), m) &> \eps,\quad &\text{for all $s\in [t, t+M]\cap \Z$.}\label{eq:h':bound}
    \end{align}
    
    Define
    \begin{align*}
        \eta = [\gamma^+, \spath'^-]_{\mc P}\ast \spath' \ast[\spath'^+, g\cdot\spath^+]_{\mc P} \ast g\cdot \spath^{-1}\ast [g\cdot\spath^-, g\cdot \gamma^-]_{\mc P}
    \end{align*}

    We check that $g$ and $\eta$ satisfy all the conditions in the statement. We start by proving \eqref{eq:eta-far1}. Inequality \eqref{eq:upper-distance-bounded} implies that \eqref{eq:upper-path-bounded} holds. In particular, by the triangle inequality, $d([\spath'^+, g\spath^+]_{\mc P}, \spath^-) > \eps$, yielding  
    \[
    d([\spath'^+, g\spath^+]_{\mc P}, m)\geq  d([\spath'^+, g\spath^+]_{\mc P}, \spath^-) - d(\spath^-, m)\geq \eps - 2\rho - \diam(\gamma) \geq \abs{\gamma}/Q''.
    \]
    Further, by Lemma~\ref{lemma:attached-quasi-geo}, we have
    \begin{align*}
        d(m, [\gamma^+, \spath'^-]_{\mc P}\ast \spath')\geq \abs{\gamma}/Q''.
    \end{align*}

    Note that by \eqref{eq:h':bound} we have 
    \begin{align*}
        d(m, g\spath) > \eps - D_{\mc P}(1)\geq \abs{\gamma}/Q''.
    \end{align*}
    Lastly, by the triangle inequality
    \begin{align*}
    d(m, [g\spath^-, g\gamma^+]_{\mc P}\cup g\gamma) &\geq d(m, g\spath^-) - \diam([g\spath^-, g\gamma^+]_{\mc P}) - \diam(\gamma)\\
    &\geq  \eps - D_{\mc P}(2\rho) - D_{\mc P}(\abs{\gamma})\geq \abs{\gamma}/Q''.
    \end{align*}
    Hence, we have that $d(m, \eta \cup g\gamma)\geq \abs{\gamma}/Q''$. The proof that $d(gm, \gamma \cup \eta)\geq \abs{\gamma}/Q''$ works analogously.
\end{proof}

Let $\gamma$ and $g$ be as in Lemma~\ref{lemma:good-connector}. For $k\geq 1$ we define the $(6k, \mc P)$--polygonal line $\Lambda(\gamma, k)$ as 
\begin{align*}
    \Lambda(\gamma, k) = \prod_{i=0}^{k-1} g^{i}\gamma\ast g^i\eta.
\end{align*}
Further define $\Lambda$ as $\Lambda(\gamma, \infty)$.

\begin{proposition}\label{prop:no-skips}
    Let $G,X, \gamma$ and $m$ be as in Lemma~\ref{lemma:good-connector}. Moreover, assume that $\abs{\gamma}\geq Q''(r+D_{\mc P}(r))+1$ and that $\gamma$ is $(r; n, \mc P)$--midthin. Let $\pgot$ be a $(c, \mc P)$--polygonal line. If $d(\pgot, g^im)\leq r$ and $d(\pgot, g^jm)\leq r$ for some $i\leq j$ with $6(j-i)+3+\max(c, 6)\leq n$, then $d(\pgot, g^lm)\leq r$ for all $i <l <j$. In particular,
    \begin{align}
        d(g^j\gamma, g^im) > r\quad &\text{for all $1\leq \abs{i - j} \leq (n-10)/6$,}\\
        d(g^j \eta, g^im) > r \quad &\text{for all $0\leq \abs{i-j} \leq (n - 10)/6$.}
    \end{align}  
\end{proposition}

\begin{proof}
    Fix $c\geq 6$ and assume that $\pgot$ is a $(c, \mc P)$--polygonal line with $d(\pgot, g^im)\leq r$ and $d(\pgot, g^jm)\leq r$ for some $i\leq j$ with $6(j-i)+4+c\leq n$ but $d(\pgot, g^lm)> r$ for some $i< l < j$. Finally, assume that amongst all such paths, $\pgot$ minimizes $j - i$. By potentially passing to a subsegment of $\pgot$ or its inverse, we can assume that $d(\pgot^-, g^jm)\leq r$ and $d(\pgot^+, g^im)\leq r$.

    Let $ \qgot$ be the path from $g^{i+1}\gamma^+$ to $g^{i+1}\gamma^-$ which `follows along $\Lambda(\gamma, k)$ and then uses $\pgot$ to get back to $g^{i+1}\gamma^-$.  That is, 
    \begin{align}
         \qgot =  \psub{\Lambda}{g^{i+1}\gamma^+, g^jm} \ast [g^jm, \pgot^-]_{\mc P}\ast \pgot\ast [\pgot^+, g^im]_{\mc P} \ast \psub{\Lambda}{g^im, g^{i+1}\gamma^-}.
    \end{align}

         The path $ \qgot$ is a $(n, \mc P)$--polygonal line. Hence $d( \qgot, g^{i+1}m)\leq r$. We show that all defining subpaths of $ \qgot$ have distance more than $r$ from $g^{i+1}m$, a contradiction.
    \begin{itemize}
        \item \textbf{the path $\pgot$.} Having $d(\pgot, g^{i+1}m)\leq r$, would contradict the minimality of $j-i$. 
        \item \textbf{the connector $\xi_1 = [g^jm, \pgot^-]_{\mc P}$.} We have $\diam\left(\xi_1\right)\leq D_{\mc P}(r)$. In particular, 
        \[
        d(\xi_1, g^{j-1}m) \geq d(g^jm, g^{j-1}m) - \diam\left(\xi_1\right) > \abs{\gamma}/Q'' - D_{\mc P}(r) > r.
        \]
        If $d(\xi_1, g^{i+1}m)\leq r$, then the special path $\xi_1$ (which can be viewed as a $(c, \mc P)$--polygonal line) intersects the $r$ neighbourhood of $g^{i+1}m$ and $g^jm$ but not $g^{j-1}m$, contradicting the minimality of $j-i$. 
        \item \textbf{the connector $\xi_2 = [\pgot^+, g^im]_{\mc P}$.} We have shown above that $d(\xi_1, g^{j-1}m) > r$. An analogous argument shows that $d(\xi_2, g^{i+1}m)> r$. 
        \item \textbf{a subsegment of one of $g^i\gamma$, $g^{i+2}\gamma$, $g^i\eta$ or $g^{i+1}\eta$.} This holds by Lemma~\ref{lemma:good-connector} and the lower bound on $\abs{\gamma}$.
        \item \textbf{a subsegment of $g^{h-1}\eta$ or $g^h\gamma$ for some $ i+2\leq h \leq j$.} Define $\pgot' = g^{h-1}\eta\ast g^h\gamma$. If $d(\pgot', g^{i+1}m)\leq r$, then $\pgot'$ is a $(6, \mc P)$--polygonal line with $d(\pgot', g^{i+1}m)\leq r$ and $d(\pgot', g^hm)\leq r$. Further, by Lemma~\ref{lemma:good-connector} and the lower bound on $\abs{\gamma}$, we have $d(\pgot', g^{h-1}m) > r$. Hence $\pgot'$ intersects the $r$ neighbourhood of $g^{i+1}m$ and $g^hm$ but not $g^{h-1}m$, contradicting the minimality of $j-i$. Here we use that $c\geq 6$.
    \end{itemize}
    This concludes the proof of the first part of the statement for $c\geq 6$. For $c\leq 6$ the first part of the statement follows since for any $c\leq 6$ a $(c, \mc P)$--polygonal line is a $(6, \mc P)$--polygonal line. 

    We are left the show the in particular part of the statement. As used above, by the lower bound on $\abs{\gamma}$ and Lemma~\ref{lemma:good-connector}, we have that $d(g^i\gamma \cup g^i\eta, g^{i+1}m) > r$ and $d(g^i\gamma \cup g^{i-1}\eta, g^{i-1}m) > r$. If $d(g^i\gamma \cup g^i\eta, g^jm)> r$  (or similarly, $d(g^i\gamma \cup g^{i-1}\eta, g^{j}m)\leq r )$ for some $j$ with $6\abs{j-i} + 10\leq n$, then the main part of the statement yields a contradiction. 
\end{proof}

\subsection{Proof of Theorem~\ref{thm:strict-diameter-dichotomy}}

We are now ready to prove Theorem~\ref{thm:strict-diameter-dichotomy}, whose statement we recall here.

\strongdiameter*

\begin{proof}
    Let $\mc K' = (K', n', \mc P)$ be an allowed contraction triple. Let $\rho>0$ be a constant such that $G$ acts $\rho$--coboundedly on $X$. Let $\Delta = \ceil{\Delta(\rho+1, \delta_0)}$ be as in Lemma~\ref{lemma:actions-on-finite-hyperbolic-spaces}. In particular, by Corollary~\ref{cor:weak-diameter}, $\diam_{\mc K'}(\hat{X}_{\mc K'})$ is either infinite or bounded by $\Delta$. Let $Q''$ be the constant from Lemma~\ref{lemma:good-connector}. Further, let $\mc K = (K, n, \mc P)$ be an allowed contraction triple with $n$ and $K$ large enough, to be determined later.

    If all pairs of vertices $(x, y)\in V(X)\times V(X)$ are $\mc K$--anti-contracting, the statement holds. So we may assume that there exists a special path $\gamma\in \mc P$ which is $\mc K$--midthin and has neck-size $r$ for some $r\geq 0$. Denote the endpoints of $\gamma$ by $x$ and $y$ and let $m$ be the midpoint of $\gamma$. Let the path $\eta$ and the element $g\in G$ be as in Lemma~\ref{lemma:good-connector}. 
    
    \begin{claim}\label{claim:dhatdelta}
        For $n$ and $K$ large enough, we have
        \begin{align*}
                \hat{d}_{\mc K'} (x, g^{2\Delta +4 } x) > \Delta.
        \end{align*}
    \end{claim}

    As a consequence of the claim and Corollary~\ref{cor:weak-diameter}, we have that $\widehat{\diam}(\hat{X}_{\mc K'}) = \infty$. It thus remains to prove the claim.

    \textit{Proof of Claim \ref{claim:dhatdelta}.}
    Assume that $\hat{d}_{\mc K'} (x, g^{2\Delta +4 } x) \leq  \Delta$. Let $z_0 = x, z_{\Delta+2} = g^{2\Delta+4}x$. Further, for $1\leq i \leq \Delta+1$, let $z_i\in V(X)$ be a vertex such that $d(z_0, z_1)\leq 1$, $d(z_{\Delta+1}, z_{\Delta+2})\leq 1$ and $\hat{d}_{\mc K'}(z_{i}, z_{i+1})\leq 1$ for all $i$. Such a sequence exists since $\hat{d}_{\mc K'} (x, g^{2\Delta +4 } x) \leq  \Delta$ and $\hat{X}_{\mc K'}$ is a graph. 

    Now, for all $0\leq i \leq \Delta$, let $\alpha_i\in \mc P$ be a $\mc K'$--anti-contracting special path from $z_i$ to $z_{i+1}$. Such paths $\alpha_i$ exist, by Lemma~\ref{lemma:small-implies-anti-contracting}. Let $\alpha = \alpha_1\ast\ldots \ast \alpha_{\Delta+1}$.

    Observe that $\gamma^{-1}\ast \alpha\ast g^{2\Delta+4}\gamma$ is a $(\Delta+4, \mc P)$--polygonal line going through $m$ and $g^{2\Delta+4}m$. For $n$ and $K$ large enough, Proposition~\ref{prop:no-skips} implies that $d(\alpha, g^jm)\leq r$ for all $1\leq j < 2\Delta+4$. By the pigeonhole principle, there has to exist some $j$ such that $d(\alpha_j, g^im)\leq r$ for at least 3 different indices $i$ with $1\leq i <2\Delta +4$. Again by Proposition~\ref{prop:no-skips} we can assume that those indices are $i-1, i$ and $i+1$ for some $i$. Now we apply Lemma~\ref{lemma:close-to-thin-implies-thin-contraction-gauge-version} to $\gamma = \alpha_j$ and $\pgot_1  = g^{i-1}\gamma\ast g^{i-1}\eta$, $\spath = g^i\gamma$ and $\pgot_2 =g^i\eta \ast g^{i+1}\gamma$. Note that for $n$ and $K$ large enough, the conditions for Lemma~\ref{lemma:close-to-thin-implies-thin-contraction-gauge-version} are satisfied by Lemma~\ref{lemma:good-connector} and the definition of $\gamma$. Hence, there exists a $\mc K'$--midthin subpath $\alpha_j'$ of $\alpha_j$. This contradicts $\alpha_j$ being $\mc K'$--anti-contracting and hence concludes the proof of the claim.
    \hfill$\blacksquare$
\end{proof}

\section{Weak rank rigidity}\label{sect:the_dichotomy}
We are finally ready to prove the main result of the paper. We start by controlling the divergence of a space in terms of its generalised contraction space.

\begin{theorem}\label{thm:divergence-dichotomy1}
    Let $G$ be a non-virtually cyclic group which acts properly and coboundedly by graph isomorphisms on a graph $X$ equipped with a navigable, undirected and $G$--invariant path system $\mc P$. If there exists an allowed contraction triple $\mc K = (K, n, \mc P)$ with $\widehat{\diam}_{\mc K}(V(X)) = 1$, then $X$ has linear divergence.
\end{theorem}
\begin{proof}
   Let $\mc K = (K, n, \mc P)$ be an allowed contraction triple with $\widehat{\diam}_{\mc K}(V(X)) = 1$. Let $k, C'$ be such that $\mc P$ is $(C', k)$--navigable. Let $a, b, c\in X$ be points. It suffices to show that $\dv_{\epsilon} (a, b, c; \delta)\leq  C d(a, b) + C$ for some suitably chosen constants $C, \eps, \delta$ not depending on $a, b, c$. 
    Let $\spath : [s, t]\to X\in \mc P$ be a special path from $a$ to $b$. If $d(\spath, c) >  \delta \ell - \eps$, then $\dv_{\epsilon} (a, b, c; \delta)\leq \norm{\spath}\leq C d(a, b) + C$ for $C$ large enough, depending only on the quasi-geodesic constants $(\lao, \kao)$ of the path system $\mc P$.
    
    Thus, without loss of generality, we may assume that there exists $c'\in \spath$ with $d(c', c)\leq \delta\ell - \eps$. Note that any path that avoids the $2\delta \ell$--neighbourhood of $c'$ avoids the $\delta \ell$--neighbourhood of $c$ so it is enough to show that there exists a path $\pgot$ from $a$ to $b$ which avoids the $2\delta\ell$--neighbourhood of $c'$ and has $\norm{\pgot}\leq Cd(a, b) + C$.

    \begin{claim}\label{claim:detour}
        There exists a $(n+6, \mc P)$--polygonal line $\pgot'$ from $b$ to $a$ with $d(\pgot', c') > 2\delta \ell C' = R$. 
    \end{claim}

    \textit{Proof of Claim \ref{claim:detour}.}
    Let $\rho$ be such that $G$ acts $\rho$--coboundedly on $X$. Let $M$ be a constant to be determined later. For $\delta$ small enough compared to the other constants, there exists a special path $\eta_1\in \mc P$ such that $d(a, \eta_1^-)\leq \rho$, $\abs{\eta_1}\geq M$ and $d([a, \eta_1^-]_{\mc P}\ast \eta_1, c')\geq R$ (see Lemma~\ref{lemma:quasi-geodesic-concatenation}). We denote $[a, \eta_1^-]_{\mc P}$ by $\xi_1$. 

    Analogously, there exist special paths $\eta_2, \xi_2 =[b, \eta_2^-]_{\mc P}$ such that $d(b, \eta_2^-)\leq \rho$, $\abs{\eta_2}\geq M$ and $d(\eta_2\ast \xi, c')\geq R$. Let $M'\gg M$ be another constant to be determined later and let $a', b'\in V(X)$ be vertices on $\eta_1$ and $\eta_2$, respectively, such that $d(a, a')\geq M'$ and $d(b, b')\geq M'$. Such points exist if $M$ is large enough compared to $M'$. Let $\gamma\in \mc P$ be a special path from $a'$ to $b'$. Since $\widehat{\diam}_{\mc K}(V(X)) = 1$, we know that $(a', b')$ is $\mc K$--anti-contracting and we can choose $\gamma$ to be $\mc K$--anti-contracting. This is depicted in Figure~\ref{fig:divergence}. There are two cases. 

      \begin{figure}
        \centering
        \includegraphics[width=0.6\linewidth]{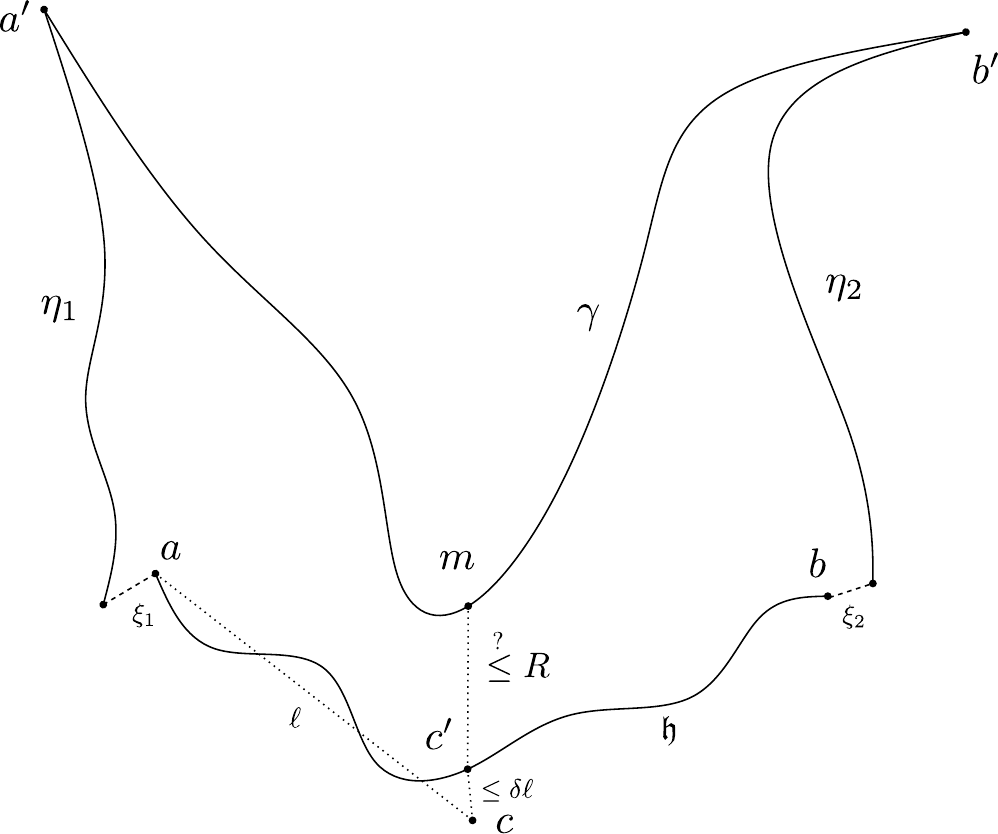}
        \caption{Proof of Claim~\ref{claim:detour}. We finish the proof differently depending on whether $d(c', m) > R$ or not.}
        \label{fig:divergence} 
    \end{figure}

    \textbf{Case 1:} $d(\gamma, c') > R$. In this case, 
    \begin{align*}
        \pgot' = \xi_2\ast \psub{\eta_2}{\eta_2^-, b'}\ast \gamma^{-1}\ast \psub{\eta_1^{-1}}{a',\eta_1^-}\ast \xi_1^{-1}
    \end{align*}
    is a $(5, \mc P)$--polygonal line with the desired properties.

    \textbf{Case 2:} $d(\gamma, c')\leq R$. Let $m\in \gamma$ be one of the points with $d(m, c')\leq R$ and let $\gamma'$ be a largest subsegment of $\gamma$ with $m_{\gamma'} = m$. In particular, one of the endpoints of $\gamma'$ is $a'$ or $b'$. If we choose $M'$ large enough, we can guarantee that $\abs{\gamma'}\geq K(2R)$.
    Since $\gamma$ is $\mc K$--anti-contracting, and $\abs{\gamma'}\geq K(2R)$, there exists a $(n, \mc P)$--polygonal line $\pgot''$ from $\gamma'^+$ to $\gamma'^-$ with $d(\pgot'', m) > 2R$ (and consequently $d(\pgot'', c')> R$). Define
    \begin{align*}
        \pgot ' = \xi_2\ast \psub{\eta_2}{\eta_2^-, b'}\ast \psub{\gamma^{-1}}{b', \gamma'^+}\ast\pgot''\ast \psub{\gamma^{-1}}{\gamma'^-, a'}\ast \psub{\eta_1^{-1}}{a',\eta_1^-}\ast \xi_1^{-1}.
    \end{align*}
    Using Remark~\ref{rem:lemmas-for-not-midthin} and the observations above, we obtain $d(\pgot', c)\geq R$. Hence $\pgot'$ is a $(n+6, \mc P)$--polygonal line with the desired properties.
    \hfill$\blacksquare$
   
    \smallskip

    Let $\pgot'$ be as in Claim~\ref{claim:detour}. Let $a'' = \spath(s''), b'' = \spath(t'')$ be the first and last points on $\spath$ which are in $B(c', 2R)$. Consider the $(n+8, \mc P)$--polygonal line $\pgot'' = \isub{\spath}{t'', t}\ast \pgot'\ast \isub{\spath}{s, s''}$. This path avoids $B(c', R)$ and by construction has endpoints $b''$ and $a''$ in $B(c', 2R)$. 
    For $\eps > 1$ we have $\delta\ell > 1$, and hence $R\geq C'$. Thus, by the property of navigability of $\mc P$, there exists a $(k(n+8), \mc P)$--polygonal line $\tilde{\pgot}$ from $b''$ to $a''$ which avoids $B(c', R/C') = B(c', 2\delta \ell)$ and has $\norm{\tilde{\pgot}}\leq (n+6)C'R$. 
    Hence, for $C$ large enough compared to $C', n, \delta$, the path $\pgot = \isub{\spath}{s, s''}\ast \tilde{\pgot}^{-1}\ast \isub{\spath}{t'',t}$ avoids $B(c', 2\delta\ell)$ and satisfies $\norm{\pgot}\leq Cd(a, b)+C$.
\end{proof}

By combining all the above results, we prove the main result of the paper. 

\begin{theorem}\label{thm:main-dichotomy}
    Let $X$ be a geodesic metric space and let $G$ be a non-virtually cyclic group acting properly, coboundedly and by isometries on $X$. If $X$ admits a navigable, undirected, $G$--invariant path system $\mc P$, then either $G$ has linear divergence, or $G$ is acylindrically hyperbolic and has a $\mc P$--contracting element. 
\end{theorem}

\begin{proof}
    We want to use Theorem~\ref{thm:divergence-dichotomy1}, so we need to replace $X$ with a graph $\Gamma$. In what follows, we construct this graph.
  
    Fix a basepoint $x_0\in X$ and let $\rho >0$ be a constant such that $X \subset G\cdot B(x_0, \rho)$. We define the graph $\Gamma$ as follows: its vertex set $V(\Gamma)$ is the orbit $G\cdot x_0$; the edge set $E(\Gamma)$ consists of all pairs of vertices $x, y$ with $d(x, y)\leq 4\rho$. The group $G$ has an obvious left action by graph isomorphisms on $\Gamma$, which inherits the properties of properness and coboundedness from the action of $G$ on $X$.

    A standard argument shows that $X$ being geodesic implies that $\Gamma$ is connected. Furthermore, there is a $G$--equivariant quasi-isometry $f$ from $X$ to $\Gamma$. Let $\mc Q$ be the directional closure of the push-forward $\mc P_f$. The path system $\mc Q$ on $\Gamma$ inherits navigability and $G$--invariance from $\mc P$, see Lemma~\ref{lem:navigable-push-forward} and Remark~\ref{rem:push-forward-properties}.

    Thus, $(\Gamma, \mc Q)$ satisfies the standing assumptions of Section~\ref{sec:generalised_contraction} and the assumptions of Theorems~\ref{thm:strict-diameter-dichotomy} and~\ref{thm:divergence-dichotomy1}.

    By Theorem~\ref{thm:strict-diameter-dichotomy}, there exists an allowed contraction triple $\mc K$ such that $\widehat{\diam}_{\mc K}(V(\Gamma))$ is either $1$ or $\infty$. If it is $1$ then Theorem~\ref{thm:divergence-dichotomy1} implies that the divergence of $\Gamma$ is linear. Since divergence is invariant under quasi-isometries, the divergence of $G$ and of $X$ are also linear. 

    If, on the other hand, $\widehat{\diam}_{\mc K}(V(\Gamma)) =\infty$, then $G$ is acylindrically hyperbolic and has a $\mc P$--contracting element, by Corollary~\ref{cor:contracting-element}. 
\end{proof}

\section{Examples and applications}\label{sec:examples}

\subsection{Metric spaces}\label{sec:metric-spaces}

We begin by listing families of metric spaces to which our theorems apply.

\medskip

{\textbf{Example M1. Hilbert and Thompson geometries.}}\quad  Consider a cone $C$ in a vector space $V$ and $\Omega \subset {\mathbb{P}}V$ its projectivisation. We say that $C$ and $\Omega $ are {\emph{convex}} if there exists a hyperplane $H$ disjoint from $C\setminus \{ 0\}$ such that $\Omega$ is identified with a convex subset of the affine space ${\mathbb{P}}V \setminus {\mathbb{P}}H$. We say that $C$ and $\Omega $ are {\emph{properly convex}}  if there exists a hyperplane $H$ disjoint from the closure of $C$ with the origin removed. In finite dimension, this is equivalent to the property that the cone does not contain an affine line.      

A cone can be endowed with either Thompson's part metric or with Hilbert's projective pseudo-metric. The latter induces a genuine metric on the projectivisation $\Omega$ (see \cite[Chapter 1]{NussbaumMemoirs} and\cite{NussbaumWalsh}). If $V$ is finite dimensional and $\Omega$ is identified with an ellipsoid in an affine space ${\mathbb{P}}V \setminus {\mathbb{P}}H$ then the Hilbert metric on it defines a model of the hyperbolic space. 

{\emph{Both the Thompson and the Hilbert metric on a properly convex cone admit a bounded bicombing that is consistent and invariant under linear, respectively projective, automorphisms. The paths composing the bicombing are both geodesics and projective straight lines \cite[Theorems 1 and 2]{NussbaumWalsh}}}. 

More precisely, the bicombing $\phi$ described in \cite[Formula (1.2)]{NussbaumWalsh} has the required properties for the Thompson metric, while its projectivisation has these properties for the Hilbert metric. Indeed, the intersection of the cone with the plane defined by two points $x$ and $y$ is a 2-dimensional cone, a positive quadrant for some basis of ${\mathbb R}^2$; by considering the logarithm of each of the two coordinates one obtains a map from the quadrant to ${\mathbb R}^2$, mapping $\phi(s; x, y)$ on the line segment between the images of $x$ and $y$. This shows that $\phi$ is consistent, therefore its projectivisation is also consistent.   

The bicombing $\phi$ is invariant under linear automorphisms of the cone, since the parameters $M(x/y)$ and $M(y/x)$ used to define it are. Therefore, the projectivisation of $\phi$ is likewise invariant under projective automorphisms. 

In general, this bicombing does not come from a CAT(0) geometry. The space $\Omega$ endowed with a Hilbert metric is CAT(0) if and only if the cone $C$ is a Lorentz cone \cite{NussbaumWalsh}, equivalently, if $C$ is linearly isomorphic to a cone over an ellipsoid,
$C= \{ (t, x_1,\dots, x_n) \mid t^2 > x_1^2 + x_2^2+\cdots +x_n^2 \}.$ Indeed, by \cite{FoertschLytchakSchroeder}, a CAT(0) space is always Ptolemaic and, according to Kay \cite{Kay:ptolemaic}, a Hilbert geometry is (locally) Ptolemaic if and only if the cone is over an ellipsoid.

The Hilbert geometry is not in general Gromov hyperbolic either. If the cone $C$ is properly convex, then  $\Omega$ endowed with a Hilbert metric is Gromov hyperbolic if and only if it is quasi-symmetrically convex \cite{Benoist:conesIHES03}.  

\medskip

{\textbf{Example M2. Injective metric spaces.}}\quad Injective metric spaces are the injective objects in the category of metric spaces with morphisms the 1-Lipschitz maps. They are always complete, contractible and geodesic.  

For equivalent definitions and further properties of injective metric spaces we refer the reader to \cite{AronszajnP}, where they are called \textit{hyperconvex spaces}, to \cite{Isbell:injective}, \cite{Lang:injective}, \cite{Basso:combing-improvement} and to references therein. Examples of injective metric spaces are:

\begin{itemize}
    \item real trees;

    \item complete connected finite rank median spaces admit bi-Lipschitz deformations of their metric that remain ${\mathrm{Isom}} (X)$--invariant \cite[Theorem 7.8]{Bowditch:injective} and are injective. Asymptotic cones of mapping class groups and of Teichm\"uller spaces endowed with either the Teichm\"uller or the Weil-Petersson metric are examples of such median spaces.     
\end{itemize}

{\emph{ Every injective metric space $X$ admits an $\isom(X)$-equivariant reversible conical geodesic bicombing \cite[Proposition 3.8]{Lang:injective}}. 

\medskip

{\textbf{Example M3.}}\quad Another source of examples comes from {\emph{metric spaces that admit conical geodesic bicombings}}. 

One of the features that brought this class of examples at the forefront of current research is that spaces with conical geodesic bicombings have strong fixed-point properties \cite{Karlsson:fixed-point}. For instance, every isometry of such a space fixes a point in its metric compactification. The latter is the closure of the metric space $X$ seen as a subspace of the space ${\mathrm{Hom}}(X)$ of 1-Lipschitz functions $X\to {\mathbb{R}}$, endowed with the topology of pointwise convergence. The embedding of $X$ into ${\mathrm{Hom}}(X)$ is \textit{via} a map $x\mapsto \dist (\cdot , x) - \dist (x_0, x)$, where $x_0\in X$ is an arbitrary base point.  

\medskip

The class of spaces with conical geodesic bicombings includes: 
\begin{itemize}
    \item (convex subsets of) Banach spaces;
    \item CAT(0) spaces;
    \item ${\mathrm{Pos}}$, the space of positive bounded linear operators of a Hilbert space, on which every invertible bounded linear operator acts by isometry;
    \item spaces of probability measures endowed with the $1$-Kantorovich-Wasserstein distance;   
    \item the Teichmüller space endowed with the Weil-Petersson metric. 
\end{itemize}

{\emph{In a certain sense, subsets of injective metric spaces are generic examples of spaces with conical geodesic bicombings}}. Indeed, every space with a conical geodesic bicombing embeds isometrically into an injective metric space as a $\sigma$--convex subset (i.e. a subset for which there exists a conical bicombing in the ambient space such that all the combing lines with points in the subset are contained in the subset).

\medskip
{\textbf{Example M4}.}\quad {\emph{Spaces with a unique convex geodesic bicombing.}} These are studied in \cite{Haettel:CUB} under the name of Convexly Uniquely Bicombable (or CUB for shortness) metric spaces. These include 
\begin{itemize}
\item simplicial complexes with a polyhedral metric and a link condition essentially requiring that the complex be a lattice locally; examples are Euclidean buildings, some arc complexes of punctured spheres, the complex of homologous multicurves on a surface, the Kakimizu complex of minimal Seifert surfaces of a link etc; we refer to the section on groups for further examples; 
\item Busemann convex spaces; examples include normed vector spaces and certain Finsler metric spaces \cite{IvanovLytchak}. 
\end{itemize}

\medskip
{\textbf{Example M5}.}\quad {\emph{Systolic and weakly systolic simplicial complexes admit a consistent bounded quasi-geodesic bicombing.}}

In the case of systolic complexes, this was proved in \cite[Section 13]{JS2006}. This was extended to weakly systolic complexes in \cite{ChalopinChepoiG}, where the construction of normal clique paths from  \cite{JS2006}, which define the bicombing, was extended. As in \cite{JS2006}, since these paths have two equivalent definitions, one global and one local, they satisfy the (quasi)-consistency condition.

Examples of systolic simplicial complexes include the following.

\begin{itemize}
    \item A connected simply connected $2$-dimensional
simplicial complex is systolic if and only if it is a C(3)--T(6)
simplicial complex in the sense of \cite{Lyndon-Schupp(1977)}. For example, buildings of type $\widetilde{A}_2$ are systolic. In particular, the Cayley complex of a group with a triangular presentation is systolic if and only if the group (presentation) satisfies the C(3)--T(6) small cancellation condition. 
    \item A graph is the $1$-skeleton of a systolic complex if and only if it is bridged \cite{Chepoi:graphs}.
\end{itemize}

\medskip

{\textbf{Example M6}.}\quad Other examples of spaces endowed with a bounded, quasi-consistent quasi-geodesic bicombing are: 
\medskip
\begin{itemize}

\item Helly graphs of bounded valency (as they are $1$-skeleta of complexes which come endowed with a structure of injective metric spaces as above);
\item in the case of hierarchically hyperbolic spaces, the bicombing due to Petyt and Zalloum is bounded, quasi-consistent and quasi-geodesic \cite[Theorem~F and Proposition~4.7]{petytzalloum:constructing}; 
\item piecewise $L^p$--spaces, where $1<p<\infty$, that are locally uniquely geodesic are globally uniquely geodesic and the geodesic bicombing is bounded \cite{HaettelHodaPetytGT,HaettelHodaPetytLMS};
\end{itemize}

{\textbf{Example M7}.} \quad \textit{Spaces quasi-isometric to any of the above.} This is because having a navigable quasi-geodesic path system is a quasi-isometric invariant (see Lemma~\ref{lem:navigable-push-forward}). Notable examples in this category are coarsely injective (or coarsely hyperconvex) metric spaces. A definition of these spaces can be found in \cite[Section 3.3]{Chalopin:Helly}. According to \cite[Proposition 3.12]{Chalopin:Helly}, any coarsely hyperconvex metric spaces is quasi-isometric to a (possibly non proper) injective metric space. 

\subsection{Groups}\label{sec:groups}

In this section we enumerate groups acting properly discontinuously and coboundedly on spaces endowed with bounded equivariant bicombings that are either quasi-consistent or geodesic. The order in which the examples are listed follows, as much as possible, the one in Section \ref{sec:metric-spaces}  Metric spaces.

\medskip

\textbf{Example G1. Convex cocompact groups.} Consider a cone $C\subset V$ and its projectivisation $\Omega$, as in Example M1. The cone $C$ is called {\emph{divisible}} if there exists a discrete subgroup $\Lambda$ in ${\mathrm{Aut}} (C) \leq GL(V)$ such that $C/\Lambda $ is compact. Likewise, $\Omega$ is called {\emph{divisible}} if there exists a discrete subgroup $\Gamma$ in ${\mathrm{Aut}} (\Omega ) \leq PGL(V)$ such that $\Omega /\Gamma $ is compact. 

The divisible properly convex cones that are homogeneous (or symmetric) have been classified by Koecher \cite{Koecher}, Vinberg \cite{Vinberg:hom-cones1, Vinberg:hom-cones2} and Rothaus \cite{Rothaus}. There are also many such cones that are not homogeneous, with the first such examples due to Kac and Vinberg \cite{KacVinberg}, see also \cite{BenoistDim3,Ballas,ChoiLeeMarquis1,ChoiLeeMarquis2,MarquisGD10}.

A systematic study of divisible convex cones was started by Benz\'ecri \cite{Benzecri} and continued by Yves Benoist (see \cite{Benoist:conesIHES03, BenoistDuke, BenoistTata, Quint:Bourbaki} and references therein). An interesting feature is that if $C$ is a properly convex cone that is irreducible (that is, it cannot be decomposed into a sum corresponding to a direct sum decomposition of $V$), the Zariski closure of a group $\Lambda$ as above is either commensurable with ${\mathrm{Aut}} (C)$ (when $C$ is homogeneous) or it is the entire $GL(V)$ \cite{BenoistDuke}. 

In \cite{BenoistTata} it is proven that if $\Omega $ is properly convex then $\Gamma$ is Gromov hyperbolic if and only if $\partial \Omega $  is of class $C^1$, equivalently if $\partial \Omega$ contains no segment of a projective line other than a point. 

This family of examples covers the following subcase. A compact manifold with a flat projective structure (i.e. a $( PGL(V), {\mathbb{P}}V )$--structure) has the property that every curve on it is homotopic (with fixed endpoints) to a geodesic segment if and only if it can be identified with the quotient of an open convex subset $\Omega$ of ${\mathbb{P}}V$ by a discrete subgroup of automorphisms of $\Omega$. Thus, fundamental groups of such manifolds are one set of examples.

A larger class of groups that still satisfy the assumptions of the main dichotomy theorem is that of {\emph{naively convex cocompact groups}}, in the terminology of \cite{DancigerGueritaudKassel}. These are groups $\Gamma$ that act properly discontinuously on a properly convex open subset $\Omega$ of $\mathbb{P}V$, and cocompactly on a nonempty $\Gamma$-invariant closed convex subset of $\Omega$.

Thus, a consequence of Theorem \ref{thm:main-dichotomy} are the following.

\begin{corollary}
 A naively convex cocompact group is either  acylindrically hyperbolic or with linear divergence.   
\end{corollary}

This, in particular, provides examples of groups that cannot be naively convex cocompact.

\begin{corollary}
Lacunary hyperbolic groups \cite{OOS} cannot be  naively convex cocompact.   
\end{corollary}

\medskip

\textbf{Example G2.} Among the groups acting  properly discontinuously and cocompactly  on injective spaces (which are therefore proper) are the  {\emph{Helly groups}}, i.e. the groups acting p.d.c. on Helly graphs \cite[Theorem 1.5 (5)]{Chalopin:Helly}. According to \cite{Chalopin:Weaklymodular}, these include groups acting geometrically on swm-graphs.  

The class of Helly groups includes: 
\begin{itemize}
\item hyperbolic groups; 

\item groups hyperbolic relative to Helly subgroups;
\item cubulable groups; 
\item finitely presented graphical C(4)-T(4) small cancellation groups;
\item type-preserving uniform lattices in Euclidean buildings of type  $C_n$;
\item spherical Artin groups; 
\item FC Artin groups as well as certain affine Artin groups;  
\item Garside groups.
\end{itemize} 

\medskip

\textbf{Example G3. Coxeter groups.} Some Coxeter groups are already covered by Example G1. In \cite[Theorems 1.3 and 1.8]{DancigerGueritaudKassel} Coxeter groups that have a convex cocompact representation are characterized. While not all Coxeter groups have such a representation, they all satisfy the assumptions of Theorem \ref{thm:main-dichotomy}, due to another type of geometry they can be endowed with. Namely, in \cite{OsajdaPrzytycki:Coxeter} the authors show that all Coxeter groups are biautomatic by constructing the \emph{voracious language} $\mathcal{V}$ and showing it satisfies the requirements for biautomaticity. Since words in $\mc V$ are geodesics, this yields a geodesic combing and the assumptions for our theorem are satisfied.

\medskip

\textbf{Example G4.} Groups acting properly discontinuously and coboundedly on CUB spaces. Among these are:
\begin{itemize}
    \item general simplices of groups;
    \item Euclidean Artin groups;
    \item (weak) Garside groups \cite{HuangOsajda:HellymeetsGarside}; the latter include Deligne's groups and fundamental groups of Salvetti complexes of oriented matroids \cite{Haettel:CUB}.
\end{itemize}

\medskip

\textbf{Example G5.} Groups acting properly discontinuously coboundedly on (weakly) systolic complexes. These are sometimes called {\emph{(weakly) systolic groups}}. Examples include:

\begin{itemize}
    \item Artin groups of almost large type \cite{HuangOsajda:largeArtin}; An Artin group $A_\Gamma $ is {\emph{of almost large type}} if and only if no triangle in $\Gamma $ has an edge with label $2$ and no square in  $\Gamma $ has three edges labelled by $2$. This class obviously contains the class of Artin groups of large type (i.e. the class where the label $2$ is not allowed).

    \medskip
    
    \item finitely presented subgroups of systolic groups are systolic \cite{Zadnik};
\end{itemize}  

Examples of groups that are weakly systolic and not systolic are provided in \cite{Osajda:preprint}.

\medskip 
\textbf{Example G6.} All hierarchically hyperbolic groups \cite{HaettelHodaPetytGT}. This includes:

\medskip
\begin{itemize}
    \item Mapping Class Groups,

    \item All three manifold groups that do not admit Nil or Sol geometries, including non-geometric ones \cite{HRSS:equivariant},

    \item Artin groups of extra‑large, large, and hyperbolic type \cite{HagenMartinSisto:ExtraLarge},

    \item Extensions of lattice Veech groups \cite{DowdallDurhamLeiningerSisto:extensionsII}.
\end{itemize}

\medskip

\textbf{Example G7.}  Groups with a presentation that satisfies a combinatorial condition described in \cite{HaettelHuang}. Examples are Artin groups of cyclic type and certain T(5) groups.

\medskip

\textbf{Example G8.} Biautomatic groups whose induced bicombing is either geodesic or quasi-consistent. Besides the examples presented elsewhere, these include all C(6) or C(3)-T(6) small cancellation groups \cite{GerstenShort:automatic}.

\medskip

\textbf{Example G9.} Fundamental groups of manifolds without focal points.
A generalization of the condition of non-positive curvature for Riemannian manifolds is that there exist no conjugate points. 
Given a closed Riemannian manifold $M$ without conjugate points, the universal cover $\widetilde M$ has a natural (consistent) geodesic bicombing that is $G$-invariant, for $G=\pi_1 (M)$. This can be generalized to compact locally simply connected length spaces $M$ with no conjugate points \cite{Dibble}.  

This bicombing is known to be bounded if $M$ has no focal points \cite[p. 239]{IvanovKapovitch}. In the more general case of a Riemannian manifold without conjugate points this is still an open question \cite[Question 8.3]{IvanovKapovitch}.

\medskip

\textbf{Example G10.} Outer automorphism groups of one-ended hyperbolic groups with JSJ decompositions such that all quadratically hanging subgroups are orientable. These groups are virtually hierarchically hyperbolic (Example \textbf{G6}) \cite{HadziosmanovicMangioni:outer}. In particular, they do not satisfy Theorem~\ref{thmi:dichotomy}, but do satisfy Corollary~\ref{cori:virtual_wrr}.

\bibliographystyle{alpha} 
\bibliography{cornelia} 

\end{document}